\documentclass[12pt,a4paper]{article}%
\usepackage{amsmath}
\usepackage{graphicx}
\usepackage{epsfig}%
\usepackage{amsfonts}%
\usepackage{amssymb}

\usepackage{color}

\usepackage[normalem]{ulem}

\newtheorem{theorem}{Theorem}[section]

\newtheorem{definition}[theorem]{Definition}
\newtheorem{example}[theorem]{Example}
\newtheorem{lemma}[theorem]{Lemma}
\newtheorem{proposition}[theorem]{Proposition}
\newtheorem{remark}[theorem]{Remark}

\newenvironment{proof}[1][Proof]{\noindent \emph{#1.} }{\hfill \ 
\rule{0.5em}{0.5em}}

\makeatletter\@addtoreset{equation}{section}\makeatother
\makeatletter\@addtoreset{figure}{section}\makeatother
\makeatletter\@addtoreset{table}{section}\makeatother
\begin{document}

\title{Range-separated  tensor format for numerical modeling of many-particle 
interaction potentials}


\author{Peter Benner \thanks{Max Planck Institute for Dynamics of Complex 
Systems, Sandtorstr.~1, D-39106 Magdeburg, Germany
({\tt benner@mpi-magdeburg.mpg.de})}
\and
  Venera Khoromskaia \thanks{Max Planck Institute for
        Mathematics in the Sciences, Leipzig;
        Max Planck Institute for Dynamics of Complex Systems, Magdeburg ({\tt vekh@mis.mpg.de}).}
        \and Boris N. Khoromskij \thanks{Max Planck Institute for
        Mathematics in the Sciences, Inselstr.~22-26, D-04103 Leipzig,
        Germany ({\tt bokh@mis.mpg.de}).}\\  
            }

 \date{}

\maketitle

\begin{abstract}
We introduce and analyze the  new range-separated (RS) canonical/Tucker tensor  format 
   which aims for   numerical modeling of the 3D long-range interaction
potentials in multi-particle systems.
The main idea of the RS tensor format is the independent  grid-based 
low-rank representation of 
the localized and global parts in the target tensor which allows the efficient  numerical
approximation of $N$-particle interaction potentials.
The single-particle reference potential like $1/\|x\|$ is split into a sum of localized
and long-range low-rank canonical tensors represented on a fine
3D $n\times n\times n$ Cartesian grid.
The smoothed long-range contribution to the total potential sum  
is represented on the 3D grid in $O(n)$ storage via the low-rank canonical/Tucker tensor. 
We prove that the Tucker rank parameters depend only logarithmically 
on the number of particles $N$ and the grid-size $n$.
Agglomeration of the short range part in the sum is reduced to an independent 
treatment of $N$ localized terms with almost disjoint effective supports, 
calculated in $O(N)$ operations.
Thus, the cumulated sum of short range clusters is parametrized
by a single low-rank canonical reference tensor with a local support,
accomplished by a list of particle coordinates and their charges.
The RS canonical/Tucker tensor representations reduce the cost of multi-linear algebraic operations
on the 3D potential sums arising in  modeling of multi-dimensional data 
by radial basis functions, 
say, in computation of the electrostatic potential of a protein,
in 3D integration and convolution transforms, computation of gradients, 
forces and the interaction energy of a many-particle systems, and in low parametric 
fitting of multi-dimensional scattered data by reducing all of them to 1D calculations. 
\end{abstract}



\section{Introduction}\label{sec:Intro}

 
 Numerical treatment of long-range  potentials is a challenging task
 in computer modeling of dynamics and structure of multiparticle systems, 
 for example,  in molecular dynamics simulations of large solvated biological systems 
 like proteins, in analysis of periodic Coulombic systems or scattered data in geosciences,
 Monte Carlo sampling etc. 
\cite{PolGlos:96,DesHolm:98,HuMcCam:1999,Stein:2010,LiStCaMaMe:13,FoFl:15}.
For a given non-local generating kernel $p(\|x\|)$, $x \in \mathbb{R}^3$, 
the calculation of a weighted sum 
of  interaction potentials in the  large $N$-particle system,
with the particle locations at $x_\nu \in \mathbb{R}^3$, $\nu=1,...,N$, 
\begin{equation}\label{eqn:PotSum1}
 P(x)= {\sum}_{\nu=1}^{N} {Z_\nu}\,p({\|x-x_\nu\|}), \quad Z_\nu \in \mathbb{R},
  \quad x_\nu, x \in \Omega=[-b,b]^3,
\end{equation}
leads to computationally intensive numerical task. 
Indeed, the generating radial basis function $p(\|x\|)$ is allowed to have a slow polynomial 
decay in $1/\|x\|$  as $\|x\| \to \infty$  so that  each individual term 
in (\ref{eqn:PotSum1}) contributes essentially to the total potential 
at each point in the computational domain $\Omega$, thus predicting the 
$O(N)$ complexity for the straightforward 
summation at every fixed target $x \in \mathbb{R}^3$.  Moreover, in general, 
the function $p(\|x\|)$ has a singularity or a cusp at the origin, $x=0$, making 
its full grid representation problematic. 
Typical examples of the radial basis functions $p(\|x\|)$ are given by the Newton $1/\|x\|$, 
Slater $e^{-\lambda \|x\|}$, Yukawa/Helmholtz $e^{-\lambda \|x\|}/\|x\|$ and other Green's kernels 
(see examples in \S\ref{ssec:DataFit_Applic}). 

The traditional approaches based on the Ewald summation method \cite{Ewald:27}  combined with  
the fast Fourier transform (FFT)
usually apply to calculation of the interaction energy or the interparticle forces
of a system of $N$-particles with the periodic closure, 
which reduces the complexity scaling in a particle number from $O(N^2)$ 
to $O(N\log N)$ \cite{DesHolm:98,DesHolmII:98}.
These approaches need meshing up the result of Ewald sums over 3D 
Cartesian grid  for the charge assignment onto a $n_m \times n_m \times n_m$ mesh. 
Generation of the smoothed charge distribution on the right-hand side
of the arising Poisson's equation is the main complexity limitation 
since it requires the $N$-term summation of the grid functions of size $O(n_m^3)$,
presuming the dominating cost $O(n_m^3 \,N)$. This procedure is accomplished by
the cheap FFT solver with periodic boundary conditions that amounts 
to $O(n_m^3\log n_m)$ operations.

The mesh implementation approaches trace back to the original 
so-called particle-particle-particle-mesh (P$^3$M) methods \cite{HoEa:88}.

The fast multipole  expansion \cite{RochGreen:87} method  allows 
to compute some characteristics of the multi-particle potential, say, 
the interaction energy, at the expense $O(N\log N)$
by evaluation of the potential only at $N$  sampling  points $x_\nu$. 

 Computation of long-range interaction potentials of large many-particle systems is discussed
 for example in \cite{DYP:93,KuScuser:04,PolGlos:96}, and using grid-based approaches in 
 \cite{HoEa:88,Beck:Rev-2000,DesHolm:98,DesHolmII:98,ZuAmir:15,Fornberg-book-2015}.
 Ewald-type splitting of the Coulomb interaction into long- and short-range components 
 was applied in density functional theory calculations \cite{ToCoSa:04}.

In this paper, we introduce and analyze the  new range-separated (RS) canonical/Tucker 
tensor  format which aims for the efficient  numerical treatment of 
3D long-range interaction potentials in a system of rather generally distributed particles.
The main idea  of the RS format is the independent grid-based low-rank tensor 
representation to the long- and short-range parts in the  total sum of 
single-particle (say, electrostatic) potentials  in (\ref{eqn:PotSum1})  discretized on a fine
3D $n\times n\times n$ Cartesian grid $\Omega_n$ in the computational 
box  $\Omega \in \mathbb{R}^3$. 
 Such a representation is based on the
splitting of a single reference potential like $p(\|x\|)=1/\|x\|$ into a sum of localized
and long-range  low-rank  canonical tensors both represented on the computational 
grid $\Omega_n$.

 The main advantage of the RS format  is efficient representation
 of the long-range  contributions to the total potential sum in (\ref{eqn:PotSum1})
 by using the multigrid accelerated canonical-to-Tucker transform \cite{khor-ml-2009},
 which returns this part in a form of a low-rank canonical/Tucker tensor
 at the asymptotical cost $O(N\, n) $.
  In  Theorem \ref{thm:Rank_LongRange}, we prove that  the corresponding tensor rank
 only weakly (logarithmically) depends on the number of particles $N$.
 Hence, the  long-range contribution to the target sum is represented via the low-rank global
 canonical/Tucker tensor defined  on the fine $n\times n \times n$ grid $\Omega_n$,
 in the $O(n)$ storage.
 These features are demonstrated by numerical tests for
 the large 3D clusters of generally distributed particles.

 In turn, the short-range contribution to the total sum is constructed by using a \emph{single} 
 reference low-rank tensor of local support  selected  
from  the "short-range" canonical vectors in the tensor decomposition of the radial basis function $p(\|x\|)$. 
To that end the whole set of $N$ short-range clusters  is represented
 by replication and rescaling of the small-size localized canonical tensor
 defined on an $n_s\times n_s \times n_s$ Cartesian grid with $n_s\ll n$, 
thus reducing the storage to the $O(1)$-parametrization of the reference canonical tensor 
and the list of coordinates and charges of particles. 
Summation of the short-range part over $n\times n \times n$ grid
needs  $O(N \, n_s)$  computational work for $N$-particle system.
Such  cumulated sum of the short-range components allows  
 "local operations" in the  {RS-canonical format}, 
 making it particularly efficient  for tensor multilinear algebra.

The particular benefit of the RS approach is the low-parametric representation of
the collective interaction potential on a large 
3D Cartesian grid in the whole computational domain  $\Omega$  
at the linear cost $O(n)$, thus outperforming the traditional grid-based summation 
techniques based on the full-grid $O(n^3)$-representation in the volume.
Both global and local summation schemes are quite easy in program implementation.
The prototype algorithms in MATLAB\textsuperscript{\textregistered}  
applied  on a laptop allow to compute the RS-tensor representation of electrostatic potential for 
large  many-particle systems  on fine grids of size up to $n^3=10^{12}$.

The efficient numerical realization of RS formats  can be  achieved by a
trade off between the rank parameters in the long-range part and 
the effective support of the local sub-tensors. 
Indeed, the range separation step can be 
realized \emph{adaptively}  by a simple tuning of splitting rank parameters in the reference tensor
based on an $\varepsilon$-tolerance threshold in estimating the effective local support. 
 The low-rank RS canonical/Tucker tensor representation simplifies further operations 
on the resultant interaction potential, for example, 3D integration,
computation of gradients and forces, or evaluation of the interaction energy of a system 
by reducing all of them to 1D calculations. 

 As one of  many possible  applications of the RS tensor format, we propose a new numerical 
 scheme for  calculation of the free interaction energy of proteins, 
 and the enhanced regularized formulation for  solving the Poisson-Boltzmann equation (PBE) 
 that models the electrostatic  potential of proteins in a solvent. 
 We also demonstrate that the RS tensor 
formats can be useful in numerical modeling of the multi-dimensional scattered data
by means of the efficient data sparse approximation to the "inter-distance" matrix via
the short term sum of Kronecker product matrices with the "univariate" factors.
 
 The RS tensor format was motivated by the recent method for efficient
 summation of the long-range electrostatic potentials on large lattices with defects by using 
 the  assembled  canonical and Tucker tensors \cite{VeBoKh:Ewald:14,VeKhor_NLLA:15},
 which provides a competitive alternative to the Ewald summation schemes \cite{Ewald:27}.
In case of 3D finite lattice  systems, the grid-based tensor summation technique  yields
 asymptotic complexity $O(N^{1/3})$ in the number of particles $N$, and 
 almost linear complexity in the univariate grid-size $n$. 
 
 The RS-tensor approach can be interpreted as the model reduction based on the low-rank tensor approximations 
 (i.e., via a small number of representation parameters).
 The model reduction techniques for PDEs and control problems were described 
 in detail in \cite{BeMeSo-book-2004,QuMaNe:15,BeGuWi:15}.
 
 In the recent years, the tensor numerical methods  have been recognized as a powerful tool
 in scientific computing for multidimensional problems, 
 see for example \cite{KhorSurv:10,GraKresTo:13,VeKhorTromsoe:15,BeKhKh_BSE:15,BDKK_BSE:16} and 
    \cite{DeStMi:13,DeSt:12,RaOs:15,KrStUs:14,HaSch:14,CaEhLe:14,BeDoOnSt:16,DoPeSaSt:16}.
 In particular, the approximating properties of tensor decompositions in modeling 
 of high-dimensional problems have been addressed in 
 \cite{Stenger,Braess:BookApTh:86,HaKhtens:04I,DaDeGrSu:15}.
  Here we notice that in the case of higher dimensions
 the local canonical tensors can be combined with the global low-rank tensor train (TT) 
 representation \cite{Osel_TT:11}
 thus introducing the RS-TT format, see Remark \ref{rem:RS-TT}.

  The rest of the paper is organized as follows. In Section \ref{sec:S_L_split}, we 
  introduce the canonical and Tucker tensor formats,
  and provide a  description of the short-long range splitting to 
  the canonical tensor approximation of the Newton kernel 
  by using $\mbox{sinc}$-quadratures applied to the Laplace transform. 
   Section \ref{sec:S_L_split} also discusses the principles for 
  selection of the short and long-range parts in the  reference  electrostatic potential.
  Grid-based tensor splitting of the electrostatic potential sums is addressed in Section
  \ref{sec:Fast_Sum_Split},  where the efficient computation of  
  the long-range part of the potential is described.  
 In particular, Section \ref{ssec:Cumulat_CanTens} introduces and analyses the classes
 of range-separated  tensor formats.
  The possible application to protein modeling is  addressed  in Section \ref{sec:SepSum_Applic}.
  Furthermore, we  discuss  how the RS tensor formats may be utilized in the
 numerical  treatment  of multi-dimensional scattered data, 
 for calculation of gradients, forces and interaction energy of the system.
 Appendix recalls the main ingredients of the reduced HOSVD tensor approximation and 
 the canonical-to-Tucker tensor transform  applied in the numerical implementations. 
 

\section{Range separated tensor form of a reference potential}\label{sec:S_L_split}
 
\subsection{Representation of multivariate functions via low-rank tensors}\label{ssec:Tensors}

In this section we recall the commonly used rank-structured tensor 
formats\footnote{ The commonly used notion \emph{rank-structured tensor formats} 
 for the compressed representation 
 of multidimensional data is usually understood in sense of the (nonlinear) parametrization 
by a small number of parameters that allows low storage costs, a simple
representation of each entry in the target data array, and the efficient "formatted" multilinear
algebra via reduction to univariate operations. } 
utilized in this paper 
(see also the literature surveys \cite{KoldaB:07,KhorSurv:10,GraKresTo:13}). 
The traditional canonical and Tucker tensor representations were long since
known in computer science for the quantitative analysis of correlations in the
multidimensional data arising in image processing, chemometrics, psychometrics etc.,
see \cite{DMV-SIAM2:00,KoldaB:07} and references therein.


These formats have attracted the attention of the scientific computation community when it was recently 
shown numerically and rigorously proved that in most cases function related tensors allow   
low-rank tensor decomposition \cite{khor-rstruct-2006,KhKh:06}. 
In particular, they proved to be efficient  
for real-space calculations in computational quantum chemistry 
\cite{VeBoKh:Ewald:14,VeKhorTromsoe:15}.

A tensor of order $d$ is defined as a multidimensional array
over a $d$-tuple index set,
$$
{\bf A}=[a_{i_1,\ldots,i_d}]\equiv [a(i_1,\ldots,i_d)] \; \in 
\mathbb{R}^{n_1 \times \ldots \times n_d}\;
\mbox{  with } \quad i_\ell\in I_\ell:=\{1,\ldots,n_\ell\},
$$
considered as an element of a linear vector space equipped with the Euclidean scalar product.
Tensors 
with all dimensions having equal size $n_\ell =n$, $\ell=1,\ldots d$,
will be called an $n^{\otimes d}$ tensor. The  required storage size 
scales exponentially in the dimension, $n^d$, which results in the 
so-called ''curse of dimensionality``. 

To get rid of exponential scaling in the dimension, one can apply  the 
rank-structured separable representations (approximations) of multidimensional tensors.
The simplest separable element is given by the rank-$1$ tensor, 
\[ {\bf U} = {\bf u}^{(1)}\otimes \ldots 
 \otimes {\bf u}^{(d)}\in \mathbb{R}^{n_1 \times \ldots \times n_d},
\]
with entries
$
 u_{i_1,\ldots, i_d}=  u^{(1)}_{i_1} \cdots u^{(d)}_{i_d},
$
requiring only $n_1+ \ldots +n_d$ numbers to store it. 
A tensor in the $R$-term canonical format is defined by a finite sum of rank-$1$ tensors,
 \begin{equation}\label{eqn:CP_form}
   {\bf U} = {\sum}_{k =1}^{R} \xi_k
   {\bf u}_k^{(1)}  \otimes \ldots \otimes {\bf u}_k^{(d)},  \quad  \xi_k \in \mathbb{R},
\end{equation}
where ${\bf u}_k^{(\ell)}\in \mathbb{R}^{n_\ell}$ are normalized vectors, 
and $R$ is called the canonical rank of a tensor. Now the storage cost is bounded by  $d R n$.
For $d\geq 3$, there are no algorithms for computation of the canonical rank of a tensor ${\bf U}$, i.e.
the minimal number $R$ in representation (\ref{eqn:CP_form}) and the respective
decomposition with the polynomial cost in $d$.


We say that a tensor ${\bf V}$ is represented in the rank-$\bf r$ orthogonal Tucker format 
with the rank parameter ${\bf r}=(r_1,\ldots,r_d)$, if 
\begin{equation}\label{eqn:Tucker_form}
  {\bf V} =\sum\limits_{\nu_1 =1}^{r_1}\ldots
\sum\limits^{r_d}_{{\nu_d}=1} \beta_{\nu_1, \ldots ,\nu_d}
\,  {\bf v}^{(1)}_{\nu_1} \otimes \ldots \otimes {\bf v}^{(d)}_{\nu_d},\quad \ell=1,\ldots,d, 
\end{equation}
where $\{{\bf v}^{(\ell)}_{\nu_\ell}\}_{\nu_\ell=1}^{r_\ell}\in \mathbb{R}^{n_\ell}$,  
represents a set of orthonormal vectors for $\ell=1,\ldots,d$, 
and $\boldsymbol{\beta}=[\beta_{\nu_1,\ldots,\nu_d}] \in \mathbb{R}^{r_1\times \cdots \times r_d}$ is 
the Tucker core tensor. 
The storage cost for the Tucker tensor is bounded by $d r n +r^d$, 
with $r=|{\bf r}|:=\max_\ell r_\ell$.

In the case $d=2$, the orthogonal Tucker decomposition is equivalent to the singular 
value decomposition (SVD) of a rectangular matrix. 

An equivalent notation for the Tucker tensor format can be used,
\begin{equation}\label{eqn:Tucker_CS}
 {\bf V} = \boldsymbol{\beta} \times_1 V^{(1)} \times_2 V^{(2)}\ldots \times_d V^{(d)},
 \end{equation}
where $\times_\ell$ denotes the contraction along the mode $\ell$ and orthogonal matrices 
$V^{(\ell)}=[{\bf v}^{(\ell)}_{1} \ldots {\bf v}^{(\ell)}_{r_\ell}]\in 
\mathbb{R}^{n_\ell \times r_\ell}$ 
incorporate the set of orthogonal vectors $\{{\bf v}^{(\ell)}_{\nu_\ell}\}$. 
Likewise, the representation (\ref{eqn:CP_form}) can be written as the rank-$(R,\ldots,R)$ 
(non-orthogonal) Tucker tensor
\begin{equation}\label{eqn:CP_form_ContrpTuck}
 {\bf U}=\boldsymbol{\xi} \times_1 {U}^{(1)}\times_2 {U}^{(2)} \ldots \times_d {U}^{(d)}, 
\end{equation}
by introducing the so-called side matrices 
$
 U^{(\ell)}=[{\bf u}_1^{(\ell)}\ldots {\bf u}_R^{(\ell)}] \in \mathbb{R}^{n_\ell \times R}, 
$
$\ell=1, ...,d,$
obtained by concatenation of the canonical vectors ${\bf u}_k^{(\ell)}$, $k=1,\ldots R$, 
and the diagonal Tucker core tensor
$\boldsymbol{\xi}:=\mbox{diag} \{ \xi_1,\ldots,\xi_R\}\in \mathbb{R}^{R\times \ldots\times R}$ such that
$\xi_{\nu_1,\ldots,\nu_d}=0$ except when $\nu_1=\ldots=\nu_d$ with 
$\xi_{\nu,\ldots,\nu}=\xi_{\nu}$ ($\nu=1,\ldots,R$). 

The exceptional properties of the Tucker decomposition for the approximation of  
discretized multidimensional functions have been revealed in \cite{khor-rstruct-2006,KhKh:06}, 
where it was proven that for a class of function-related tensors the approximation error of the 
Tucker decomposition {decays exponentially} in the Tucker rank.  

Rank-structured tensor representations  provide fast multilinear algebra with linear 
complexity scaling in the dimension $d$. For example, 
for given canonical tensors (\ref{eqn:CP_form}), the Euclidean scalar product, the 
Hadamard product and $d$-dimensional convolution can be computed by simple tensor
operations in $1D$ complexity \cite{KhKh:06}.
In tensor-structured numerical methods,  calculation of the $d$-dimensional
convolution integrals  is replaced by a sequence of $1D$  scalar and Hadamard products, 
and $1D$ convolution transforms \cite{khor-ml-2009,KhKh:06},
leading to $O( d n \log n)$ computational work instead of $O(n^d)$.
However, the multilinear tensor operations in the above mentioned formats mandatory lead to increase 
of tensor ranks which can be then reduced by the canonical-to-Tucker and Tucker-to-canonical algorithms  
introduced in \cite{KhKh:06,khor-ml-2009}, see Appendix for the description of the canonical-to-Tucker
algorithm.

\subsection{Canonical tensor representation of the 3D Newton kernel}
 \label{ssec:Coulomb}
 
Methods of separable approximation to the 3D Newton kernel (electrostatic potential) 
using the Gaussian sums have been addressed in the chemical and mathematical literature 
since \cite{Boys:56} and \cite{Stenger,Braess:BookApTh:86,HaKhtens:04I}, respectively.
The approach to tensor decomposition for a class of lattice -structured 
interaction potentials $p(\|x\|)$ was presented in \cite{VeBoKh:Ewald:14,VeKhor_NLLA:15}.
In this section, we recall the grid-based method for the low-rank canonical  
representation of a spherically symmetric kernel function $p(\|x\|)$, 
$x\in \mathbb{R}^d$ for $d=1,2,3$, by its projection onto the set
of piecewise constant basis functions, see \cite{BeHaKh:08} for the case of 
Newton and Yukawa kernels 
$p(\|x\|)=\frac{1}{\|x\|}$, and $p(\|x\|)=\frac{e^{-\lambda \|x\|}}{\|x\|}$, for $x\in \mathbb{R}^3$.
The single reference potential like $1/\|x\|$ can be 
represented on a fine 3D $n\times n\times n$ Cartesian grid
in the form low-rank canonical tensor \cite{HaKhtens:04I,BeHaKh:08}.

In the computational domain  $\Omega=[-b,b]^3$, 
let us introduce the uniform $n \times n \times n$ rectangular Cartesian grid $\Omega_{n}$
with mesh size $h=2b/n$ ($n$ even).
Let $\{ \psi_\textbf{i}\}$ be a set of tensor-product piecewise constant basis functions,
$  \psi_\textbf{i}(\textbf{x})=\prod_{\ell=1}^3 \psi_{i_\ell}^{(\ell)}(x_\ell)$,
for the $3$-tuple index ${\bf i}=(i_1,i_2,i_3)$, $i_\ell \in \{1,...,n\}$, $\ell=1,\, 2,\, 3 $.
The generating kernel $p(\|x\|)$ is discretized by its projection onto the basis 
set $\{ \psi_\textbf{i}\}$
in the form of a third order tensor of size $n\times n \times n$, defined entry-wise as
\begin{equation}  \label{galten}
\mathbf{P}:=[p_{\bf i}] \in \mathbb{R}^{n\times n \times n},  \quad
 p_{\bf i} = 
\int_{\mathbb{R}^3} \psi_{\bf i} ({x}) p(\|{x}\|) \,\, \mathrm{d}{x}.
\end{equation}

The low-rank canonical decomposition of the $3$rd order tensor $\mathbf{P}$ is based 
on using exponentially convergent 
$\operatorname*{sinc}$-quadratures for approximation of the Laplace-Gauss transform 
to the analytic function $p(z)$, $z \in \mathbb{C}$, specified by a certain weight $a(t) >0$,
\begin{align} \label{eqn:laplace} 
p(z)=\int_{\mathbb{R}_+} a(t) e^{- t^2 z^2} \,\mathrm{d}t \approx
\sum_{k=-M}^{M} a_k e^{- t_k^2 z^2} \quad \mbox{for} \quad |z| > 0,\quad z \in \mathbb{R},
\end{align} 
where the quadrature points and weights are given by 
\begin{equation} \label{eqn:hM}
t_k=k \mathfrak{h}_M , \quad a_k=a(t_k) \mathfrak{h}_M, \quad 
\mathfrak{h}_M=C_0 \log(M)/M , \quad C_0>0.
\end{equation}
Under the assumption $0< a \leq |z |  < \infty$
this quadrature can be proven to provide an exponential convergence rate in $M$
for a class of analytic functions $p(z)$, see \cite{Stenger,HaKhtens:04I}. 
In particular, 
for the Newton kernel, $p(z)=1/z$, the Laplace-Gauss transform takes the form
\[
 \frac{1}{z}= \frac{2}{\sqrt{\pi}}\int_{\mathbb{R}_+} e^{- z^2 t^2 } dt, \quad 
 \mbox{where}\quad z=\sqrt{x_1^2 + x_2^2  + x_3^2}.
\]

Now, for any fixed $x=(x_1,x_2,x_3)\in \mathbb{R}^3$, 
such that $\|{x}\| > a > 0$, 
we apply the $\operatorname*{sinc}$-quadrature approximation (\ref{eqn:laplace}), (\ref{eqn:hM}) 
to obtain the separable expansion
\begin{equation} \label{eqn:sinc_Newt}
 p({\|{x}\|}) =   \int_{\mathbb{R}_+} a(t)
e^{- t^2\|{x}\|^2} \,\mathrm{d}t  \approx 
\sum_{k=-M}^{M} a_k e^{- t_k^2\|{x}\|^2}= 
\sum_{k=-M}^{M} a_k  \prod_{\ell=1}^3 e^{-t_k^2 x_\ell^2},
\end{equation}
providing an exponential convergence rate in $M$,
\begin{equation} \label{sinc_conv}
\left|p({\|{x}\|}) - \sum_{k=-M}^{M} a_k e^{- t_k^2\|{x}\|^2} \right|  
\le \frac{C}{a}\, \displaystyle{e}^{-\beta \sqrt{M}},  
\quad \text{with some} \ C,\beta >0.
\end{equation}
Combining \eqref{galten} and \eqref{eqn:sinc_Newt}, and taking into account the 
separability of the Gaussian basis functions, we arrive at the low-rank 
approximation to each entry of the tensor $\mathbf{P}$,
\begin{equation*} \label{eqn:C_nD_0}
 p_{\bf i} \approx \sum_{k=-M}^{M} a_k   \int_{\mathbb{R}^3}
 \psi_{\bf i}({x}) e^{- t_k^2\|{x}\|^2} \mathrm{d}{x}
=  \sum_{k=-M}^{M} a_k  \prod_{\ell=1}^{3}  \int_{\mathbb{R}}
\psi^{(\ell)}_{i_\ell}(x_\ell) e^{- t_k^2 x^2_\ell } \mathrm{d} x_\ell.
\end{equation*}

\begin{figure}[htb]
\centering
\includegraphics[width=8.0cm]{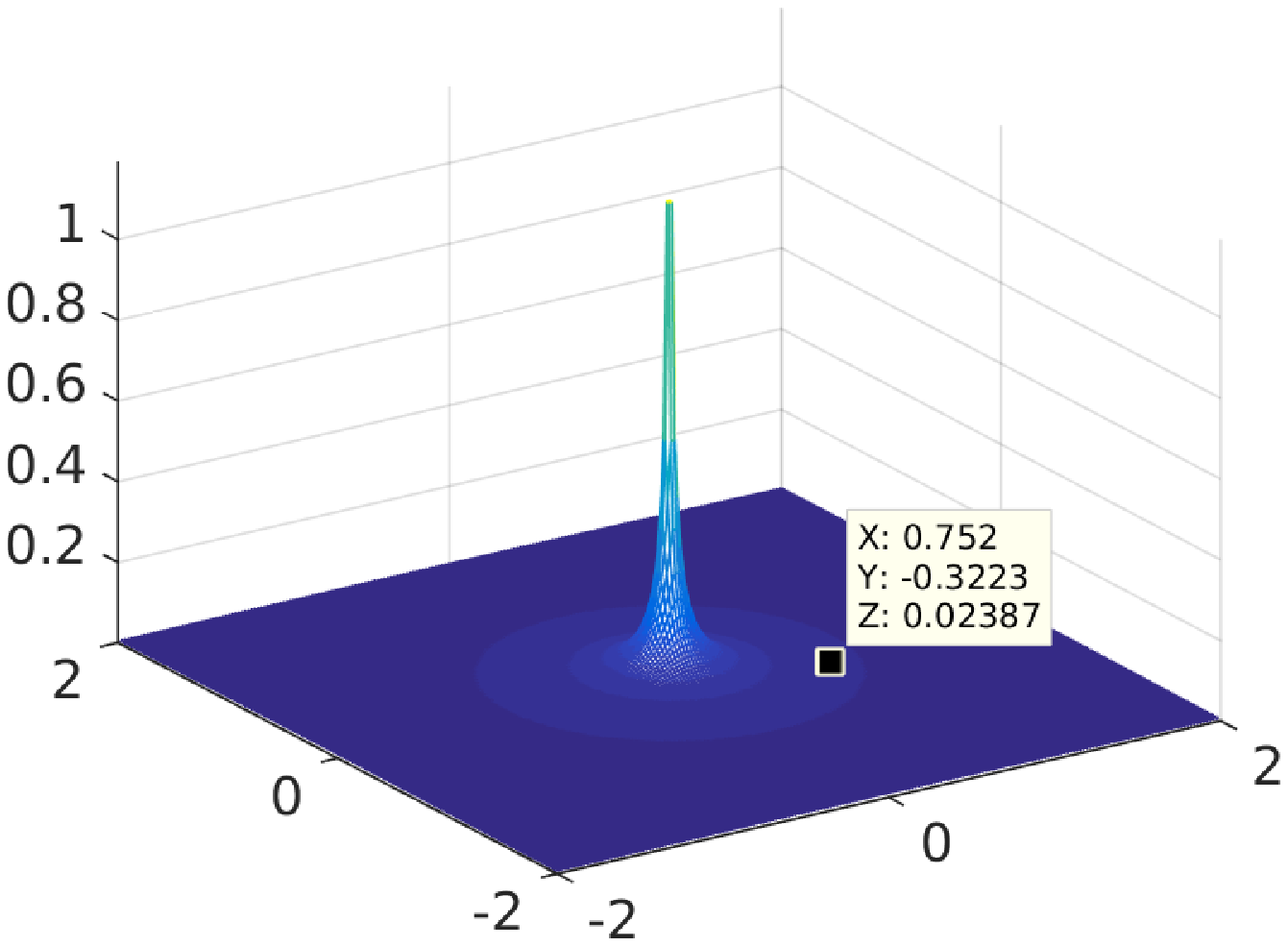}
\includegraphics[width=8.0cm]{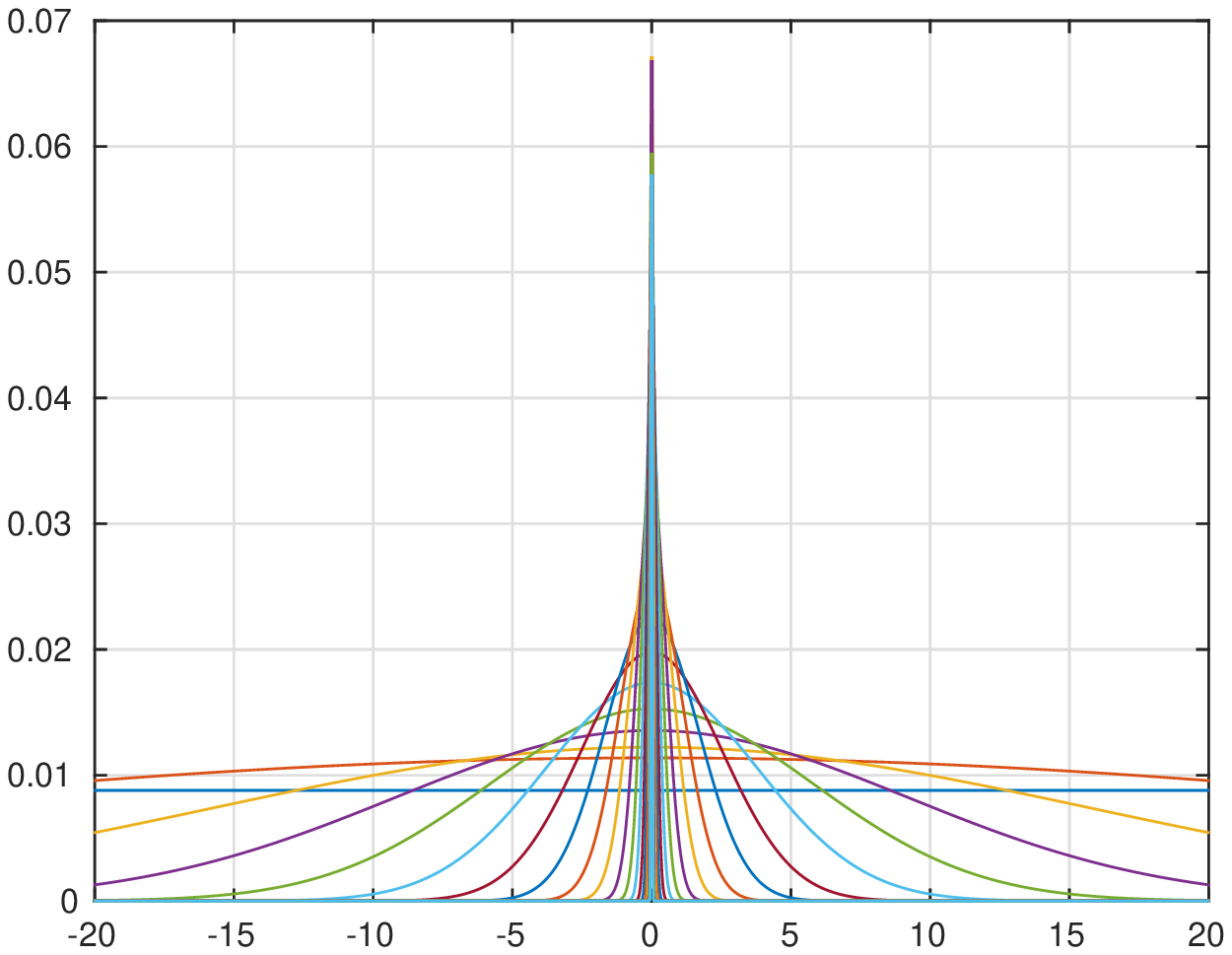}
\caption{Vectors  of the canonical tensor representation,  $\{{\bf p}^{(1)}_q\}_{q=1}^R$,
for the Newton kernel displayed along  $x$-axis: $n=1024$, $R=20$. }
\label{fig:1024_r20}
\end{figure}

Define the vector (recall that $a_k >0$)
\begin{equation*} \label{eqn:galten_int}
\textbf{p}^{(\ell)}_k
= a_k^{1/3} \left[b^{(\ell)}_{i_\ell}(t_k)\right]_{i_\ell=1}^{n_\ell} \in \mathbb{R}^{n_\ell}
\quad \text{with } \quad b^{(\ell)}_{i_\ell}(t_k)= 
\int_{\mathbb{R}} \psi^{(\ell)}_{i_\ell}(x_\ell) e^{- t_k^2 x^2_\ell } \mathrm{d}x_\ell,
\end{equation*}
then the $3$rd order tensor $\mathbf{P}$ can be approximated by 
the $R$-term ($R=2M+1$) canonical representation
\begin{equation} \label{eqn:sinc_general}
    \mathbf{P} \approx  \mathbf{P}_R =
\sum_{k=-M}^{M} a_k \bigotimes_{\ell=1}^{3}  {\bf b}^{(\ell)}(t_k)
= \sum\limits_{k=-M}^{M} {\bf p}^{(1)}_k \otimes {\bf p}^{(2)}_k \otimes {\bf p}^{(3)}_k
\in \mathbb{R}^{n\times n \times n}, \quad {\bf p}^{(\ell)}_k \in \mathbb{R}^n.
\end{equation}
Given a threshold $\varepsilon >0 $,  $M$ can be chosen as the minimal number
such that in the max-norm
\begin{equation*} \label{eqn:error_control}
\| \mathbf{P} - \mathbf{P}_R \|  \le \varepsilon \| \mathbf{P}\|.
\end{equation*}
The skeleton vectors can be re-numerated by $k \mapsto q=k+M+1$, 
${\bf p}^{(\ell)}_k \mapsto {\bf p}^{(\ell)}_{q}$, ($q=1,...,R$), $\ell=1, 2, 3$.
The canonical tensor ${\bf P}_{R}$ in (\ref{eqn:sinc_general})
approximates the 3D symmetric kernel function 
$p({\|x\|})$ ($x\in \Omega$), centered at the origin, such that 
${\bf p}^{(1)}_q={\bf p}^{(2)}_q={\bf p}^{(3)}_q$ ($q=1,...,R$).

In the case of the Newton kernel the term ${\bf p}^{(\ell)}_k$ equals  ${\bf p}^{(\ell)}_{-k}$, 
and the sum (\ref{eqn:sinc_general}) reduces to $k=0,1,...,M$, implying $R=M+1$.
Figure \ref{fig:1024_r20} displays the canonical vectors in the   tensor 
representation  (\ref{eqn:sinc_general}) for the Newton kernel along the $x$-axis from a set 
$\{{\bf p}^{(1)}_q\}_{q=1}^R$.
It is clearly seen that there are canonical  vectors representing the long- 
and short-range contributions to the total electrostatic potential. 
This interesting feature was also recognized for the rank-structured 
tensors representing a lattice sum of potentials \cite{VeBoKh:Ewald:14,VeKhor_NLLA:15}. 

\subsection{Tensor splitting of the kernel into long- and short-range parts}\label{ssec:S_L_split}

 From the definition of the quadrature (\ref{eqn:sinc_general}), (\ref{eqn:hM}), we can easily observe
that the full set of approximating Gaussians includes two classes of functions: those with 
small "effective support" and the long-range functions. Clearly, functions from different classes 
may require different tensor-based schemes for their efficient numerical treatment.
Hence, the idea of the new approach is the constructive implementation of a 
range separation scheme that allows the independent efficient treatment 
of both the long- and short-range parts in the approximating kernel.

In the following, without loss of generality, we confine ourselves to the case of the Newton kernel, 
so that the sum in (\ref{eqn:sinc_general}) reduces to $k=0,1,\ldots,M$ (due to symmetry argument).
>From (\ref{eqn:hM}) we observe that the sequence of quadrature points $\{t_k\}$,   
can be split into two subsequences, 
$$
{\cal T}:=\{t_k|k=0,1,\ldots,M\}={\cal T}_l \cup {\cal T}_s,
$$ 
with
\begin{equation} \label{eqn:Split_Qpoints}
{\cal T}_l:=\{t_k\,|k=0,1,\ldots,R_l\}, \quad  
\mbox{and} \quad {\cal T}_s:=\{t_k\,|k=R_l+1,\ldots,M\}.
\end{equation}
Here
${\cal T}_l$ includes quadrature points $t_k$ condensed ``near'' zero, hence generating 
the long-range Gaussians (low-pass filters), 
and ${\cal T}_s$ accumulates the increasing in $M\to \infty$ 
sequence of ``large'' sampling points $t_k$ with the upper bound $C_0^2 \log^2(M)$, 
corresponding to the short-range Gaussians (high-pass filters).
Notice that the quasi-optimal choice of the constant $C_0\approx 3$ was determined in \cite{BeHaKh:08}.
We futher denote ${\cal K}_l:=\{k\,| k=0,1,\ldots,R_l\}$ and ${\cal K}_s:=\{k\,| k=l+1,\ldots,M \}$. 

Splitting (\ref{eqn:Split_Qpoints}) generates the additive decomposition of the canonical tensor
$\mathbf{P}_R$ onto the short- and long-range parts,
\[
 \mathbf{P}_R = \mathbf{P}_{R_s} + \mathbf{P}_{R_l},
\]
where
\begin{equation} \label{eqn:Split_Tens}
    \mathbf{P}_{R_s} =
\sum\limits_{t_k\in {\cal T}_s} {\bf p}^{(1)}_k \otimes {\bf p}^{(2)}_k \otimes {\bf p}^{(3)}_k, 
\quad \mathbf{P}_{R_l} =
\sum\limits_{t_k\in {\cal T}_l} {\bf p}^{(1)}_k \otimes {\bf p}^{(2)}_k \otimes {\bf p}^{(3)}_k.
\end{equation}

The choice of the critical number $R_l=\# {\cal T}_l -1$ (or equivalently, $R_s=\# {\cal T}_s = M-R_l$), 
that specifies the splitting ${\cal T}={\cal T}_l \cup {\cal T}_s$,
is determined by the \emph{active support} 
of the short-range components 
such that one can cut off the functions ${\bf p}_k(x)$, $t_k\in {\cal T}_s$, 
outside of the sphere $B_{\sigma}$ of radius ${\sigma}>0$, 
subject to a certain threshold $\delta>0$.
For fixed $\delta>0$, the choice of $R_s$ is uniquely defined by 
the (small) parameter $\sigma$ and vise versa. 
The following two basic criteria, corresponding to (A) the max- and (B) $L^1$-norms estimates 
can be applied given $\sigma$:
\begin{equation} \label{eqn:Split_crit_max}
 (A) \quad {\cal T}_s=\{t_k:\,a_k e^{-t_k^2 \sigma^2} \leq \delta \}\;\Leftrightarrow \;
 R_l=\min k : \,  a_k e^{-t_k^2 \sigma^2} \leq \delta, 
\end{equation}
or
\begin{equation} \label{eqn:Split_crit_L2}
 (B)\quad {\cal T}_s:=\{t_k:\, a_k \int_{B_{\sigma}} e^{-t_k^2 x^2} dx \leq \delta\}\;
 \Leftrightarrow \; R_l=\min k : \,
  a_k \int_{B_{\sigma}} e^{-t_k^2 x^2} dx \leq \delta.
\end{equation}
The quantitative estimates on the value of $R_l$ can be easily calculated by using the explicit equation
(\ref{eqn:hM}) for the quadrature parameters. 
For example, in case $C_0=3$ and $a(t)=1$, criteria (A) implies that $R_l$ solves the equation
\[
 \left(\frac{3R_l \log M}{M}\right)^2 \sigma^2 = \log(\frac{\mathfrak{h}_M}{\delta}).
\]
Criteria (\ref{eqn:Split_crit_max}) and (\ref{eqn:Split_crit_L2}) can be slightly modified 
depending on the particular applications to many-particles systems.
For example, in electronic structure calculations, the parameter $\sigma$ can be associated with the 
typical inter-atomic distance in the molecular system of interest.

\begin{figure}[htb]
\centering
\includegraphics[width=8.0cm]{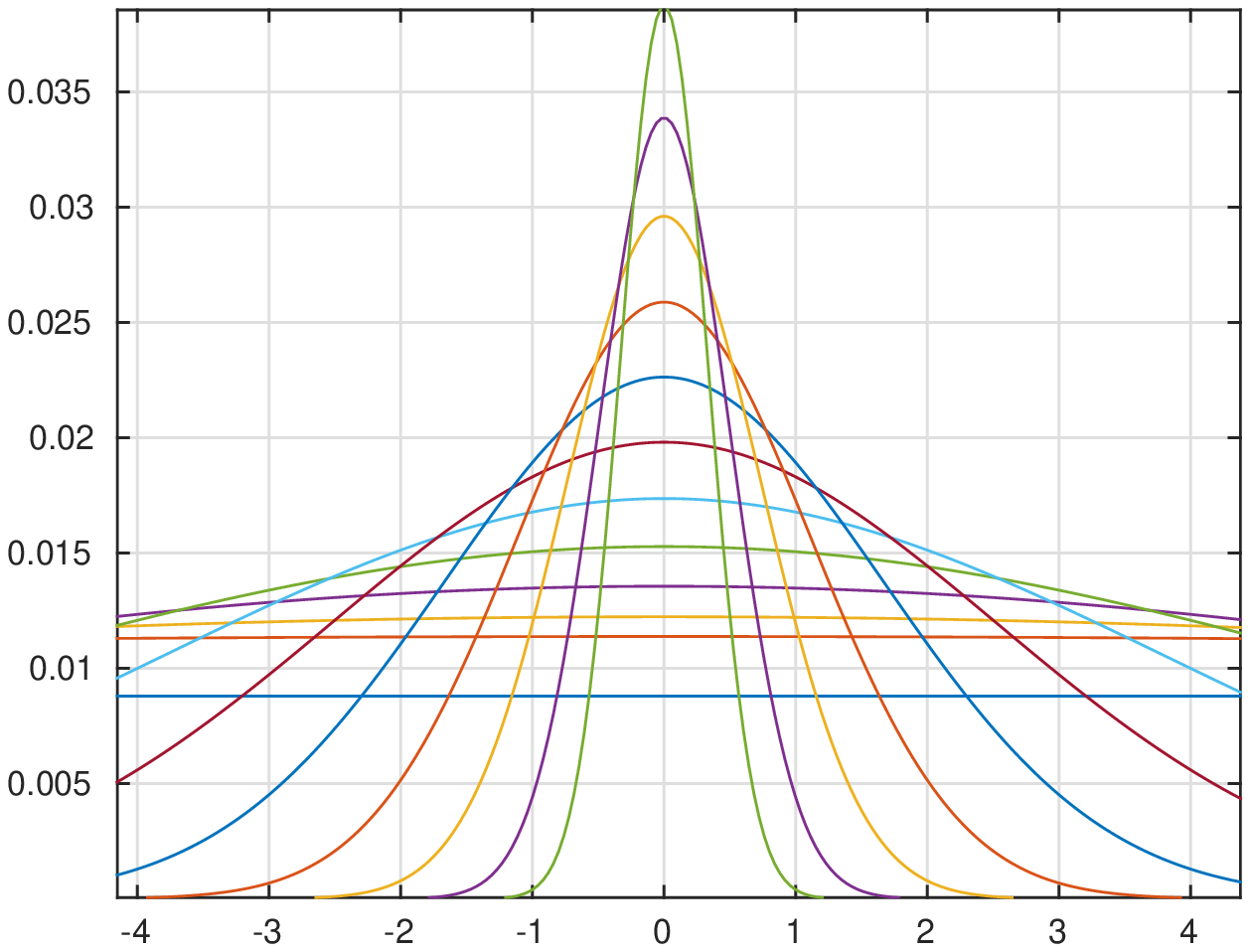} 
\includegraphics[width=8.0cm]{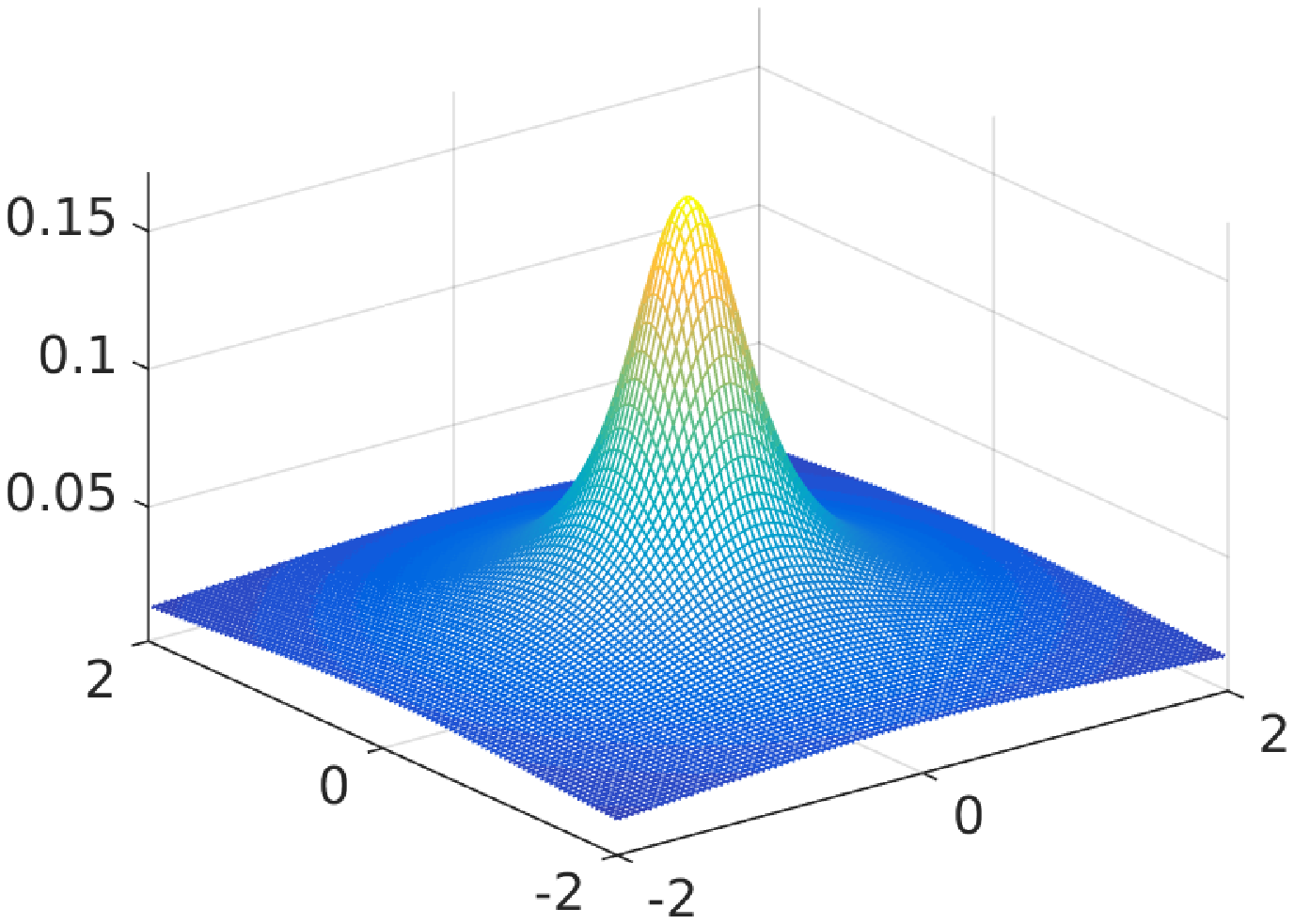}
\caption{Long-range canonical vectors for $n=1024$, $R=20, R_l=12$ and the corresponding
potential. }
\label{fig:1024_rl12}
\end{figure}

Figures \ref{fig:1024_rl12} and \ref{fig:1024_rs8}  illustrate the splitting (\ref{eqn:Split_Qpoints}) 
for the tensor ${\bf P}_R$ computed on 
the $n\times n\times n$ grid with the parameters $R=20, R_l=12$ and $R_s=8$, respectively.
Figure \ref{fig:1024_rl12} shows the long-range canonical vectors from ${\bf P}_{R_l}$ in (\ref{eqn:Split_Tens}), 
while Figure \ref{fig:1024_rs8} displays the short-range part described by ${\bf P}_{R_s}$.
Following criteria (A) with $\delta \approx 10^{-4}$, the effective support 
for this splitting is determined by $\sigma =0.9$.
It can be seen that the complete Newton kernel depicted in Figure \ref{fig:1024_r20} covers 
the long-range behavior, 
while the function values of the tensor ${\bf P}_{R_s}$ vanish exponentially fast 
apart of the effective support, as can be seen in Figure \ref{fig:1024_rs8}.

\begin{figure}[htb]
\centering
\includegraphics[width=8.0cm]{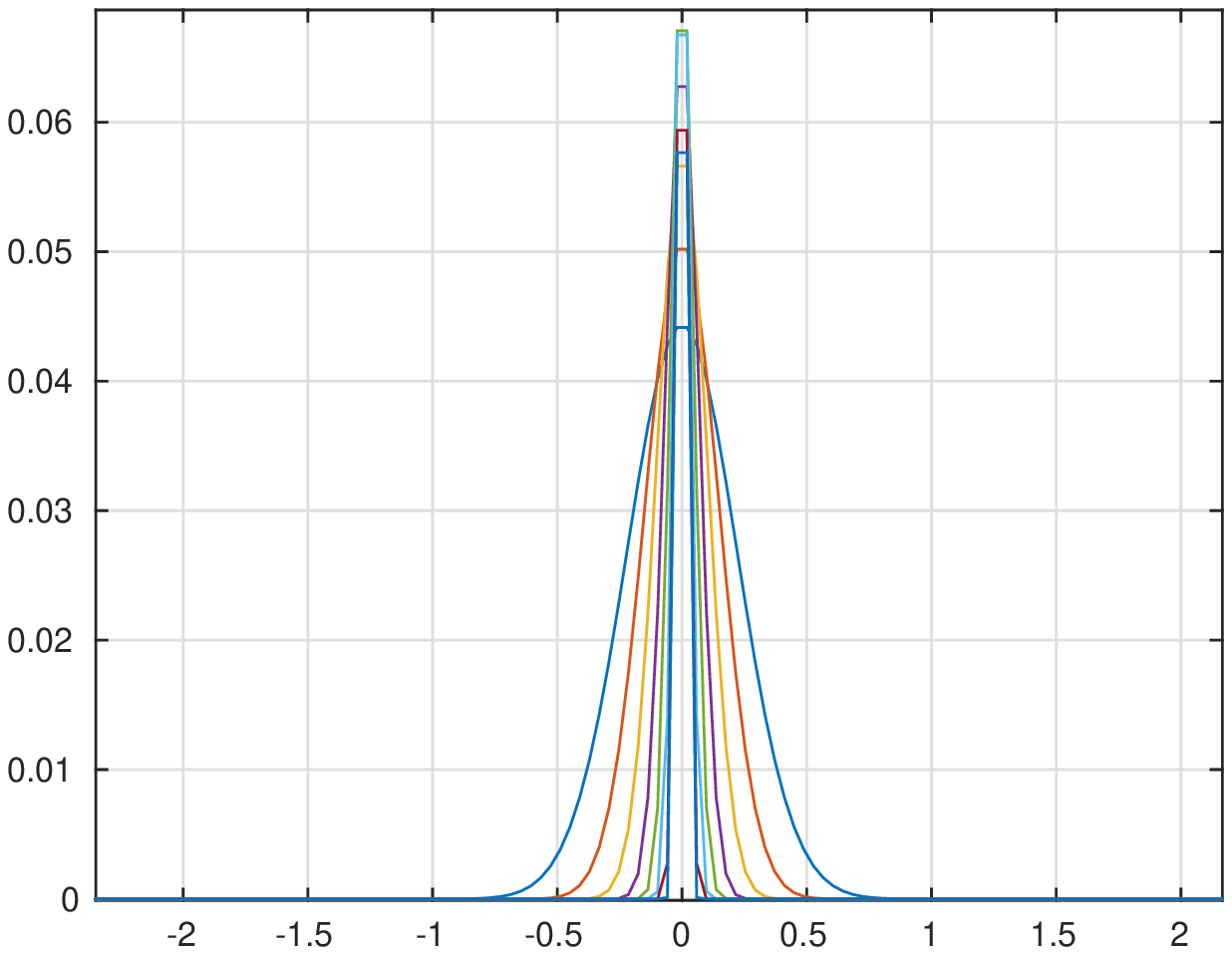}
\includegraphics[width=8.0cm]{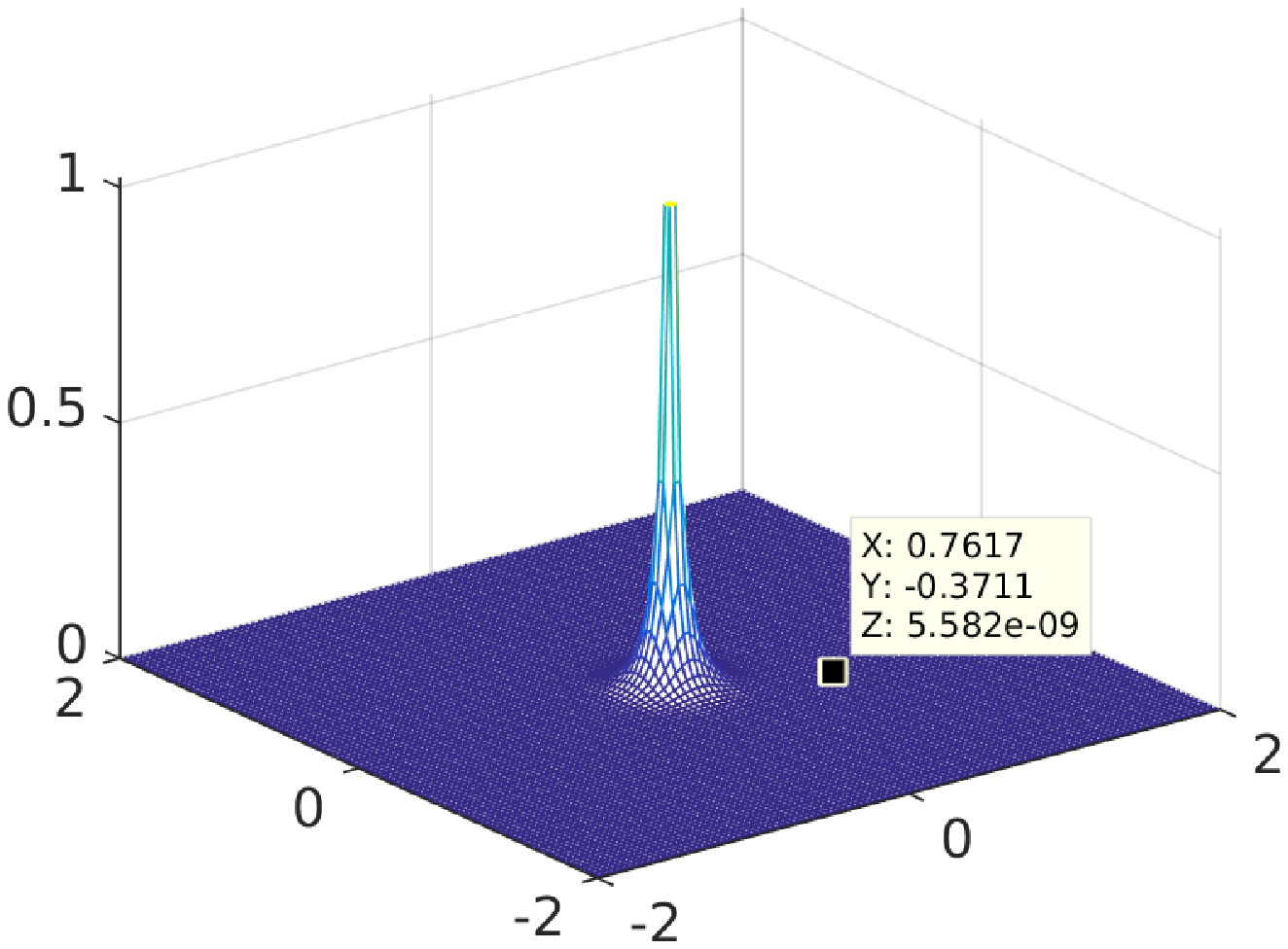}
\caption{Short-range canonical vectors for $n=1024$, $R=20, R_s=8$,
 and the corresponding potential.}
\label{fig:1024_rs8}
\end{figure}

Inspection of the quadrature point distribution in (\ref{eqn:hM}) shows that the 
short- and long-range subsequences are nearly equally balanced distributed, 
so that one can expect approximately 
\begin{equation} \label{eqn:ChoiceR_l}
 R_s \approx R_l=M/2.
\end{equation} 
The optimal choice may depend on the particular applications.

The advantage of the range separation in the splitting of 
the canonical tensor ${\bf P}_R \mapsto {\bf P}_{R_s} + {\bf P}_{R_l}$  in 
(\ref{eqn:Split_Tens}) is the opportunity for 
independent tensor representations of both sub-tensors ${\bf P}_{R_s}$ and ${\bf P}_{R_l}$ 
providing the separate treatment of the short- and long-range parts in the sum of 
many interaction potentials.

Finally, we notice that the \emph{range separation principle} can be generalized to 
more than two-term splitting, taking into account the  requirements of specific applications.

\section{Tensor summation of range-separated potentials}\label{sec:Fast_Sum_Split}

In this section we describe how the range separated tensor representation of the 
generating potential function can be applied  for the fast and accurate grid-based 
computation of a large sum of non-local potentials centered at arbitrary locations 
in the 3D volume.   This is the bottleneck problem in numerical modeling
of large 
$N$-particle systems. 

\subsection{Quasi-uniformly separable point distributions}\label{ssec:Separ_PointsDistr}

One of the main limitations for the use of direct grid-based canonical/Tucker approximations 
to the large potential sums is due to the strong increase in tensor rank proportionally to the 
number of particles $N$ in a system. Figure \ref{fig:Full_Rank} shows the Tucker ranks for 
the protein-type system consisting of $N=783$ atoms.
\begin{figure}[htb]
\centering
\includegraphics[width=7.0cm]{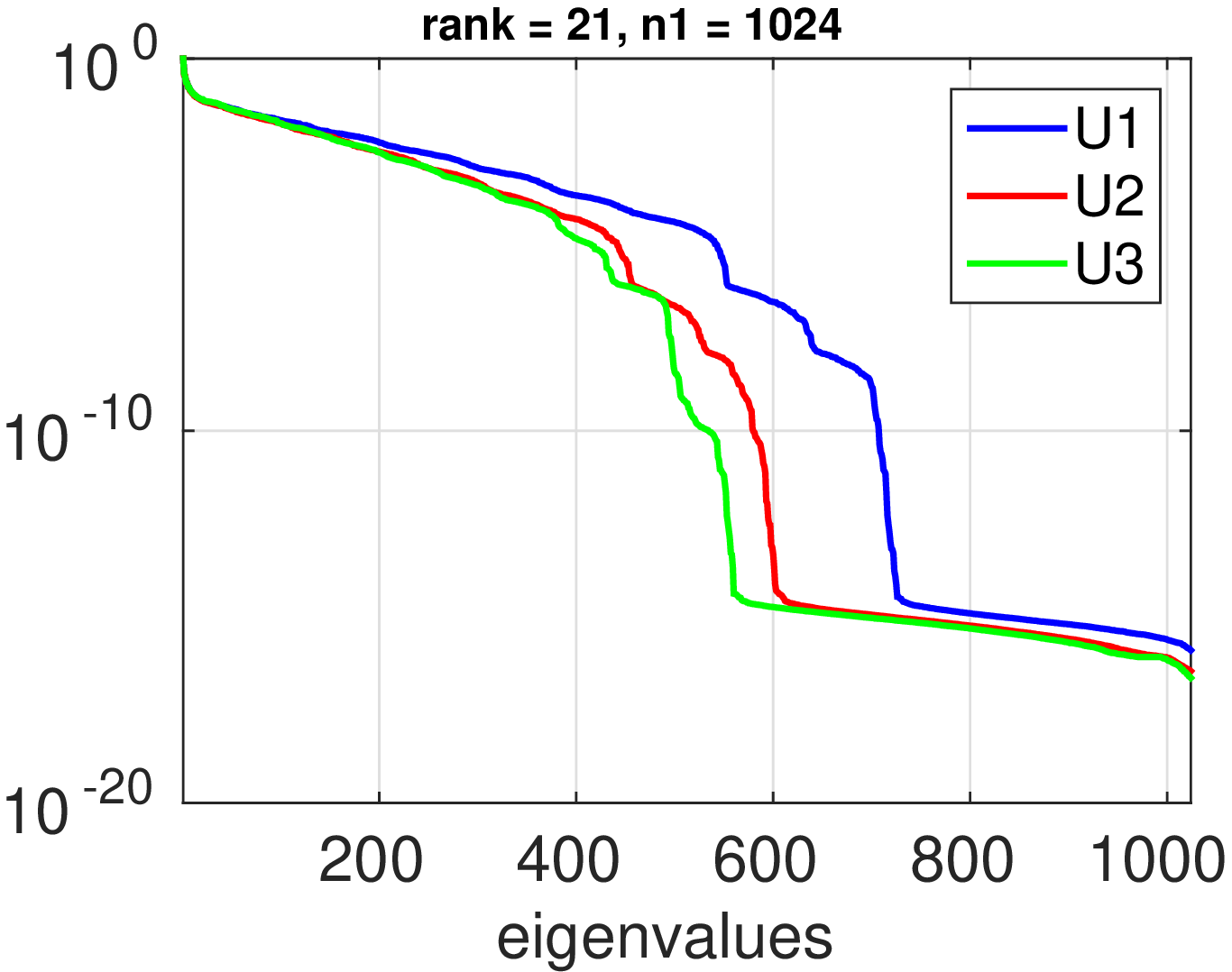} 
\includegraphics[width=7.0cm]{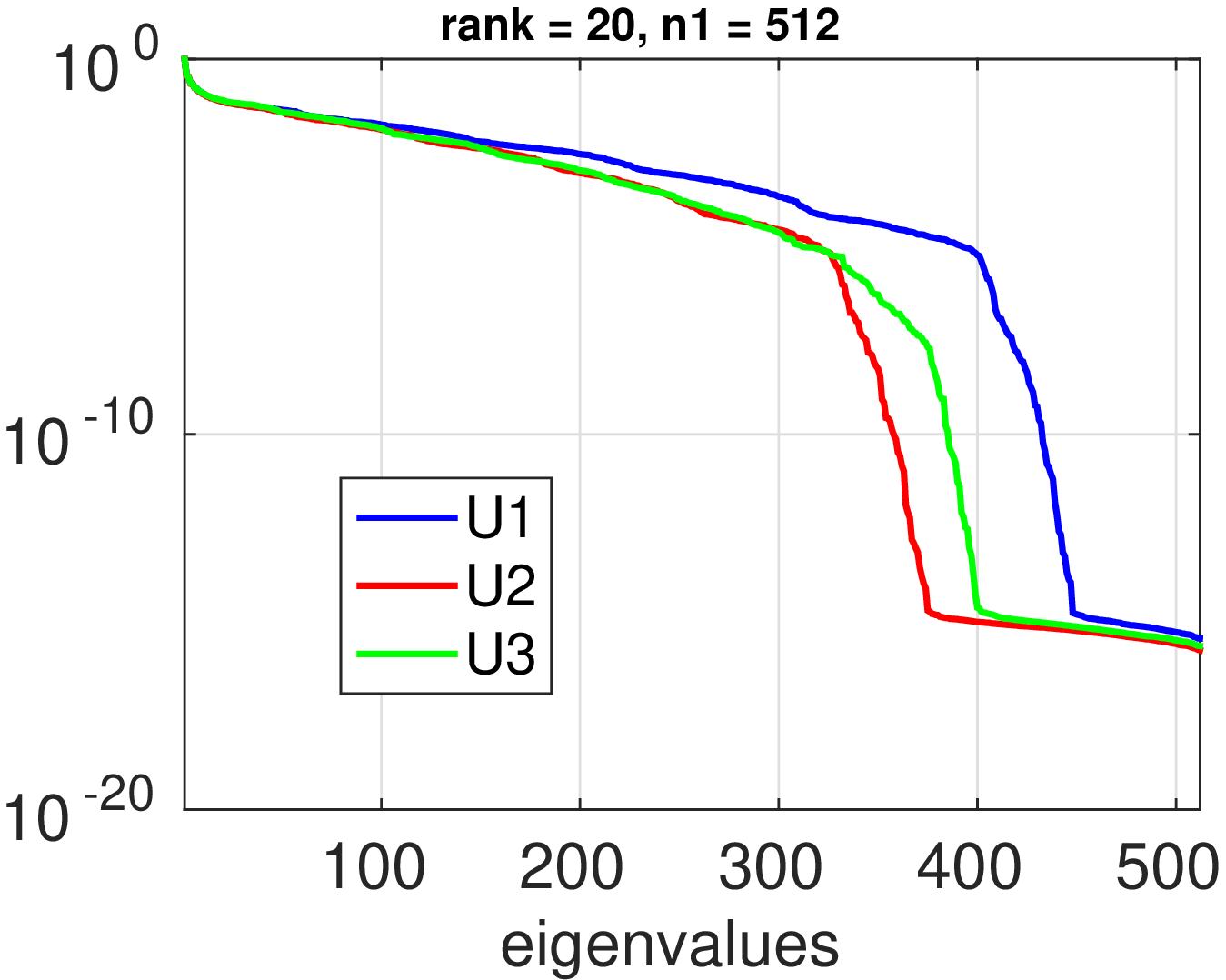}
\caption{The directional Tucker ranks computed by RHOSVD for a protein-type system with n=1024 (left) 
and n=512 (right). }
\label{fig:Full_Rank}
\end{figure}

Given the generating kernel $p(\|x\|)$, we consider the problem of efficient calculating
the weighted sum of a large number of single potentials located in a set ${\cal S}$ of 
separable distributed  points (sources), 
$x_\nu \in \mathbb{R}^3$, $\nu=1,...,N$, embedded into the fixed bounding box $\Omega=[-b,b]^3$,
\begin{equation}\label{eqn:PotSum}
 P_0(x)= \sum_{\nu=1}^{N} {z_\nu}\,p({\|x-x_\nu\|}), \quad z_\nu \in \mathbb{R}.
\end{equation}
 The function $p(\|x\|)$ is allowed to have slow polynomial decay in $1/\|x\|$ so that
  each individual source contributes essentially to the total potential at each point in $\Omega$.
  
\begin{definition}\label{def:QU_Distrib} (Well-separable point distribution).
Given a constant $\sigma_\ast>0$, a set ${\cal S}=\{x_\nu\}$ of points in $\mathbb{R}^d$ 
is called $\sigma_\ast$-separable if there holds
\begin{equation} \label{eqn:Qunif_points}
 d(x_\nu,s_{\nu'}):=\|x_\nu - s_{\nu'} \|\geq \sigma_\ast \quad \mbox{for all} \quad \nu\neq \nu'.
\end{equation}
A family of point sets $\{{\cal S}_{1},...,{\cal S}_{m}\}$,  is called uniformly 
$\sigma_\ast$-separable
if (\ref{eqn:Qunif_points}) holds for every set ${\cal S}_{m'}$, ${m'}=1,2,...,m$ independently 
of the number of particles in a set, $\#{\cal S}_{m'}$.
\end{definition}
Condition (\ref{eqn:Qunif_points}) can be reformulated in terms of the so-called separation
distance $q_{\cal S}$ of the point set ${\cal S}$
\begin{equation} \label{eqn:SepDist}
q_{\cal S}:= \min_{s\in {\cal S}} \min_{x_\nu\in {\cal S}\setminus s} d(x_\nu,s) \geq \sigma_\ast.
\end{equation}

Definition \ref{def:QU_Distrib} on 
separability of point distributions is fulfilled, in particular, 
in the case of large molecular systems (proteins, crystals, polymers, nano-clusters),
where all atomic centers are strictly separated
from each other by a certain fixed \emph{inter-atomic distance}. 
The same happens for lattice-type structures, where each atomic cluster within the unit cell 
is separated from the neighbors  by a distance proportional to the lattice step-size.

\begin{figure}[htb]
\centering
\includegraphics[width=7.0cm]{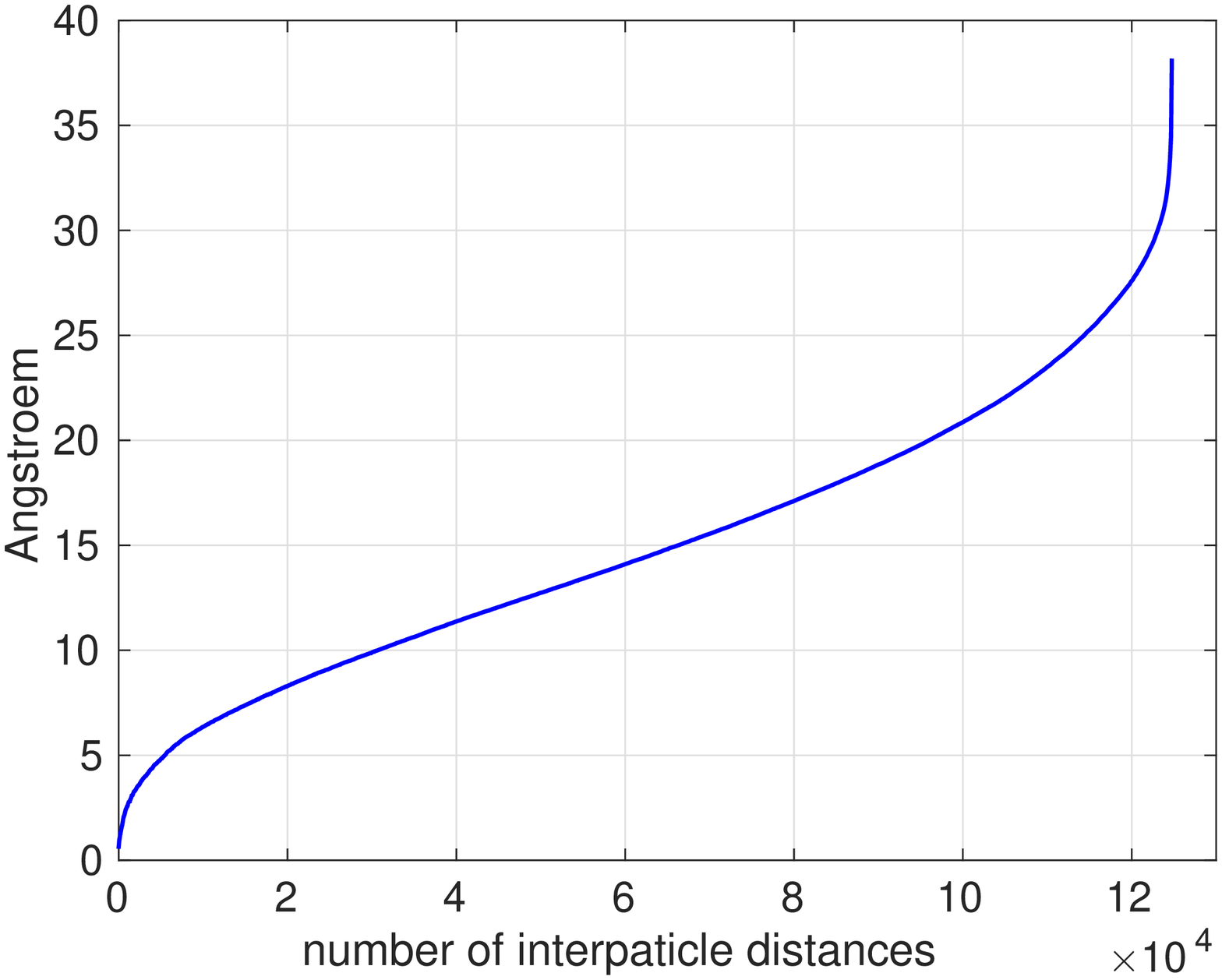} 
\includegraphics[width=7.0cm]{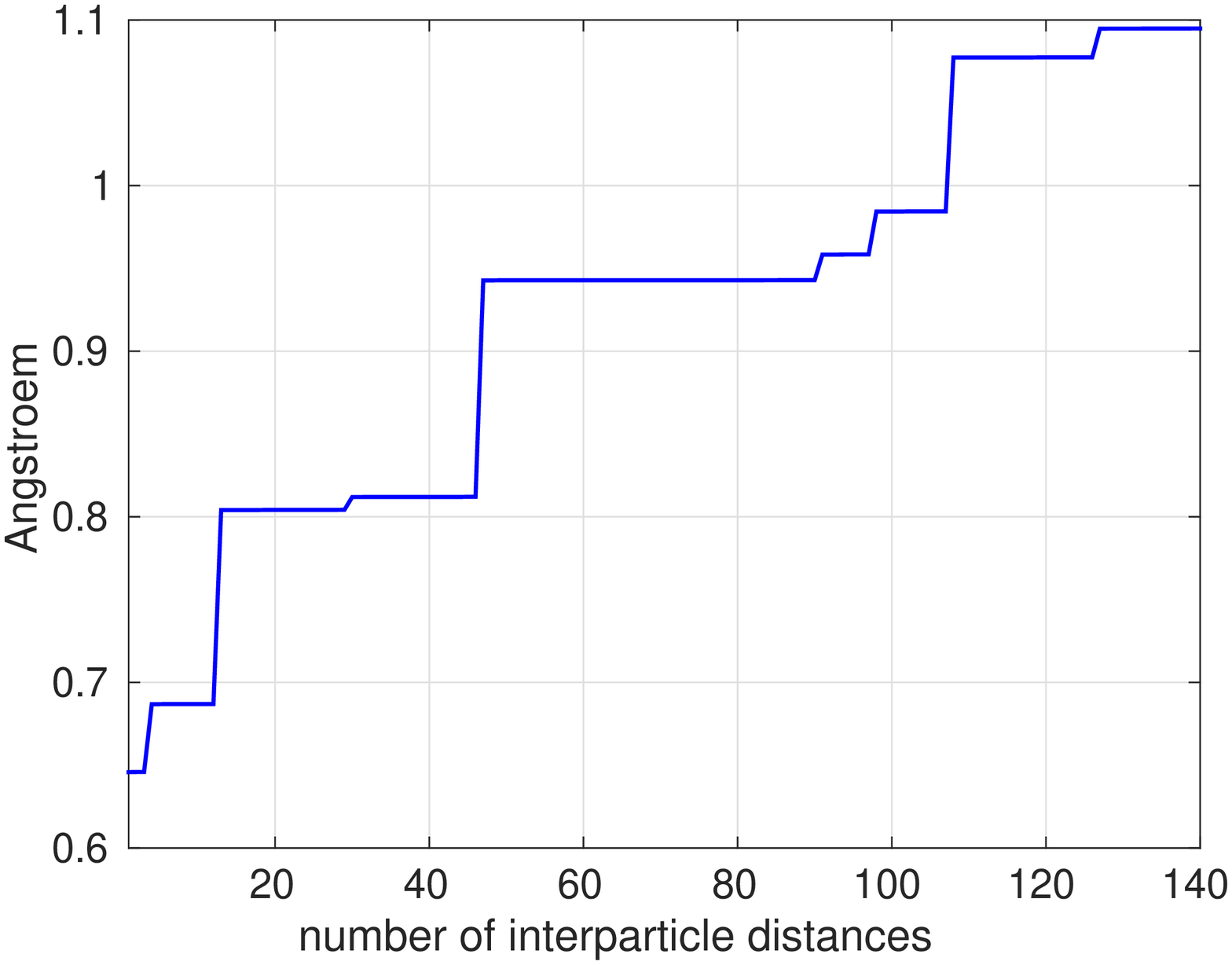}
\caption{Inter-particle distances in a ascendant order for protein-type structure with 
500 particles (left), zoom for the first 100 smallest inter-particle  distances (right) .}
\label{fig:1024_dist}
\end{figure}

Figure \ref{fig:1024_dist} (left) shows inter-particle distances in ascending order
for a protein-type structure with 500 particles. 
The total number of distances equals to $N(N-1)/2$, where $N$ is the number of particles.
Figure \ref{fig:1024_dist} (right) indicates that the number of particles with small
inter-particle distances is very moderate. In particular, for this example the number of pairs with
interparticle distances less than $1$\AA{} is about $0.04\, \%$ ($\approx 110)$) 
of the total number of $2,495 \cdot 10^5$ distances. 


In the following, for ease of presentation, we confine ourselves to the case
of electrostatic potentials described by the Newton kernel $p(\|x\|)=\frac{1}{\|x\|}$.

\subsection{Low-rank representation to the sum of long-range terms} \label{ssec:Sum_LongRange}

First, we describe the tensor summation method for calculation of the 
collective potential of a multi-particle system that includes only 
the long-range contribution from the generating kernel. We introduce 
the $n\times n \times n$ rectangular grid $\Omega_n$ in $\Omega=[-b,b]^3$, see \S\ref{ssec:Coulomb},
as well as the auxiliary $2n\times 2n \times 2n$ grid on the accompanying domain 
$\widetilde{\Omega}=2 \Omega$ of  double size. 
The canonical rank-$R$ representation of the Newton kernel projected onto the
$n\times n \times n$ grid is denoted by ${\bf P}_R\in \mathbb{R}^{n \times n \times n}$, 
see (\ref{eqn:sinc_general}).

Consider the splitting (\ref{eqn:Split_Tens})
applied to the reference canonical tensor ${\bf P}_R$ 
and to its accompanying version $\widetilde{\bf P}_R=[\widetilde{p}_R(i_1,i_2,i_3)]$, 
$i_\ell \in I_\ell$, $\ell=1,2,3$, such that
\[
 \widetilde{\bf P}_R = \widetilde{\mathbf{P}}_{R_s} + \widetilde{\mathbf{P}}_{R_l} 
 \in \mathbb{R}^{2n \times 2n \times 2n}.
\]
For technical reasons, we further assume that the tensor grid $\Omega_n$ is fine enough such that
all charge centers ${\cal S}=\{x_\nu\}$ specifying the total electrostatic potential 
in (\ref{eqn:PotSum}) belong to the set of grid points, i.e., 
$x_\nu=(x_{\nu,1}, x_{\nu,2}, x_{\nu,3} )^T =h(j^{(\nu)}_1,j^{(\nu)}_2,j^{(\nu)}_3)^T\in \Omega_h$ 
with some indices $1\leq j^{(i)}_1,j^{(i)}_2,j^{(i)}_3 \leq n$.

The total electrostatic potential $P_0(x)$ in (\ref{eqn:PotSum}) is represented by a projected tensor 
${\bf P}_0\in \mathbb{R}^{n \times n \times n}$ that can 
be constructed by a direct sum of shift-and-windowing transforms of the reference 
tensor $\widetilde{\bf P}_R$ (see \cite{VeBoKh:Ewald:14} for more details),
\begin{equation}\label{eqn:Total_Sum}
 {\bf P}_0 = \sum_{\nu=1}^{N} {z_\nu}\, {\cal W}_\nu (\widetilde{\bf P}_R)=
 \sum_{\nu=1}^{N} {z_\nu} \, {\cal W}_\nu (\widetilde{\mathbf{P}}_{R_s} + \widetilde{\mathbf{P}}_{R_l})
 =: {\bf P}_s + {\bf P}_l.
\end{equation}
The shift-and-windowing transform ${\cal W}_\nu$ maps a reference tensor 
$\widetilde{\bf P}_R\in \mathbb{R}^{2n \times 2n \times 2n}$ onto its sub-tensor 
of smaller size $n \times n \times n$, obtained by first shifting the center of
the tensor $\widetilde{\bf P}_R$ to the point $x_\nu$ and then tracing (windowing) the result 
onto the domain $\Omega_n$:
\[
 {\cal W}_\nu: \widetilde{\bf P}_R \mapsto {\bf P}^{(\nu)}=[p^{(\nu)}_{i_1,i_2,i_3}], \quad 
 p^{(\nu)}_{i_1,i_2,i_3}:= \widetilde{p}_R(i_1+j^{(\nu)}_1,i_2+j^{(\nu)}_2,i_3+j^{(\nu)}_3), 
 \quad i_\ell\in I_\ell.
\]
The point is that the  Tucker rank of the full tensor sum ${\bf P}_0$
increases almost proportionally to the number $N$ of particles in the system, 
see Figure \ref{fig:SVD_Full_Sum}, representing singular values of the side matrix in the canonical
tensor ${\bf P}_0$. On the other hand, the canonical rank of the tensor ${\bf P}_0$ shows up 
the pessimistic bound $\leq R\, N$. 

To overcome this difficulty,
in what follows, we consider the global tensor decomposition of only the  
"long-range part" in the tensor ${\bf P}_0$, defined by
\begin{equation}\label{eqn:Long-Range_Sum} 
 {\bf P}_l = \sum_{\nu=1}^{N} {z_\nu} \, {\cal W}_\nu (\widetilde{\mathbf{P}}_{R_l})=
 \sum_{\nu=1}^{N} {z_\nu} \, {\cal W}_\nu 
 (\sum\limits_{k\in {\cal K}_l} \widetilde{\bf p}^{(1)}_k \otimes \widetilde{\bf p}^{(2)}_k 
 \otimes \widetilde{\bf p}^{(3)}_k).
\end{equation}
The initial canonical rank of the tensor ${\bf P}_l$ equals to $R_l\, N$, and, again, 
it may increase dramatically for a large number of particles $N$. 
Since by construction the tensor ${\bf P}_l$ approximates 
rather smooth function on the domain $\Omega$, one may expect that the large initial rank
can be reduced considerably to some value $R_\ast$ that remains almost independent of $N$. 
The same beneficial property can be expected for the Tucker rank of ${\bf P}_l$. 
The principal ingredient of our tensor approach is the rank 
reduction in the initial canonical sum ${\bf P}_l$
by application of the multigrid accelerated canonical-to-Tucker transform \cite{khor-ml-2009}.

\begin{figure}[htb]
\centering
\includegraphics[width=7.2cm]{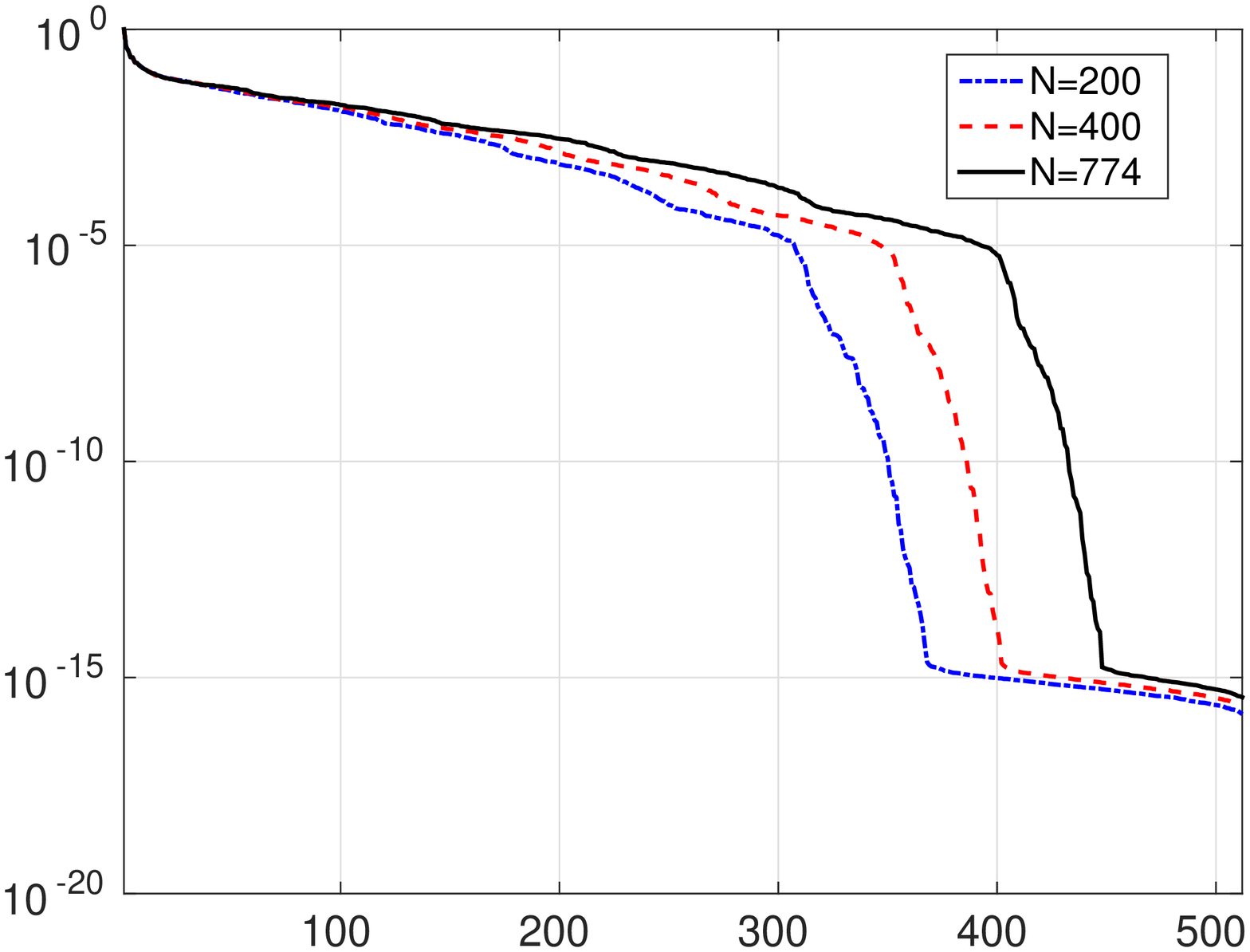} 
\includegraphics[width=7.2cm]{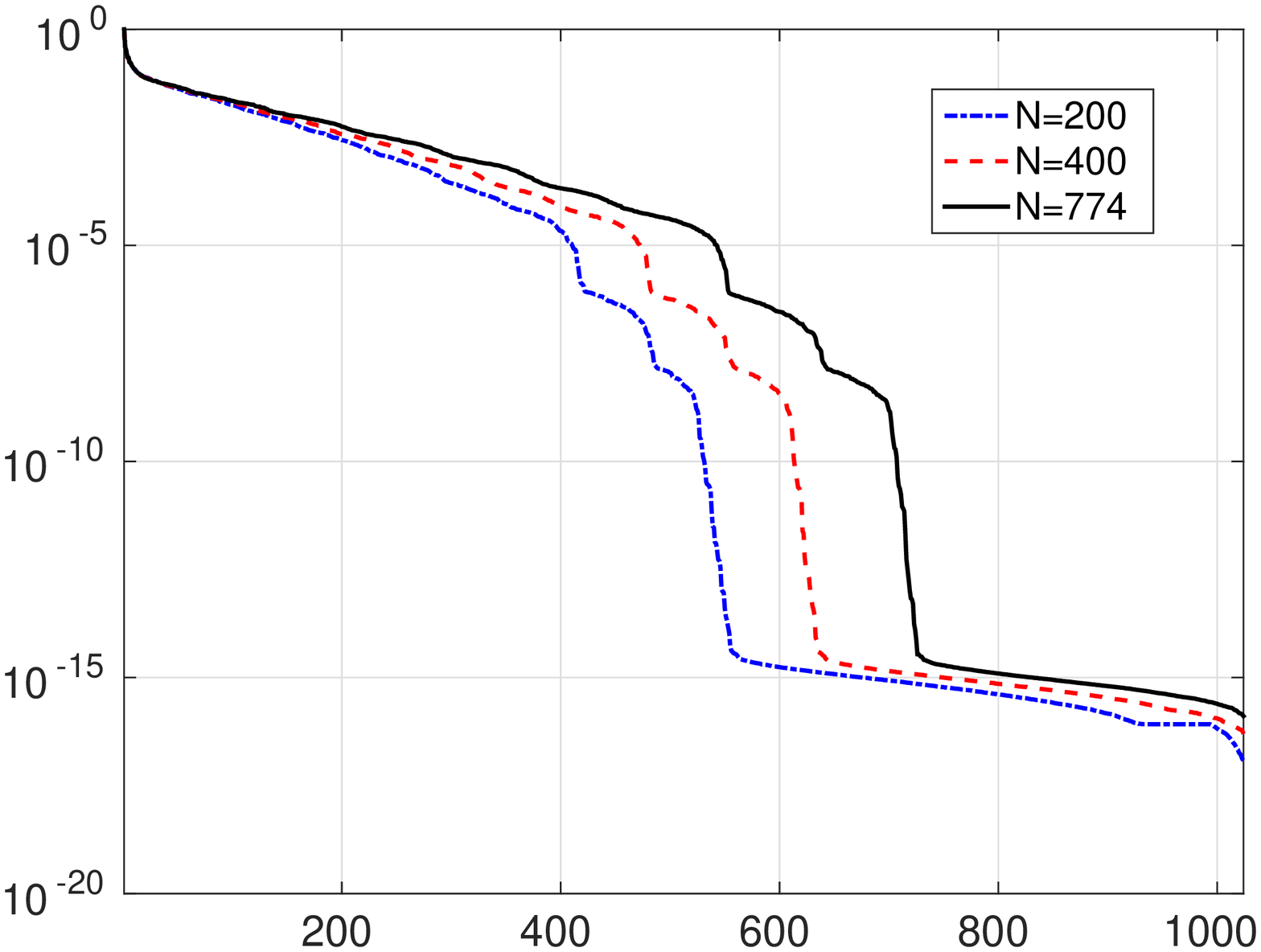}
\caption{\emph{\small  Mode-$1$ singular values of the side matrix in the full potential 
sum vs. the number of particles $N=200,400,774$ and grid-size $n$: 
$n=512$ (left), $n=1024$ (right).}}
\label{fig:SVD_Full_Sum}
\end{figure}

To simplify the exposition, we suppose that the tensor entries in ${\bf P}_l$ are computed by 
collocation of Gaussian sums at the centers of the grid-cells. This provides the representation
which is very close to that obtained by (\ref{eqn:sinc_general}).

The following theorem proves the important result justifying the efficiency of range-separated
formats applied to a class of radial basis functions $p(r)$: the Tucker $\varepsilon$-rank of the 
long-range part in accumulated sum of potentials computed in the bounding box $\Omega=[-b,b]^3$
remains almost uniformly bounded in the number of particles $N$ 
(but depends on the size $b$ of the domain).

\begin{theorem}\label{thm:Rank_LongRange}
Let the long-range part ${\bf P}_l$ in the total interaction potential, see (\ref{eqn:Long-Range_Sum}),
correspond to the choice of splitting parameter 
in (\ref{eqn:ChoiceR_l}) with $M=O(\log^2\varepsilon)$.
Then the total $\varepsilon$-rank ${\bf r}_0$ of the Tucker approximation to the canonical tensor sum ${\bf P}_l$
is bounded by
\[
 |{\bf r}_0|:=rank_{Tuck}({\bf P}_l)=C\, b \,\log^{3/2} (|\log (\varepsilon/N)|),
\] 
where the constant $C$ does not depend on the number of particles $N$.
\end{theorem}
\begin{proof} 
We consider the Gaussian in normalized form 
$G_p(x)=\mathrm{e}^{-\frac{x^2}{2p^2}}$ so that the relation  
$\mathrm{e}^{-t_k^2{x^2}}  =\mathrm{e}^{-\frac{x^2}{2p^2}} $ holds, 
i.e. we set, see (\ref{eqn:hM}),
$$
t_k=\frac{1}{\sqrt{2}p_k}, \quad \mbox{with}\;\;
t_k= k \mathfrak{h}_M, \quad k=0,1,...,M,
$$ 
where $\mathfrak{h}_M=C_0 \log M/M$.
Now criterion (B) in (\ref{eqn:Split_crit_L2}) on the bound of the $L^1$-norm reads
$$
a_k \int_{a}^{\infty} \mathrm{e}^{-\frac{x^2}{2p_k^2}} \leq \frac{\varepsilon}{2} <1, 
\quad a_k=\mathfrak{h}_M.
$$

Now, we sketch the proof to the following steps.  (A)  We represent all shifted Gaussian 
functions, contributing to the total sum, in the fixed set of basis functions 
by using truncated Fourier series. 
 (B)  We prove that, on the "long-range" index set $k\in {\cal T}_l$, 
the parameter $p_k$ remains uniformly bounded in $N$ from below, 
implying the uniform bound on the number of terms in the $\varepsilon$-truncated Fourier series. 
 (C)  The summation of functions presented in the fixed Fourier basis set 
does not enlarge the Tucker rank, but only effects the Tucker core.
The dependence on size of computational domain $b$ remains in the explicit form.

Specifically, let us consider the rank-$1$ term  in the splitting (\ref{eqn:Split_Tens}) with maximal 
index $k\in {\cal T}_l$. Taking into account the asymptotic choice $M=\log^2\varepsilon$, 
see (\ref{sinc_conv}), where $\varepsilon>0$ is the accuracy of the sinc-quadrature,
the relation (\ref{eqn:ChoiceR_l}) implies
\begin{equation} \label{eqn:Choice_t_k}
 \max_{k\in {\cal T}_l}  t_k = R_l \mathfrak{h}_M = \frac{M}{2} C_0 \log(M)/M 
 \approx \log(M) =2 \log(|\log(\varepsilon)|).
\end{equation}
Now we consider the Fourier transform of the univariate
Gaussian on $[-b,b]$,
$$
G_p(x)= \mathrm{e}^{-\frac{x^2}{2p^2}} = \sum\limits_{m=0}^M \alpha_m \cos\left(\dfrac{\pi m x}{b}\right) 
 + \eta, \quad \mbox{with}\quad 
 |\eta| = \left|\sum\limits_{m=M+1}^{\infty} \alpha_m \cos\left(\dfrac{\pi m x}{b}\right)\right| < \varepsilon,
 $$
 where
 $$
 \alpha_m = \dfrac{\int\limits_{-b}^b \mathrm{e}^{-\frac{x^2}{2p^2}} 
 \cos\left(\dfrac{\pi m x}{b}\right) dx}{|C_m|^2},
 \quad \mbox{with}\quad |C_m|^2=\int_{-b}^b \cos^2\left(\dfrac{\pi m x}{b} \right) dx
 =\left\{
   \begin{array}{ll}
    2b, & \mbox{if}~m=0, \\
    b, & \mbox{otherwise.}
   \end{array} \right.
 $$
Following arguments in \cite{DKhOs-parabolic1-2012} one obtains
$$
 \alpha_m = \left(p\mathrm{e}^{-\frac{\pi^2 m^2 p^2}{2a^2}} -\xi_m \right) / |C_m|^2, 
 \quad \mbox{where}\quad  0<\xi_m<\varepsilon.
  $$
Truncation of the coefficients $\alpha_m$ at $m=m_0$ such that 
 $\alpha_{m_0}\leq \varepsilon$, leads to the bound
  $$
 m_0 \geq \frac{\sqrt{2}}{\pi} \frac{b}{p} \log^{0.5} \left(\frac{p}{(1+|C_M|^2) \varepsilon} \right) 
 = \frac{\sqrt{2}}{\pi} \frac{b}{p} \log^{0.5} \left(\frac{p}{1+b}\frac{1}{\varepsilon}\right).
 $$
 On the other hand (\ref{eqn:Choice_t_k}) implies 
 \[ 
  1/p_k \leq c \log(|\log\varepsilon|), \quad k\in {\cal T}_l, 
  \quad \mbox{i.e.} \quad 1/p_{R_l}\approx \log(|\log\varepsilon|),
 \]
that ensures the estimate on $m_0$,
\begin{equation} \label{eqn:Choice_m0}
 m_0 =O(b \, \log^{3/2}(|\log\varepsilon|)).
\end{equation}
Now following \cite{VeBoKh:Ewald:14}, we represent the Fourier transform of the shifted Gaussians
by
\[
 G_p(x-x_\nu) = \sum\limits_{m=0}^M \alpha_m \cos\left(\dfrac{\pi m (x-x_\nu)}{b}\right)+
 \eta_\nu,\quad 
 |\eta_\nu |  < \varepsilon,
\]
which requires only the double number of trigonometric terms compared with the single Gaussian
analyzed above. To compensate the possible increase in $|\sum_\nu \eta_\nu|$, we refine 
$\varepsilon \mapsto \varepsilon/N$.
These  estimates also apply to all Gaussian functions presented in the long-range sum
since for $k\in {\cal T}_l$ they have larger values of $p_k$ than $p_{R_l}$. Indeed, 
in view of (\ref{eqn:ChoiceR_l}) the number of summands in the long-range part 
is of the order $R_l=M/2=O(\log^2\varepsilon)$. 
Combining these arguments with (\ref{eqn:Choice_m0}) proves the resulting estimate.  
\end{proof}

Figure \ref{fig:Fourier_Short_Long} illustrates fast decay of the Fourier coefficients
for the "long-range" discrete Gaussians sampled on $n$-point grid (left) and  
a slow decay of Fourier coefficients for the "short-range" Gaussians (right).
In the latter case, almost all coefficients remain essential, 
resulting in the full rank decomposition. The grid size is chosen as $n=1024$.

\begin{figure}[htb]
\centering
\includegraphics[width=7.2cm]{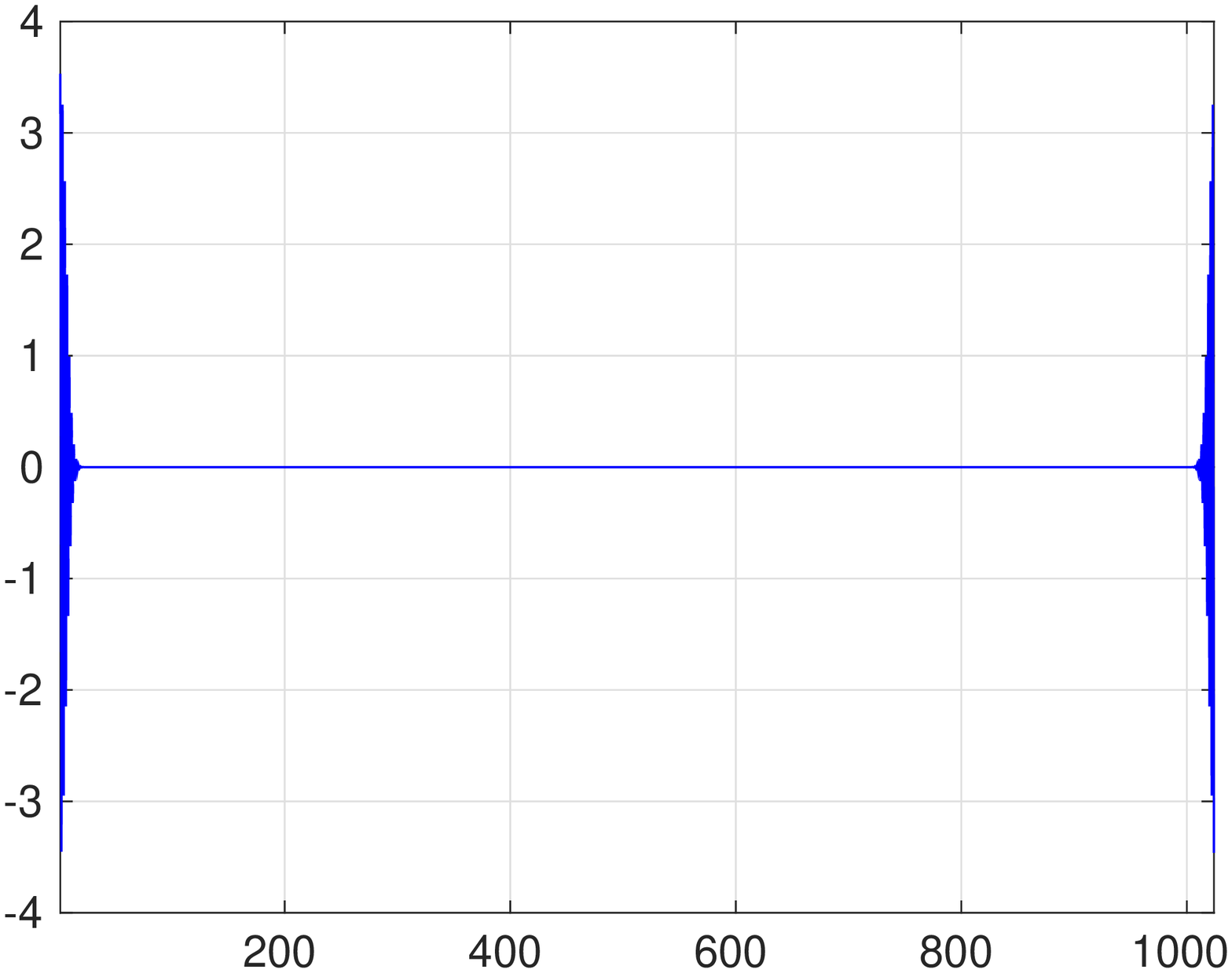} 
\includegraphics[width=7.2cm]{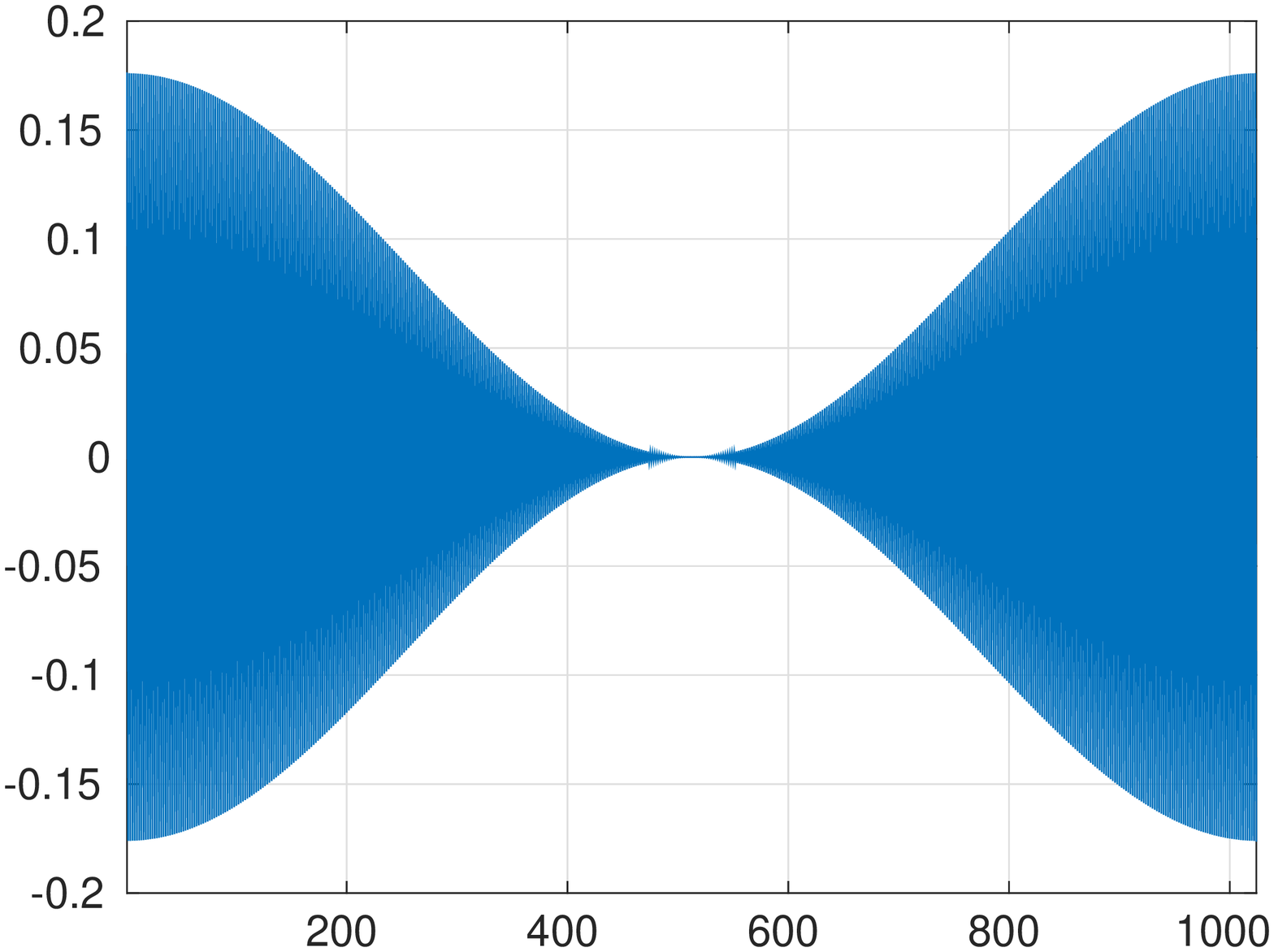}
\caption{\emph{\small Fourier coefficients of the long- (left) and short-range (right)
discrete Gaussians.}}
\label{fig:Fourier_Short_Long}
\end{figure}

\begin{remark}\label{rem:RankTuck_N0}
 Notice that for fixed $\sigma>0$ the $\sigma$-separability of the point distributions 
 (see Definition \ref{def:QU_Distrib})  implies that the volume size
 of the computational box $[-b,b]^3$ should increase proportionally to the number of particles $N$,
 i.e., $b=O(N^{1/3})$. Hence, Theorem \ref{thm:Rank_LongRange} indicates that 
 since $r_l=O(b)$ the number of entries in
 the Tucker core of size $r_1 \times r_2 \times r_3$ can be estimated by $C N$. This asymptotic 
 cost remains of the same order in $N$ as that for the short-range part in the potential sum.
\end{remark}

Figure \ref{fig:SVD_Long_Sum}, left, illustrates that the singular values of side matrices
(i.e. bounds on the Tucker rank)
for the long-range part (by choosing $R_l=12$)  exhibit fast exponential decay with a 
rate independent of the number of particles $N=214,405$, and $754$ 
(cf. Figure \ref{fig:SVD_Full_Sum}). 
Figure \ref{fig:SVD_Long_Sum}, right, zooms into the first $50$ singular values which are almost 
identical for the different values of $N$. The fast decay in these singular values guarantees  the 
low-rank RHOSVD-based Tucker decomposition of the long-range part in the potential 
sum (see Appendix).

\begin{figure}[htb]
\centering
\includegraphics[width=7.0cm]{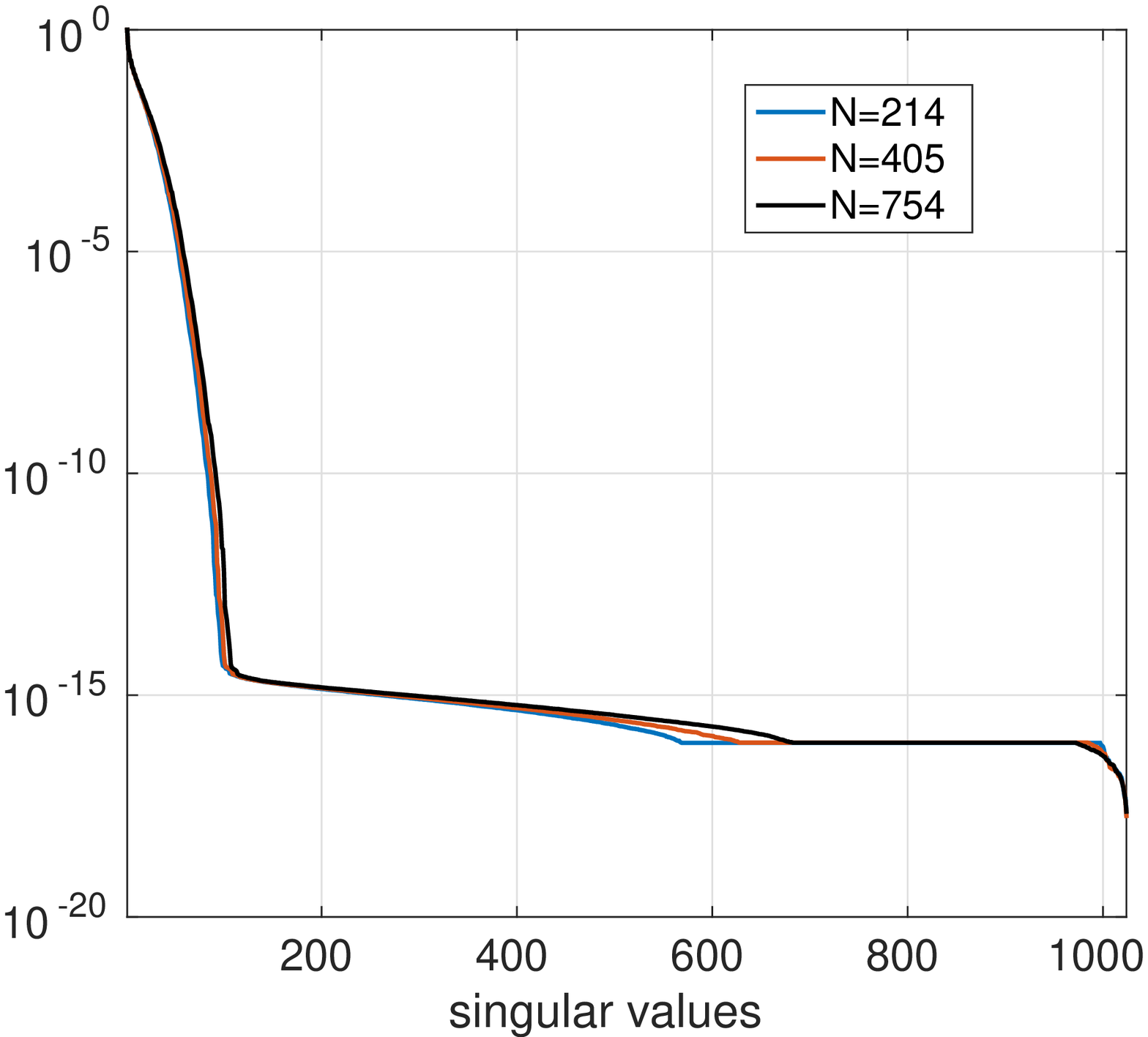} 
\includegraphics[width=7.0cm]{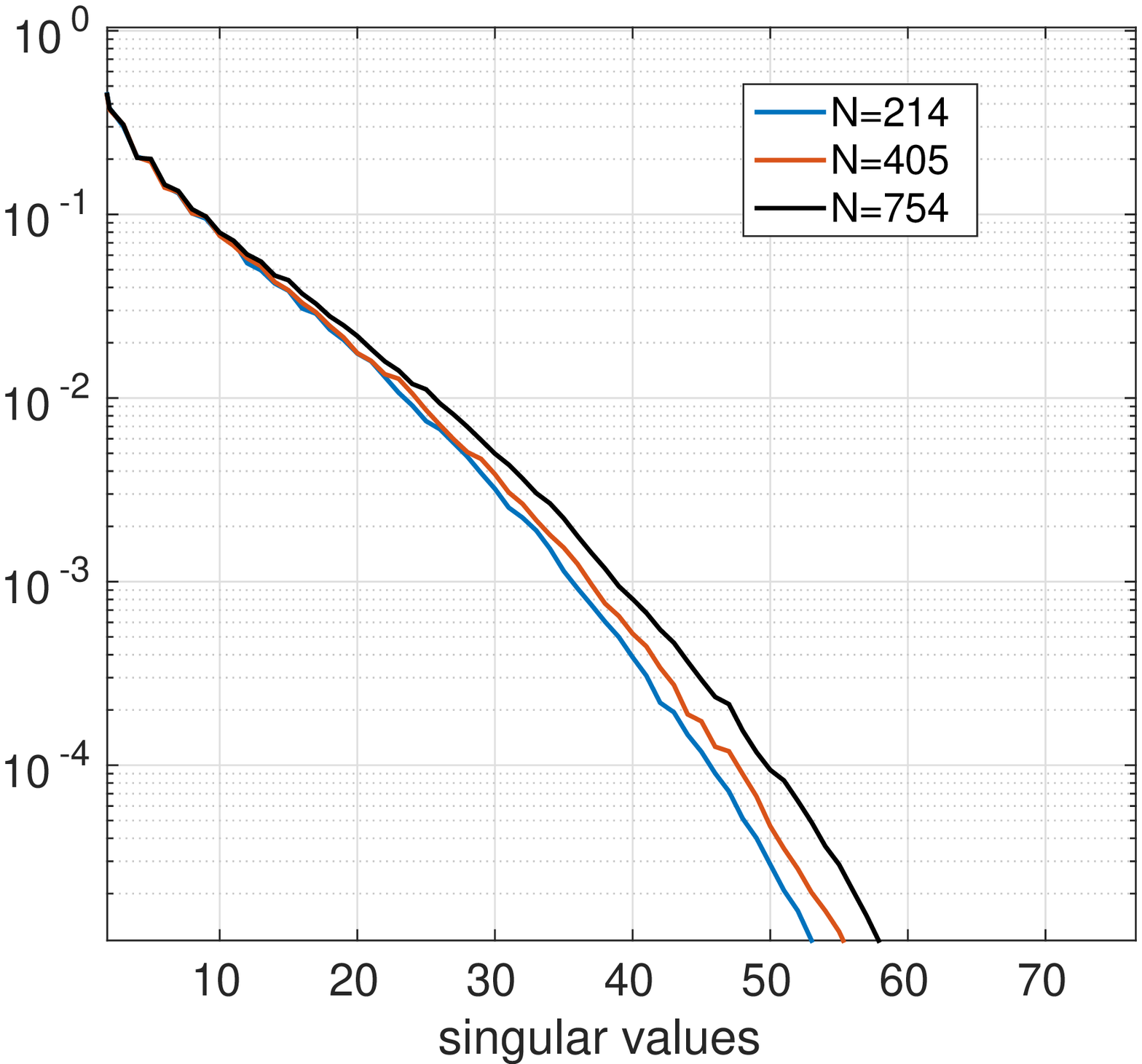}
\caption{\emph{\small Mode-$1$ singular values of side matrices for the long range part 
($R_l=12$) in the total potential vs. the number of particles $N$.}}
\label{fig:SVD_Long_Sum}
\end{figure}
\begin{table}[tbh]
\begin{center}\footnotesize
\begin{tabular}
[c]{|r|r|r|r|r|r|r| }%
\hline
  & $N$      &  $200$     & $400$     & $782$    & $1728$  & $4096$ \\
\hline   
 $R_\ell/R_s $  & Ranks full can. &  4200 & 8400 & 16422 & 32288 & 86016  \\
 \hline  \hline  
9/12   & Ranks long range &  1800 & 3600 & 7038 & 15552 & 36864  \\
 \cline{2-7}
 & RS-Tucker ranks  & 21,16,18 & 22,19,23  & 24,22,24 & 23,24,24 & 24,24,24  \\
 \cline{2-7} 
 & RS-canonical rank & 254 & 292  & 362 & 207 & 243  \\
 \cline{2-7}
   \hline\hline
 10/11 & Ranks long range &  2000 & 4000 & 7820 & 17280 & 40960  \\
 \cline{2-7}
 & RS-Tucker ranks  & 30,22,25 & 32,25,33  & 36,32,34 & 25,25,25 & 29,29,29  \\
 \cline{2-7} 
 & RS-canonical rank & 476 & 579  & 768 & 286 & 426  \\
 \cline{2-7}
    \hline  
 \end{tabular}
\caption{\emph{\small Tucker ranks and the RS canonical rank of the  multiparticle potential sum  
  versus the number of particles $N$ for varying parameters $R_\ell$ and $R_s$ for 
  the grid size $n^3=1024^3$.} 
}
\label{Tab:Low_Tuck_ranks}
\end{center}
\end{table}

Table \ref{Tab:Low_Tuck_ranks} shows the Tucker ranks of sums of long-range  ingredients
in the electrostatic potentials for the $N$-particle clusters. 
The Newton kernel is generated on the grid with $n^3=1024^3$
in the computational box of size $b^3=40^3$\AA{}, with accuracy
$\varepsilon= 10^{-4}$ and canonical rank $21$.
Particle clusters with 200, 400 and 782 atoms are taken as a part of protein-like multiparticle system.
The clusters of size $1728$ and $4096$ correspond to the lattice structures 
of sizes $12\times 12\times 12$ and 
$16\times 16\times 16$,  with randomly generated charges. The line  ``RS-canonical rank''
shows the resulting rank after the canonical-to-Tucker and Tucker-to-canonical transform, with
$\varepsilon_{C2T}=4\cdot 10^{-5}$ and $\varepsilon_{T2C}=4\cdot10^{-6}$.

\begin{table}[tbh]
\begin{center}\footnotesize
\begin{tabular}
[c]{|l|r|r|r|r|r|r|r|}%
\hline
$N$ /$R_l$ & $8$ & $9$   & $10$      & $11$ & $12$      & $13$ \\
\hline
200  & 10,10,11 & 13,12,12 & 18,15,16 & 23,19,21 & 32,24,27 & 42,30,34 \\
\hline
400  & 11,10,11 & 14,13,14 & 19,16,20 & 26,21,26 & 35,27,36 & 47,34,47 \\
\hline
782  & 11,11,12 & 15,14,15 & 20,18,20 & 28,26,27 & 39,35,37 & 52,46,50 \\
\hline
 \end{tabular}
\caption{\emph{\small Tucker ranks ${\bf r}=(r_1,r_2,r_3)$ for the long-range parts of
 $N$-particle potentials. }}
\label{Tab:Tucker_ranks}
\end{center}
\end{table} 
Table \ref{Tab:Tucker_ranks} represents the Tucker ranks ${\bf r}=(r_1,r_2,r_3)$
for the long-range parts of $N$-particle potentials.
The reference  Newton kernel is approximated on a 3D grid of size $2048^3$,
with the rank $R=29$ and with accuracy $\varepsilon_{\cal N}=10^{-5}$. 
Here the  Tucker tensor is computed with the stopping criteria $\varepsilon_{T2C}=10^{-5}$
in the ALS iteration. It can be seen that for fixed $R_l$ the Tucker ranks increase very moderately
in the system size $N$.

Fig. \ref{fig:SVD_R_L} demonstrates the decay in singular values of the 
side matrices (i.e., upper bound on the Tucker rank) in the
canonical tensor representing potential sums of long-range parts for different 
$R_l=10, 11$, and $12$.
 
 \begin{figure}[htb]
\centering
\includegraphics[width=7.2cm]{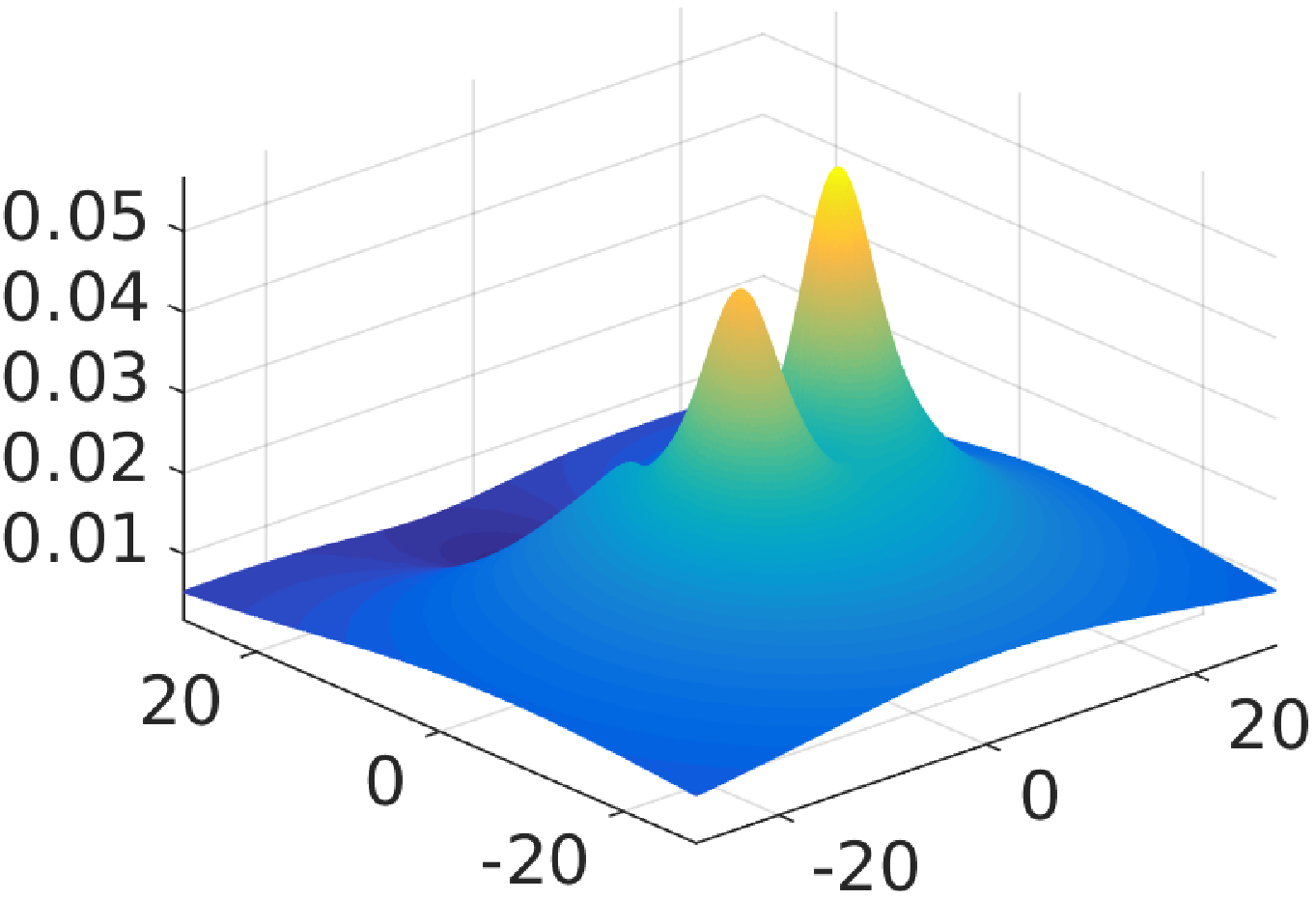} 
\includegraphics[width=7.2cm]{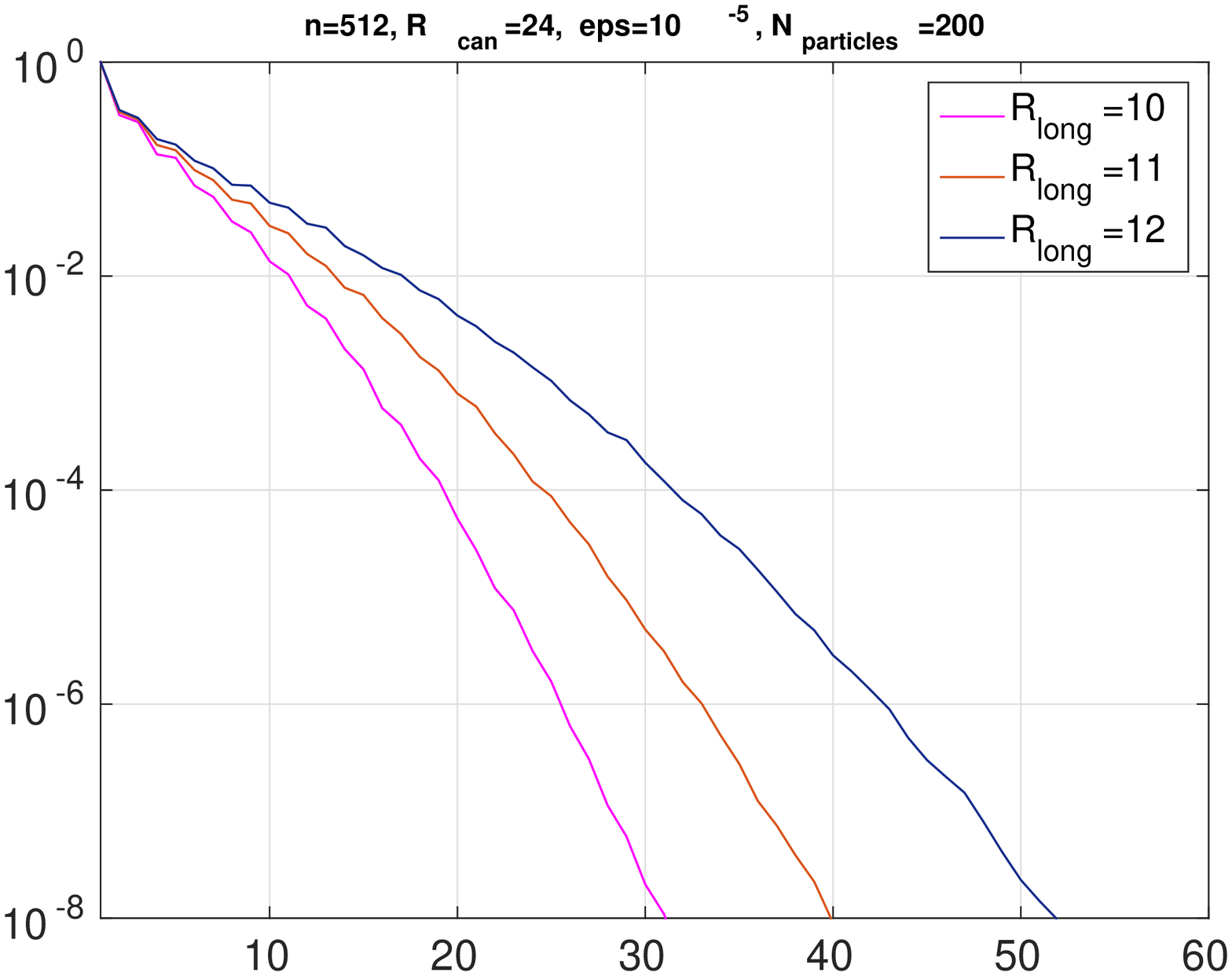}
\caption{\emph{\small Example of potential surface at level $z=0$ (left) for a sum of $N_0=200$ particles 
computed using only their long-range parts with $R_l=12$. Decay in singular values of the side matrices for the
canonical tensor representing sums of long-range parts for  $R_l=10, 11$, and $12$.}}
\label{fig:SVD_R_L}
\end{figure}

The proof of Theorem \ref{thm:Rank_LongRange} indicates that the Tucker directional vectors,
living on large $n^{\otimes d}$ spatial grids, 
are represented in the uniform Fourier basis with a small number of terms.
Hence, following the arguments in \cite{DKhOs-parabolic1-2012} and \cite {VeBoKh:Ewald:14},
we are able to apply
the low-rank QTT tensor approximation \cite{KhQuant:09} to these long vectors
(see \cite{osel-2d2d-2010} for the case of matrices).
The QTT tensor compression makes it possible to reduce the representation complexity
of the long-range part in an RS tensor to the logarithmic scale in the 
univariate grid size, $O(\log n)$. This topic will be addressed in a forthcoming paper.

\subsection{Range-separated canonical and Tucker tensor formats}\label{ssec:Cumulat_CanTens}

We recall that the general canonical tensor is specified by a $R$-term sum 
of arbitrary rank-$1$ tensors as in (\ref{eqn:CP_form}), which makes it difficult to
perform approximation process and multilinear algebra in such tensor format for large values of $R$.
In  applications to many-particle modeling the initial rank parameter $R$ is proportional to
the (large) number of particles $N$ with pre-factor about $30$, 
while the weights $z_k$ can be rather arbitrary\footnote{
Notice that the sub-class of the so-called
\emph{orthogonal canonical tensors} \cite{Kolda:01} allows stable rank reduction,
but suffers from the poor approximation capacity. Another class of "monotone" tensors
providing stable canonical representation is specified by
all positive canonical vectors, see \cite{khor-ml-2009,VeKhor_NLLA:15} for definition, 
which is the case in decomposition of the elliptic Green's kernels.
Both classes of tensors do not suite problems like (\ref{eqn:PotSum}).}.

The idea on how to get rid of the "curse of ranks", that is the critical 
bottleneck in application of tensor methods to the problems like (\ref{eqn:PotSum}),
is suggested by results in Theorem \ref{thm:Rank_LongRange} on the
almost uniform bound (in the number of particles $N$)
of the Tucker rank for the long-range part in a multi-particle potential. 
Thanks to this beneficial property, we are able to introduce the 
new range-separated (RS) tensor formats based on 
the aggregated composition of the global low-rank canonical/Tucker tensor and the locally 
supported canonical tensors living on non-intersecting index sub-sets embedded into 
the large corporate multi-index set ${\cal I}=I_1 \times \ldots \times I_d$, 
$I_\ell=\{1,\ldots,n\}$.
Such a parametrization attempts to represent the large multi-dimensional arrays with 
a storage cost linearly proportional to the number of cumulated inclusions (sub-tensors).


The structure of the \emph{range-separated canonical/Tucker tensor formats} is  
specified by a combination of the local-global low parametric representations,
which provide good approximation features in application to 
the problems of grid-based representation to many-particle interaction potentials 
with multiple singularities.

The following Definition \ref{def:ComCanF} introduces the description
of a sum of short range potentials having the local (up to some threshold) 
non-intersecting supports.


\begin{definition}\label{def:ComCanF} (Cumulated canonical tensors).
Given the index set ${\cal I}$, a set of multi-indices (sources) 
${\cal J}=\{{\bf j}^{(\nu)}:=(j^{(\nu)}_1,j^{(\nu)}_2,\ldots,j^{(\nu)}_d)\}$, $\nu=1,\ldots,N$, 
$j^{(\nu)}_\ell \in I_\ell$,
and the width index parameter $\gamma \in \mathbb{N}$ such that the $\gamma$-vicinity of 
each point ${\bf j}^{(\nu)}\in {\cal J}$, i.e. 
${\cal J}^{(\nu)}_\gamma:=\{{\bf j}: |{\bf j} -{\bf j}^{(\nu)} |\leq \gamma\}$ 
does not intersect all others
\[
 {\cal J}^{(\nu)}_\gamma \cap {\cal J}^{(\nu')}_\gamma = \varnothing, \quad \nu \neq \nu'.
\]
A rank-$R_0$ cumulated canonical tensor ${\bf U}$, associated with 
${\cal J}$ and width parameter $\gamma$, is defined as a set of tensors which 
can be represented in form 
\begin{equation}\label{eqn:Cum_CP_form}
   {\bf U} = {\sum}_{\nu =1}^{N} c_\nu {\bf U}_\nu, \quad \mbox{with} \quad rank({\bf U}_\nu)\leq R_0,
\end{equation}
where the rank-$R_0$ canonical tensors ${\bf U}_\nu=[u_{\bf j}]$ are vanishing beyond the $\gamma$-vicinity of 
${\bf j}^{(\nu)}$, 
\begin{equation}\label{eqn:Cum_CP_Support}
  u_{\bf j} = 0\quad \mbox{for} \quad  {\bf j} \subset {\cal I}\setminus 
  {\cal J}^{(\nu)}_\gamma, \quad \nu =1,\ldots,N.
\end{equation}
\end{definition}

Given the particular point distribution,
the effective support of the localized sub-tensors should be of the size close 
to the parameter $\sigma_\ast$,
appearing in Def. \ref{def:QU_Distrib}, that introduces the $\sigma_\ast$-separable point distributions 
characterized by the separation parameter $\sigma_\ast>0$.
In this case, we use the relation $\sigma_\ast \approx \gamma h$, where $h=2b/n$ is the mesh size 
of the computational $(n \times \ldots \times n)$-grid.

\begin{figure}[htb]
\centering 
\includegraphics[width=5.2cm]{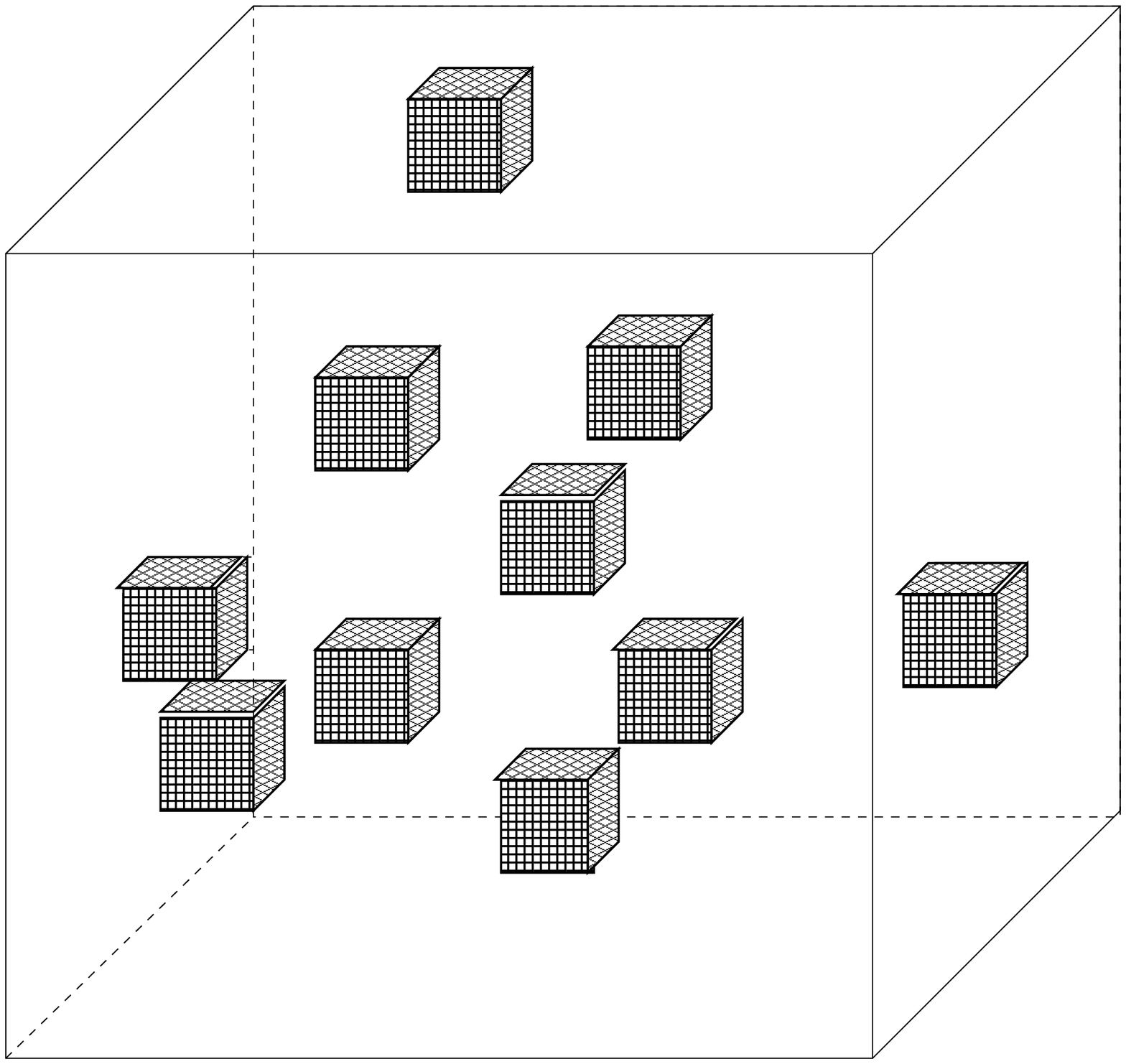} \quad \quad
\includegraphics[width=7.3cm]{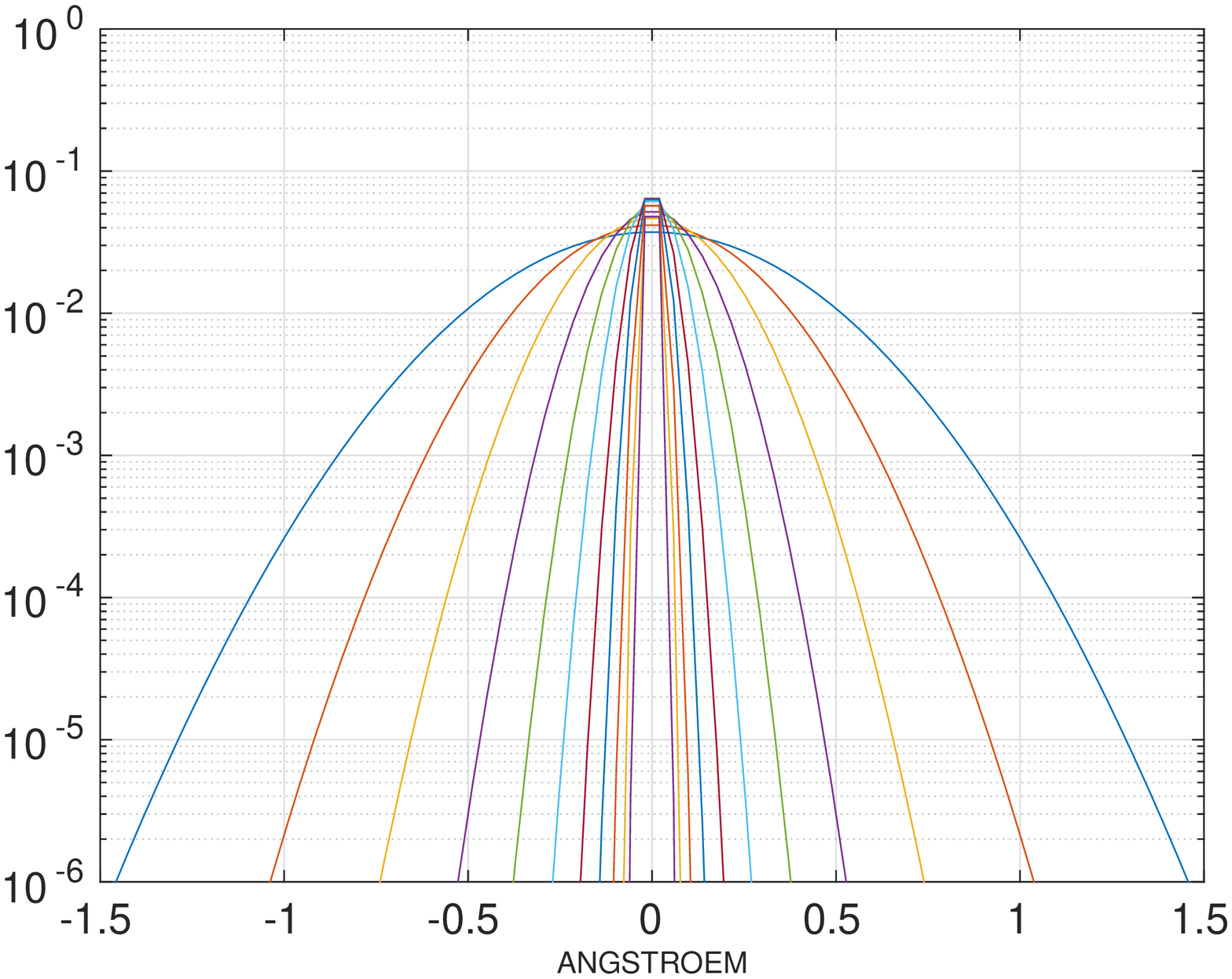}
\caption{\emph{\small Schematic illustration of effective supports of 
the cumulated canonical tensor (left); 
Short-range canonical vectors for $k=1,\ldots,11$, represented in logarithmic scale (right).}}
\label{fig:Support_Short}
\end{figure}
Figure \ref{fig:Support_Short} (left) illustrates the effective supports 
of a cumulated canonical tensor in the non-overlapping case, while Figure
\ref{fig:Support_Short} (right) presents the  supports 
for first $11$ short-range canonical vectors (selected from rank-$24$ 
reference canonical tensor ${\bf P}_R$), which allows to make a choice for 
the parameter $\gamma$ in separation criteria.

The separation criteria in Def. \ref{def:ComCanF} leads to rather "aggressive" strategy
for selection of the short-range part ${\bf P}_{R_s}$ in the reference canonical tensor
${\bf P}_{R}$ at the benefit of easy implementation of the cumulated
canonical tensor (non-overlapping case). However, in some cases this may lead
to overestimation of the Tucker/canonical rank in the long-range tensor component.
To relax the criteria in Def. \ref{def:ComCanF}, we propose the "soft" strategy  that
allows to include a few (i.e., $O(1)$ for large $N$) neighboring particles into the 
local vicinity ${\cal J}^{(\nu)}_\gamma$ of the source point $x_\nu$, which can be achieved
by increasing the overlap parameter $\gamma>0$. This allows to control the bound on the 
rank parameter of the long-range tensor almost uniformly in the system size $N$.


\begin{example}\label{exm:soft_separ}
For example, assume that the separation distance is equal to $\sigma_\ast=0.8$\AA{}, 
corresponding to the example in Fig. \ref{fig:1024_dist}, right, and the 
computational threshold is given $\varepsilon=10^{-4}$.
Then we find from Fig. \ref{fig:Support_Short} (right) that the "aggressive"
criteria in Def. \ref{def:ComCanF} leads to the choice $R_s=10$, since the value of the
canonical vector with $k=11$ at point $x=\sigma_\ast $ is about $10^{-3}$.
Hence, in order to control the required rank parameter $R_l$, we have to extend the 
overlap area to larger parameter $\sigma_\ast$ and, 
hence, to larger $\gamma$. This will lead to a small $O(1)$-overlap between supports of the 
short range tensor components, but without asymptotic increase in the total complexity.
\end{example}

In the following, we distinguish a special subclass of \emph{uniform CCT tensors}.
\begin{definition}\label{def:CCT_Uniform}
(Uniform CCT tensors).
 A CCT tensor in (\ref{eqn:Cum_CP_form}) is called uniform if all 
 components ${\bf U}_\nu$  are generated by a single rank-$R_0$ tensor 
 ${\bf U}_0= {\sum}_{m =1}^{R_0} \mu_m \hat{\bf u}_m^{(1)}  
 \otimes \cdots \otimes \hat{\bf u}_m^{(d)}$, 
 such that ${{\bf U}_\nu}|_{{\cal J}^{(\nu)}_\delta}=  {\bf U}_0$.
\end{definition}

Now we are in a position to define the range separated canonical and Tucker tensor formats
in $\mathbb{R}^{n_1\times ... \times n_d}$.
The range-separated canonical format is defined as follows.

\begin{definition}\label{Def:RS-Can_format} (RS-canonical tensors).\\
 The RS-canonical tensor format specifies the class of $d$-tensors 
 ${\bf A}  \in \mathbb{R}^{n_1\times \cdots \times n_d}$
 which can be represented as a sum 
of a rank-${R}$ canonical tensor ${\bf U}\in \mathbb{R}^{n_1\times ... \times n_d}$ and 
a (uniform) cumulated canonical tensor generated by ${\bf U}_0$ with $\mbox{rank}({\bf U}_0)\leq R_0$
as in Definition \ref{def:CCT_Uniform} (or more generally in Definition \ref{def:ComCanF}),
\begin{equation}\label{eqn:RS_Can}
 {\bf A} =  {\sum}_{k =1}^{R} \xi_k {\bf u}_k^{(1)}  \otimes \cdots \otimes {\bf u}_k^{(d)} +
 {\sum}_{\nu =1}^{N} c_\nu {\bf U}_\nu, 
\end{equation}
where $\mbox{diam}(\mbox{supp}{\bf U}_\nu)\leq 2 \gamma$ in the index size.
\end{definition}

For a given grid-point ${\bf i}\in {\cal I}=I_1 \times ... \times I_d$, we define 
the set of indices 
$$
{\cal L}({\bf i}):=\{\nu\in \{1,\ldots ,N\}: {\bf i}\in \mbox{supp}{\bf U}_\nu \},
$$
which label all short-range tensors ${\bf U}_\nu$ including the grid-point ${\bf i}$
within its effective support.
\begin{lemma}\label{lem:Property_RST_can}
The storage cost of RS-canonical tensor is estimated by
$$
\mbox{stor}({\bf A})\leq d R n + (d+1)N + d R_0 \gamma.
$$
Given ${\bf i}\in {\cal I}$, denote by $\overline{\bf u}^{(\ell)}_{i_\ell}$ the  
row-vector  with index $i_\ell$ in the side matrix $U^{(\ell)}\in \mathbb{R}^{n_\ell \times R}$,
and let $\xi=(\xi_1,\ldots,\xi_d)$.
Then the ${\bf i}$-th entry of the RS-canonical tensor ${\bf A}=[a_{\bf i}]$ can be calculated 
as a sum of long- and short-range contributions by
\[
 a_{\bf i}= \left(\odot_{\ell=1}^d \overline{\bf u}^{(\ell)}_{i_\ell} \right) \xi^T +
 \sum_{\nu \in {\cal L}({\bf i})} c_\nu {\bf U}_\nu({\bf i}),
\]
at the expense  $O(d R + 2 d \gamma R_0)$.
\end{lemma}
\begin{proof}
Definition \ref{Def:RS-Can_format} implies that
each RS-canonical tensor is uniquely defined by the following parametrization:
rank-$R$ canonical tensor ${\bf U}$, the rank-$R_0$ local reference canonical tensor ${\bf U}_0$
with mode-size bounded by $2\gamma$, and list ${\cal J}$ of the coordinates and weights of $N$ 
particles. Hence the storage cost directly follows.
 To justify the representation complexity, we notice that by well-separability assumption
 (see Definition \ref{def:QU_Distrib}), we have $\#{\cal L}({\bf i}) =O(1)$
 for all  ${\bf i}\in {\cal I}$. This proves the complexity bounds.
\end{proof}

Now we define the class of RS-Tucker tensors.

\begin{definition}\label{Def:RS-Tucker_format} (RS-Tucker tensors).
 The RS-Tucker tensor format specifies the class of $d$-tensors 
 ${\bf A}  \in \mathbb{R}^{n_1\times ... \times n_d}$
 which can be represented as a sum of a rank-${\bf r}$ Tucker tensor ${\bf V}$ and 
a (uniform) cumulated canonical tensor generated by ${\bf U}_0$ with $\mbox{rank}({\bf U}_0)\leq R_0$
as in Definition \ref{def:CCT_Uniform}
(or more generally in Definition \ref{def:ComCanF}),
\begin{equation}\label{eqn:RS_Tucker}
 {\bf A} =  \boldsymbol{\beta} \times_1 V^{(1)} \times_2 V^{(2)}\ldots \times_d V^{(d)} +
 {\sum}_{\nu =1}^{N} c_\nu {\bf U}_\nu, 
 \end{equation}
where the tensor ${\bf U}_\nu$, $\nu=1,\ldots,N$, has local support, 
i.e. $diam(supp{\bf U}_\nu)\leq 2 \gamma$.
\end{definition}

Similar to Lemma \ref{lem:Property_RST_can} the corresponding statement for the RS-Tucker tensors
can be proven.
\begin{lemma}\label{lem:Property_RST_Tuck}
The storage size for RS-Tucker tensor does not exceed 
$$
\mbox{stor}({\bf A})\leq r^d + d r n + (d+1)N + d R_0 \gamma.
$$
Let the $r_\ell$-vector ${\bf v}^{(\ell)}_{i_\ell}$ be the $i_\ell$ row of the matrix $V^{(\ell)}$.
Then the ${\bf i}$-th element of the RS-Tucker tensor ${\bf A}=[a_{\bf i}]$ can be calculated by
\[
 a_{\bf i}=  \boldsymbol{\beta} \times_1 {\bf v}^{(1)}_{i_1l} \times_2 {\bf v}^{(2)}_{i_2}
 \ldots \times_d {\bf v}^{(d)}_{i_d}+
 \sum_{\nu \in {\cal L}({\bf i})} c_\nu {\bf U}_\nu({\bf i})
\]
at the expanse $O(r^d + 2 d \gamma R_0)$.
\end{lemma}
\begin{proof}
In view of Definition \ref{Def:RS-Tucker_format}
each RS-Tucker tensor is uniquely defined by the following parametrization:
the rank-${\bf r}=(r_1, ... ,r_d)$ Tucker tensor ${\bf V}\in \mathbb{R}^{n_1\times ... \times n_d}$, 
the rank-$R_0$ local reference canonical tensor ${\bf U}_0$ with 
$\mbox{diam}(\mbox{supp}{\bf U}_0)\leq 2 \gamma$,
list ${\cal J}$ of the coordinates of 
$N$ centers of particles, $\{x_\nu\}$, and $N$ weights $\{c_\nu\}$. This proves the complexity bounds.
\end{proof}

Fig. \ref{fig:TuckVect_Long} represents the first seven Tucker vectors 
of the long range part in the RS tensor  for $50$ and $500$ particles. In both cases 
we observe the very smooth shape of the orthogonal functions which 
do not demonstrate the tendency to higher oscillations for larger number of particles. 
\begin{figure}[htb]
\centering 
  \includegraphics[width=7.2cm]{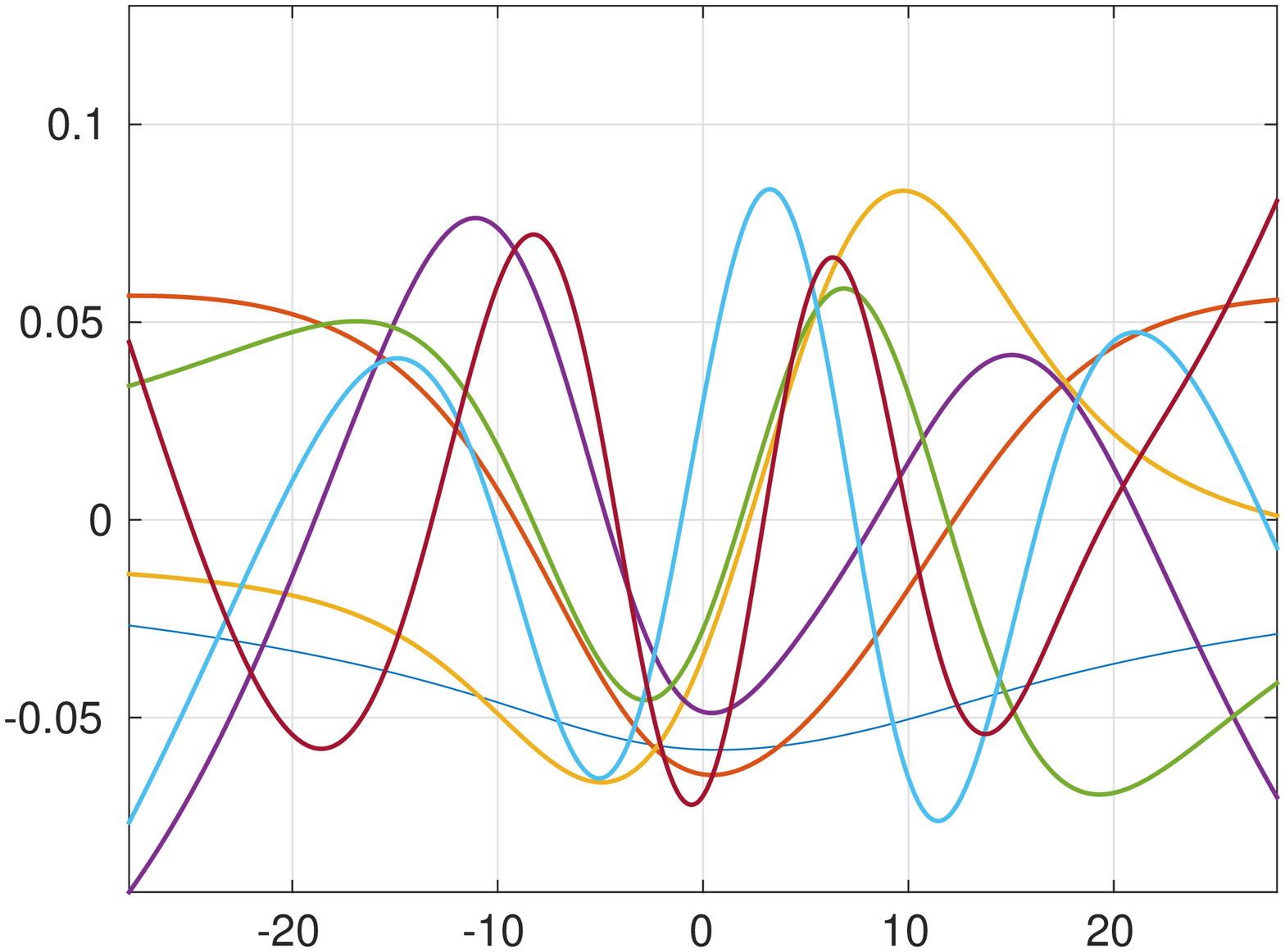}\quad
  \includegraphics[width=7.2cm]{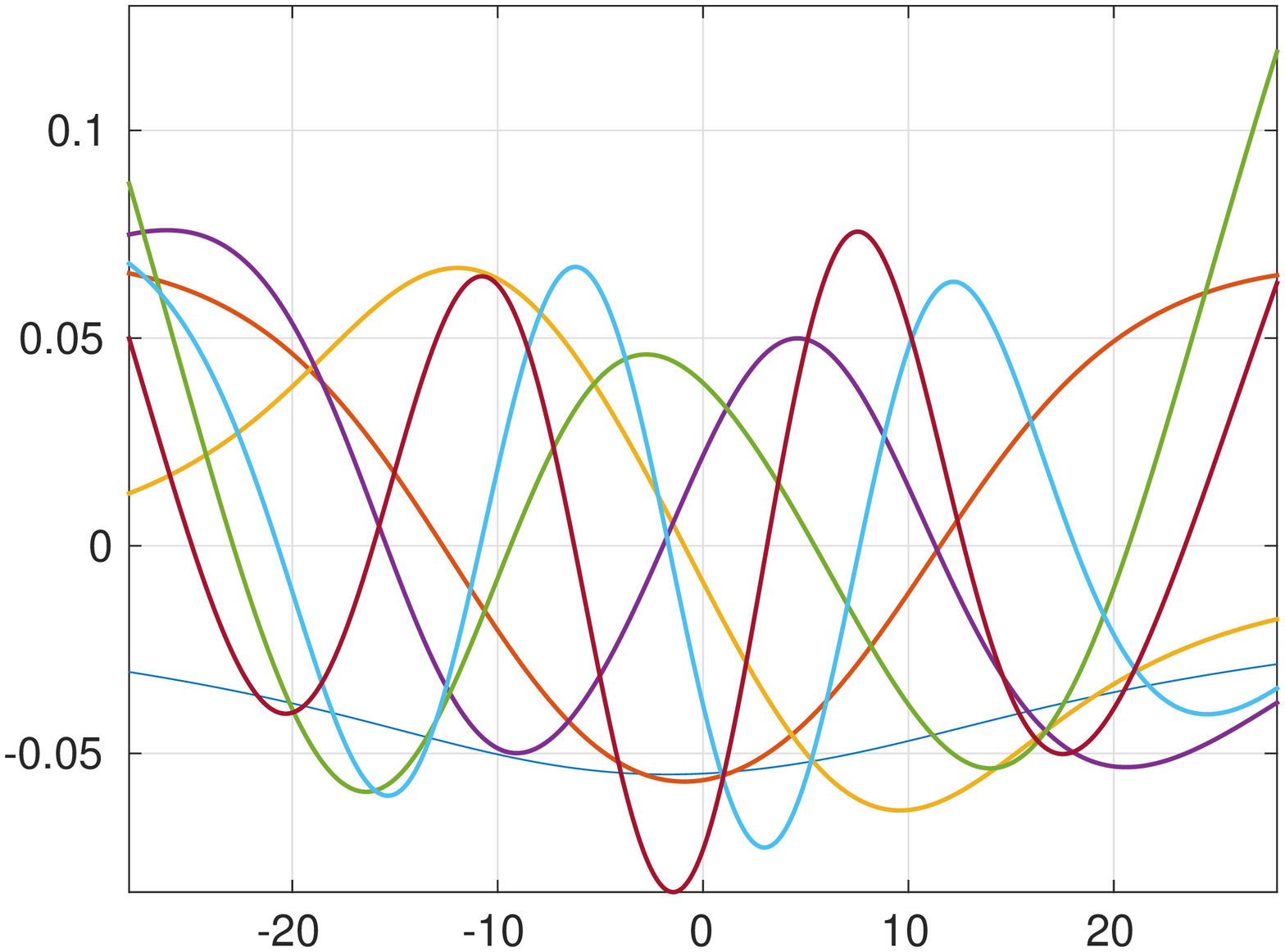} 
\caption{\emph{\small Tucker vectors of the long range  part in the RS tensor for $50$ particles (left) 
and $500$ particles (right).}}
\label{fig:TuckVect_Long}
\end{figure}

The main computational benefits of the new range-separated canonical/Tucker tensor formats
are explained by the important uniform bounds on the Tucker-rank of the long-range part
in the large sum of interaction potentials (see Theorem \ref{thm:Rank_LongRange} 
and numerics in \S\ref{ssec:Sum_LongRange}). Moreover, we have the low storage cost for 
RS-canonical/Tucker tensors, cheap representation of each entry in an RS-tensor,
possibility for simple implementation of
multi-linear algebra on these tensors (see \S\ref{ssec:Sum_ShortRange}), which opens 
the opportunities for various applications.

The total rank of the sum of canonical tensors in ${\bf U}$, see (\ref{eqn:Cum_CP_form}), may 
become large for larger $N$ since the pessimistic bound $rank({\bf U})\leq N R_0$. However,
cumulated canonical tensors (CCT) have two beneficial features which are particularly 
useful in the low-rank tensor representation of large potential sums.
\begin{proposition}\label{prop:Cum_CP_Feature}
(Properties of CCT tensors).\\
(A) The local rank of a CCT tensor ${\bf U}$ is bounded by $R_0$: 
$$
rank_{loc}({\bf U}):=\max_\nu rank({\bf U}_\nu)\leq R_0.
$$
(B) Local components in the CCT tensor (\ref{eqn:Cum_CP_form}) are ``block orthogonal" in the sense
\begin{equation}\label{eqn:Cum_CP_Orthog}
 \langle {\bf U}_\nu, {\bf U}_{\nu'} \rangle =0, \quad \forall \nu \neq \nu'.
\end{equation}
(C) There holds $\| {\bf U}  \|= {\sum}_{\nu =1}^{N} c_\nu \|{\bf U}_\nu\|$.
\end{proposition}
\begin{proof}
Properties (A) and (B) simply follow by definition of CCT, while (C) is a 
direct consequence of (B).
\end{proof}

If $R_0=1$, i.e. ${\bf U}$ is the usual rank-$N$ canonical tensor, then the property (B) 
in Proposition \ref{prop:Cum_CP_Feature} leads to the definition of orthogonal 
canonical tensors in \cite{Kolda:01}; hence, in case $R_0 > 1$, we arrive at the generalization 
further called the \emph{block orthogonal canonical tensors}.

The bound $R'=rank({\bf U})\leq N R_0$ indicates that the direct summation 
in (\ref{eqn:Cum_CP_form}) in the canonical/Tucker formats may lead to 
practically non-tractable representations.
However, the block orthogonality property in Proposition \ref{prop:Cum_CP_Feature}, (B)
allows to apply the stable RHOSVD approximation for the rank optimization, 
see Section \ref{app:C2T_HOSVD}.
The stability of RHOSVD in the case of orthogonal canonical tensors was analyzed in 
\cite{khor-ml-2009,VeKhor_NLLA:15}. In what follows, we prove the stability of 
such tensor approximation applied to CCT representations.

\begin{lemma}\label{lem:RHOSVD_CCT}
Let the local canonical tensors be stable, i.e. 
$\sum_{m=1}^{R_0} \mu_{m}^2 \leq C \|{\bf U}_\nu\|^2$ 
(see Def. \ref{def:CCT_Uniform}).
Then the rank-${\bf r}$ RHOSVD-Tucker approximation ${\bf U}_{({\bf r})}^0$ to the CCT, 
${\bf U}$, provides the stable error bound
\[
\| {\bf U} - {\bf U}_{({\bf r})}^0\| \leq C  \sum\limits_{\ell=1}^3
(\sum\limits_{k=r_\ell +1}^{\min(n,R')} \sigma_{\ell,k}^2)^{1/2}\|{\bf U}\|,
\]
where $\sigma_{\ell,k}$ denote the singular values of the side matrices $U^{(\ell)}$, see 
(\ref{eqn:SVD_SideMatr}).
\end{lemma}
\begin{proof}
 We apply the general error estimate for RHOSVD approximation \cite{khor-ml-2009} 
\[
 \|{\bf U}- {\bf U}_{({\bf r})}^0\| \leq C  \sum\limits_{\ell=1}^3
(\sum\limits_{k=r_\ell +1}^{\min(n,R')} \sigma_{\ell,k}^2)^{1/2} 
(\sum\limits_{\nu =1}^{N}  \sum\limits_{m =1}^{R_0} c_\nu^2 \mu_{m}^2)^{1/2},
\] 
and then take into account the property (C), Proposition \ref{prop:Cum_CP_Feature} to obtain 
 \[
 \sum\limits_{\nu =1}^{N}  \sum\limits_{m =1}^{R_0} c_\nu^2 \mu_{m}^2
 =\sum\limits_{\nu =1}^{N}c_\nu^2 \sum\limits_{m =1}^{R_0}  \mu_{m}^2
    \leq C \sum\limits_{\nu =1}^{N} c_\nu^2 \|{\bf U}_\nu\|^2  = C\| {\bf U}\|^2,
 \]
which completes the proof.
\end{proof}

We comment that the stability assumption in Lemma \ref{lem:RHOSVD_CCT} is satisfied 
for the constructive canonical tensor approximation to the Newton and other types of Green's kernels 
obtained by sinc-quadrature based representations, where all skeleton vectors
are non-negative and monotone.

\begin{figure}[htb]
\centering
\includegraphics[width=7.2cm]{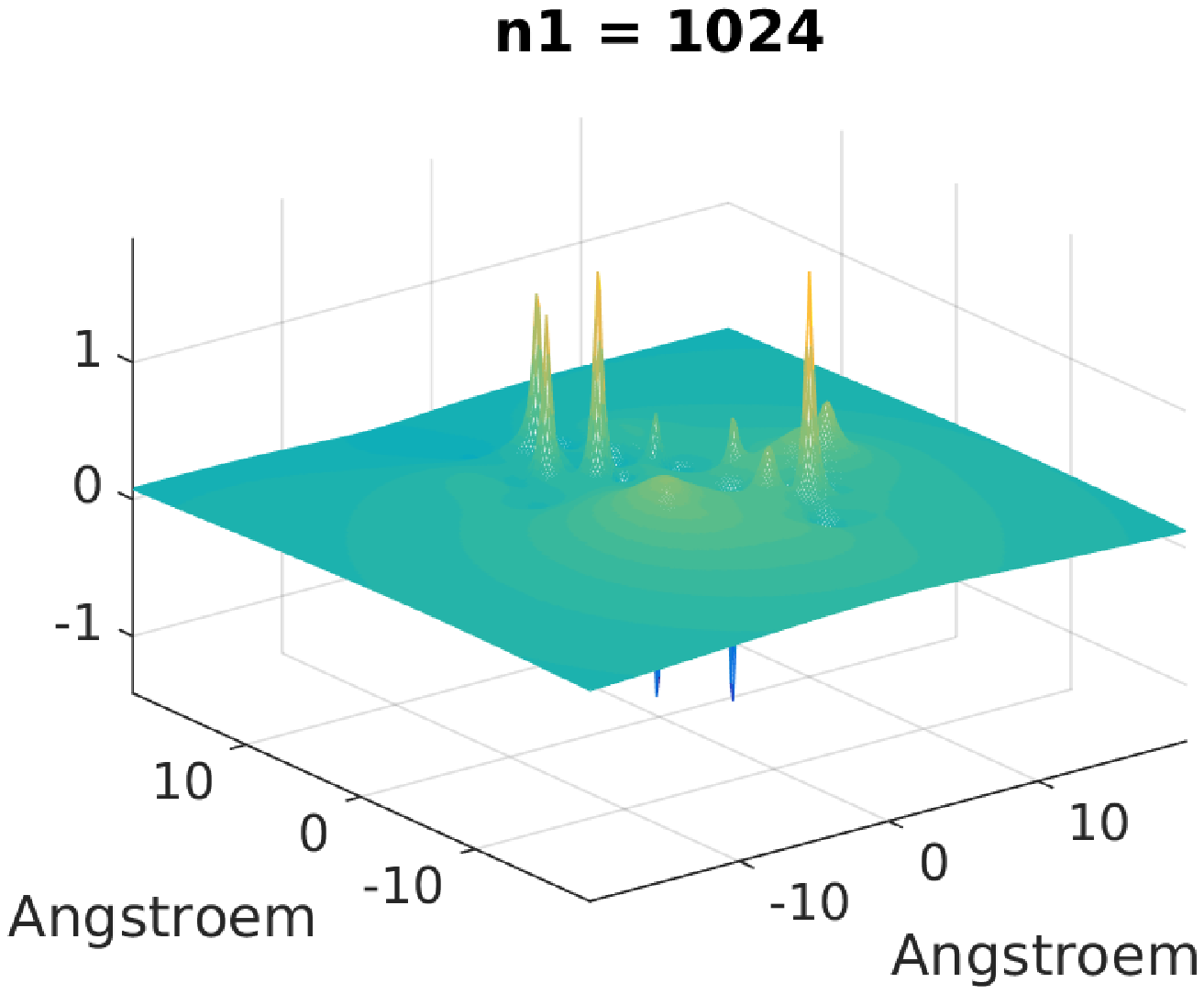} \quad
\includegraphics[width=7.2cm]{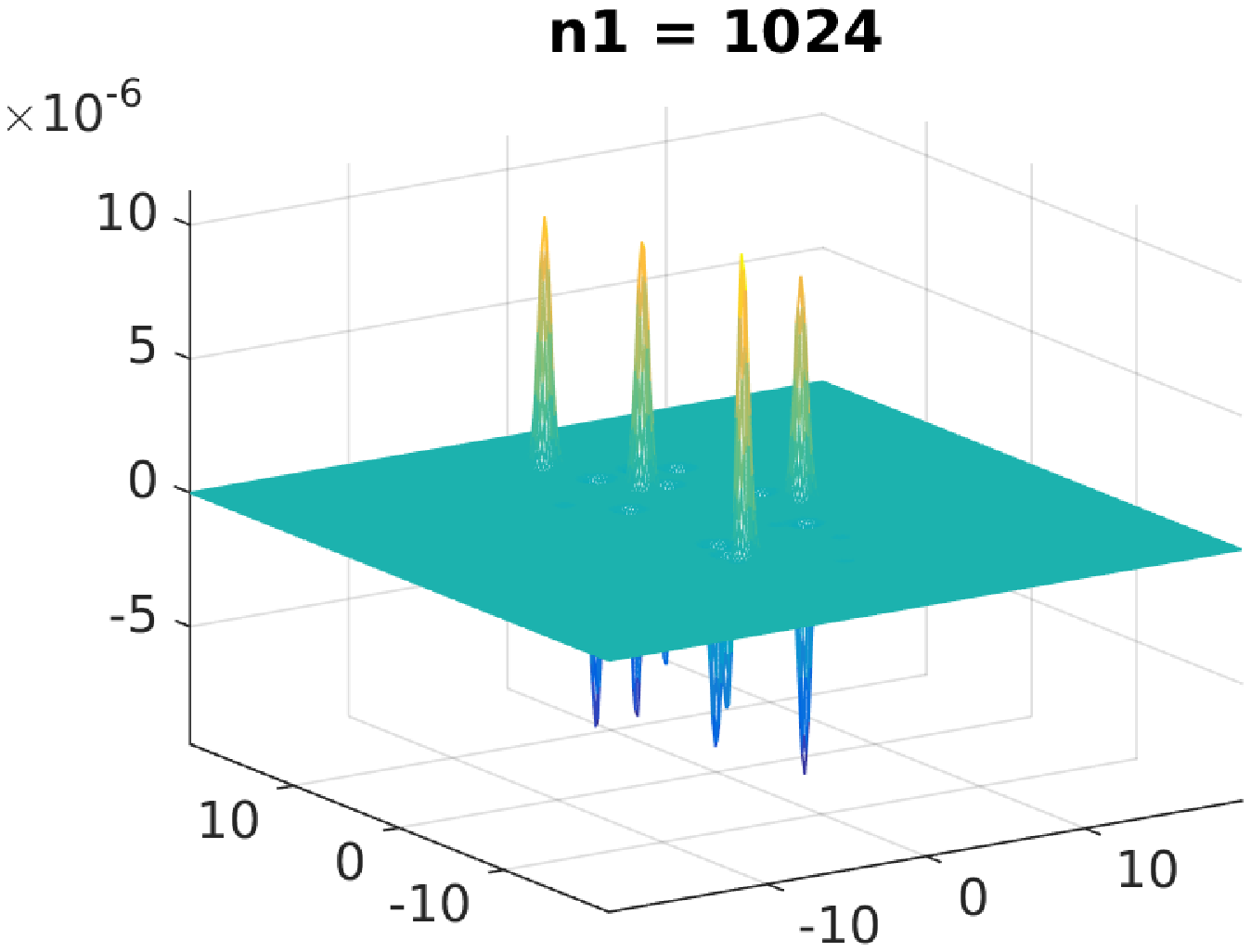}
\includegraphics[width=7.2cm]{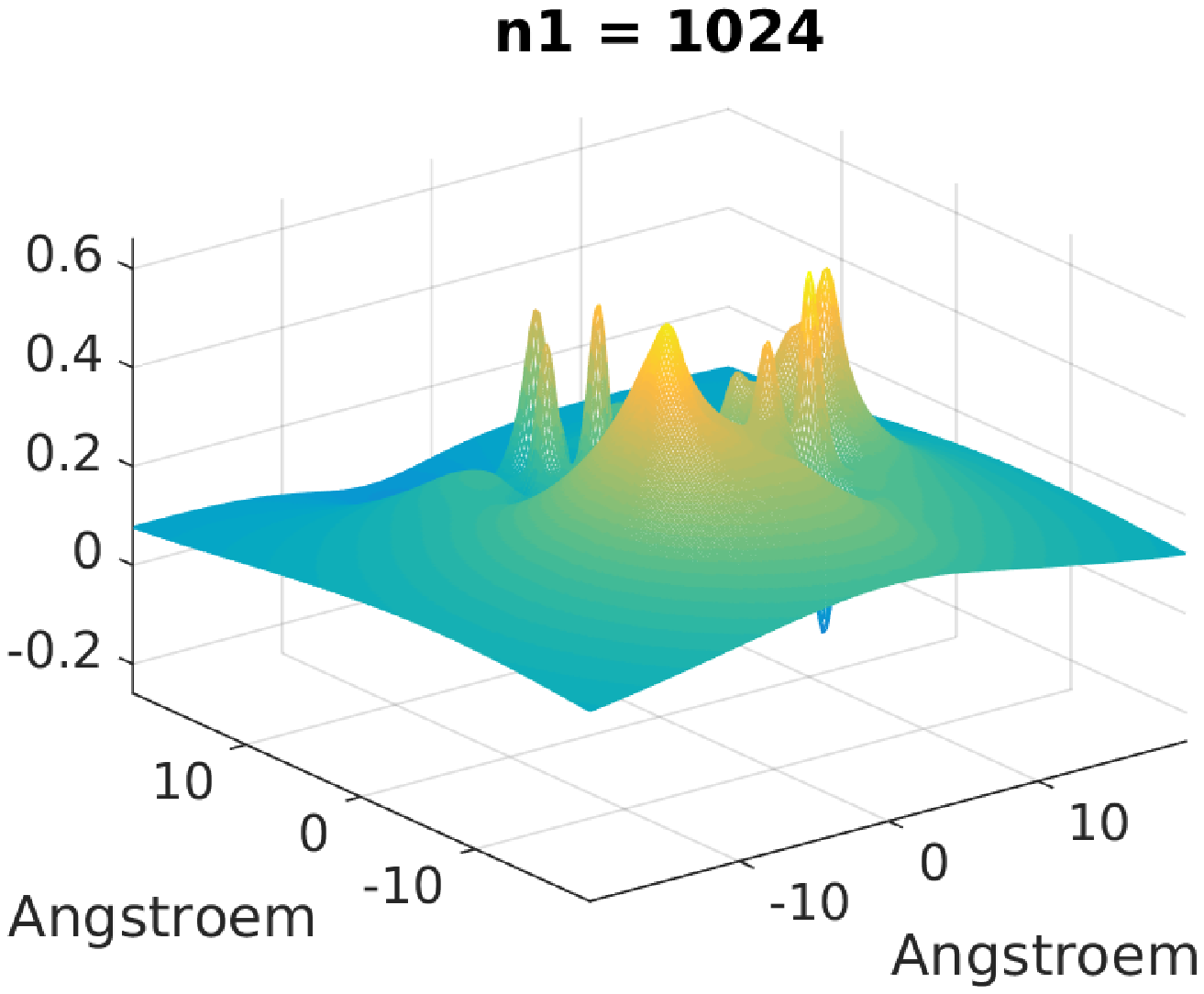} \quad
\includegraphics[width=7.2cm]{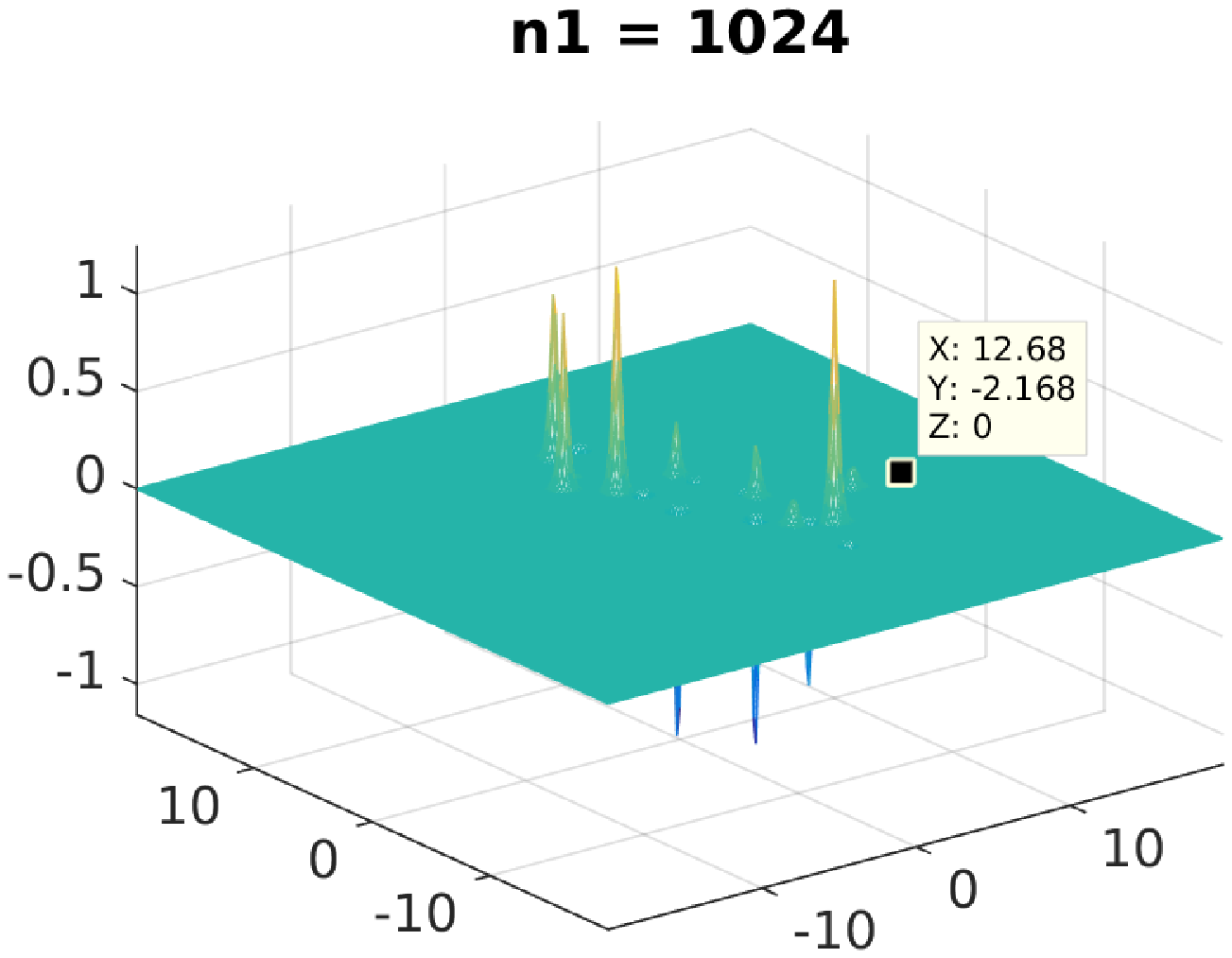}
\caption{\emph{\small Top:  the electrostatic potential sum at a middle plane of a cluster 
with 400 atoms (left), and the error of the RS-canonical approximation (right). 
Bottom:  long-range part of a sum (left), short range part of a sum (right).}}
\label{fig:Pot_Long_Sum}
\end{figure}

Figures  \ref{fig:Pot_Long_Sum} show the accuracy of the RS-canonical tensor 
approximation to the electrostatic potential of 
a cluster of 400 particles at the middle section of the
computational box $[-20,20]^3$ \AA{}, by  using an
$n\times n\times n $ 3D Cartesian grid with $n=1024$, and step size   $h=0.04$\AA{}.
The top-left figure shows the surface of the potential at the
level $z=0$, while the  top-right figure shows the absolute error of the RS approximation
with the ranks $R_l=15$, $R_s=11$, and the separation distance $\sigma^\ast = 1.5$.
Bottom figures visualize the long-range (left) and short-range (right) parts 
of the RS-tensor, representing the potential sum.

%

\begin{remark}\label{rem:RS-TT}
It is worth to note that in the case of higher dimensions, say, for $d >3$,
the local canonical tensors can be combined with the global tensor train (TT) format \cite{Osel_TT:11}
such that the simple canonical-to-TT transform can be applied.
In this case the RS-TT format can be introduced as  a set of tensor represented as a sum of CCT term and
the global TT-tensor.
The complexity and structural analysis is completely similar to the case of RS-Canonical and RS-Tucker formats.
\end{remark}

We complete this section by the short outlook of algebraic operations on the RS tensors.

\subsection{Algebraic operations on the RS canonical/Tucker tensors}
\label{ssec:Sum_ShortRange} 

Multilinear algebraic operations in the format of RS-canonical/Tucker tensor parametrization
can be implemented by using 1D vector operations
applied  to both localized and global tensor components.
In particular, the following operations on RS canonical/Tucker tensors can be 
realized efficiently: 
(a) storage of a tensor;
(b) real space representation on a fine rectangular grid;
(c) summation of many-particle interaction potentials represented on the fine tensor grid;
(d) computation of scalar products;
(e) computation of gradients and forces.


Estimates on the storage complexity for the RS-canonical and RS-Tucker formats
were presented in Lemmas \ref{lem:Property_RST_can} and \ref{lem:Property_RST_Tuck}.
Items (b) and (c) have been already addressed in the previous part.
Calculation of the scalar product of two RS-canonical tensors in the form (\ref{eqn:RS_Can}), 
defined on the same set ${\cal S}$ of particle centers, can be reduced to the standard
calculation of the cross scalar products between all elementary canonical tensors 
presented in (\ref{eqn:RS_Can}). The numerical cost can be estimated by 
$O(\frac{1}{2} R(R-1)d n + 2 \gamma R R_0 N)$.


\section{Sketch of  possible  applications}\label{sec:SepSum_Applic}

The RS tensor formats can be gainfully applied in computational problems
including functions with multiple local singularities or cusps, Green kernels with 
essentially non-local behavior, as well as in various approximation problems  treated  by means 
of radial basis functions.
In what follows, we present the brief explanations on how the RS tensor
representations can be utilized to some computationally extensive problems: 
grid representation of multi-dimensional scattered data, 
interaction energy of charged many-particle system, computation of gradients and 
forces for many-particle potentials, 
construction of approximate boundary/interface conditions  in the Poisson-Boltzmann equation
describing the electrostatic potential of proteins.
The detailed analysis of these examples will be the topic for forthcoming papers.

\subsection{Multi-dimensional data modeling}\label{ssec:DataFit_Applic}

In this section we briefly describe the model reduction approach to the problem 
of multi-dimensional data fitting based on the RS tensor approximation.
The problems of multi-dimensional scattered data modeling and data mining are known
to lead to computationally intensive simulations. 
We refer to \cite{Buhmann-book-2003,Iske-book-2004,BeGaMo:09,FoFl:15,Fornberg-book-2015,HeRoSt-book-2016}
concerning  the discussion of most commonly used approaches 
for the approximating of multi-dimensional data and partial differential equations by using the radial basis functions.

The mathematical problems in scattered data modeling are concerned with the approximation of
multi-variate function $f:\mathbb{R}^d \to \mathbb{R}$ ($d \geq 2$) by using
samples given at a certain finite set  ${\cal X}=\{x_1,\ldots,x_N\}\subset \mathbb{R}^d$
of pairwise distinct points, see e.g. \cite{Buhmann-book-2003}. 
The function $f$ may describe the surface of a solid body, the solution of a PDE, 
many-body potential field, multi-parametric characteristics of physical systems or
some other multi-dimensional data. 

In the particular problem setting 
one may be interested in recovering $f$ from a given sampling vector
$f_{|{\cal X}}=(f(x_1),\ldots,f(x_N))\in \mathbb{R}^N$. 
One of the traditional ways to tackle this problem is based on the construction 
of a suitable functional interpolant $P_N:\mathbb{R}^d \to \mathbb{R}$ satisfying 
$P_{N|{\cal X}}=f_{|{\cal X}}=:{\bf f}$, i.e.,
\begin{equation}\label{eqn:Interp_col}
 P_N(x_j)=f(x_j), \quad \forall\; 1\leq j \leq N,
\end{equation}
or approximating the sampling vector $f_{|{\cal X}}$ on the set ${\cal X}$ in the least squares sense.
We consider the approach based on using \emph{radial basis functions} 
providing the traditional tools for multivariate scattered data interpolation. 
To that end, the radial basis function (RBF)
interpolation approach deals with a class of interpolants $P_N$ in the form
\begin{equation}\label{eqn:RBF_interp}
 P_N(x) = \sum_{j=1}^N c_j p(\|x-x_j\|) + Q(x), \quad Q \; \mbox{is some smooth function},
\end{equation}
where $p:[0,\infty)\to \mathbb{R}$ is a fixed radial function, and $\|\cdot \|$ is
the Euclidean norm on $\mathbb{R}^d$. To fix the idea, here we consider the particular version 
of (\ref{eqn:RBF_interp}) by setting $Q=0$. 
Notice that the interpolation ansatz  $P_N$ in (\ref{eqn:RBF_interp}) has the same form as the 
multi-particle interaction potential in (\ref{eqn:PotSum}). This observation indicates
that the numerical treatment of various problems based on the use of interpolant $P_N$
can be handled by using the same tools of model reduction via rank-structured RS tensor 
approximation.

The particular choice of RBFs described in
\cite{Buhmann-book-2003,Iske-book-2004} includes functions $p(r)$ in the form 
$$
r^\nu, \quad (1+r^2)^\nu,\; (\nu\in \mathbb{R}),
 \quad \exp(-r^2), \quad r^2 \log (r).
$$
For our tensor based approach, the common feature of all these function classes is
the existence of low-rank tensor approximations to the grid-based discretization of the
RBF $p(\|x\|)=p(x_1,\ldots,x_d)$, $x\in \mathbb{R}^d$, where we set $r=\|x\|$.
We can extend the above examples by
traditional functions commonly used in quantum chemistry, like Coulomb potential $1/r$, 
Slater function $\exp(- \lambda r)$, Yukawa potential $\exp(- \lambda r)/r$, 
as well as to the class of Mat{\"e}rn RBFs, traditionally applied in stochastic 
modeling \cite{MaLiPaRoZa:12}.
Other examples are given by the Lennard-Jones (the Van der Waals) and dipole-dipole interaction potentials, 
\[
p(r)= 4 \epsilon 
\left[\left(\frac{\sigma }{r}\right)^{12} - \left(\frac{\sigma }{r}\right)^{6}\right],\quad \mbox{and} \quad p(r)= \frac{1}{r^3},
\]
respectively,
as well as by the Stokeslet \cite{LinTorn2:12}, specified by the $3 \times 3$ matrix
\[
p(\|x\|)= I/r + ( x x^T ) /r^3\quad  \mbox{for} \quad  x\in \mathbb{R}^3.
\]

In the context of numerical data modeling, we focus on the following computational tasks.
\begin{itemize}
 \item[(A)] For fixed coefficient vector ${\bf c}=(c_1,\ldots,c_N)^T\in \mathbb{R}^N$ find
the efficient representation and storage of the interpolant in (\ref{eqn:RBF_interp}), 
 sampled on fine tensor grid in $\Omega\subset\mathbb{R}^d$, that allows the 
 $O(1)$-fast point evaluation of $P_N$ in the whole volume $\Omega$ and computation of various 
integral-differential operations on that interpolant like, gradients, forces, scalar products,
convolution integrals, etc.
 \item[(B)] Finding the coefficient vector ${\bf c}$  that solves the 
 interpolation problem (\ref{eqn:Interp_col}).
\end{itemize}

We look on the problems (A) and (B)
with the intend to apply the RS tensor representation to the interpolant $P_N(x)$.
The point is that the representation  (\ref{eqn:RBF_interp}) can be viewed as the
many-particle interaction potential (with charges $c_j$) considered in the previous sections.
Hence, the RS tensor approximation can be successfully applied if the $d$-dimensional tensor approximating 
the RBF $p(\|x\|)$, $x\in \mathbb{R}^d$, on tensor grid,
allows the low-rank canonical representation that can be split into the short- and long-range parts.
This can be proven for functions listed above 
(see the example in \S2.2 for the Newton kernel $1/\|x\|$). Notice that the Gaussian 
is already the rank-$1$ separable function.

Problem (A).
To fix the idea, we consider the particular choice of the set ${\cal X}\subset \Omega:=[0,1]^d$, 
which can be represented by using the nearly optimal point sampling. 
The so-called optimal point sets realize the trade off between 
the separation distance 
$q_{\cal X}= \min_{s\in {\cal X}} \min_{x_\nu\in {\cal X}\setminus s} d(x_\nu,s)$,
see  (\ref{eqn:SepDist}), and the fill distance
$h_{{\cal X},\Omega}= \max_{y \in \Omega} d({\cal X},y)$, i.e. solve the 
problem, see \cite{Buhmann-book-2003},
$$
q_{\cal X} / h_{{\cal X},\Omega} \to \max.
$$

We choose the set of points ${\cal X}$ as a subset of the $n^{ \otimes} $ square  
grid $\Omega_h$ with the mesh-size $h=1/(n-1)$,
such that the separation distance satisfies $\sigma_\ast= q_{\cal X}  \geq \alpha h$, $\alpha \geq 1$. 
Here $N\leq n^d$. The square grid $\Omega_h$ realizes an example of the almost optimal 
point set (see the discussion in \cite{Iske-book-2004}).
The construction below also applies to nonuniform rectangular grids.

Now, we are in a position to apply the RS tensor representation to the total 
interpolant $P_N$. Let ${\bf P}_R$ be the $n \times n \times n$ (say, for $d=3$)
rank-$R$ tensor representing the RBF $p(\| \cdot \|)$ which allows the RS splitting 
by (\ref{eqn:Split_Tens}) generating the global RS representation (\ref{eqn:Total_Sum}).
Then $P_N(x)$ can be represented by the tensor ${\bf P}_N$ in the RS-Tucker (\ref{eqn:RS_Tucker})
or RS-canonical (\ref{eqn:RS_Can}) formats. The storage cost scales linear in both $N$ and $n$,
$O(N+ d R_l n)$.
The tensor-based computation of different functionals on ${\bf P}_N$ will be
discussed in the following sections.

Problem (B). The interpolation problem (\ref{eqn:Interp_col})
reduces to solve the linear system of equations for unknown coefficient vector 
${\bf c}=(c_1,\ldots,c_N)^T \in \mathbb{R}^N$,
\begin{equation}\label{eqn:RBF_Syst}
 A_{p,{\cal X}} {\bf c}= {\bf f}, \quad \mbox{where}\quad 
 A_{p,{\cal X}}=[p(\|x_i-x_j\|)]_{1\leq i,j\leq N}\in \mathbb{R}^{N\times N},
\end{equation}
with the symmetric matrix $A_{p,{\cal X}}$. Here, without loss of generality,
we assume that the RBF, $p(\| \cdot \|)$, is continuous.
The solvability conditions for the linear system (\ref{eqn:RBF_Syst}) with the matrix 
$A_{p,{\cal X}}$ are discussed, for example, in \cite{Buhmann-book-2003}.
We consider two principal cases.

Case (A). We assume that the point set ${\cal X}$ coincides with the set of 
grid-points in $\Omega_h$, i.e., $N=n^d$.
Introducing the $d$-tuple multi-index ${\bf i}=(i_1,\ldots,i_d)$ and 
${\bf j}=(j_1,\ldots,j_d)$, we reshape the matrix
$ A_{p,{\cal X}}$ into the tensor form 
$$ 
A_{p,{\cal X}}\mapsto {\bf A}=[a(i_1,j_1,\ldots,i_d,j_d)]\in 
  \bigotimes_{\ell=1}^d \mathbb{R}^{n\times n},
$$
which corresponds to folding of an $N$-vector to a $d$-dimensional $n^{\otimes d}$ tensor.
 This $d$-level Toeplitz matrix is generated by the tensor ${\bf P}_R$ obtained by collocation 
 of the RBF $p(\| \cdot \|)$  on the grid $\Omega_h$.
Splitting the rank-$R$ canonical tensor ${\bf P}_R$ into a sum of short- and long-range terms, 
\[
 {\bf P}_R={\bf P}_{R_s} + {\bf P}_{R_l},\quad \mbox{with}\quad 
 {\bf P}_{R_l}=\sum\limits_{k=1}^{R_l} {\bf p}_k^{(1)}\otimes \cdots \otimes {\bf p}_k^{(d)},
\]
allows to represent the matrix ${\bf A}$ in the RS-canonical form as a sum of low-rank canonical 
tensors ${\bf A}={\bf A}_{R_s} + {\bf A}_{R_l}$. 
Here, the first one corresponds to
the diagonal (nearly diagonal in the case of "soft" separation strategy) matrix 
by assumption on the locality of ${\bf P}_{R_s}$. 
The second matrix takes the form of $R_l$-term Kronecker product sum 
\[
 {\bf A}_{R_l} ={\sum}_{k=1}^{R_l} A^{(1)}_k \otimes \cdots \otimes A^{(d)}_k,
\]
where each "univariate" matrix $A^{(\ell)}_k\in \mathbb{R}^{n\times n}$, 
$\ell=1,\ldots ,d$, takes the symmetric Toeplitz form, generated by the first column vector 
${\bf p}_k^{(\ell)}$.  
The storage complexity of the resultant RS representation to the matrix ${\bf A}$
is estimated by $O(N + d R_l n)$. 
 Similar matrix decompositions can be derived for the RS-Tucker and RS-TT 
representations of ${\bf P}_R$.

Now we represent the coefficient vector ${\bf c}\in \mathbb{R}^N$ as the $d$-dimensional 
$n^{\otimes d}$ tensor, ${\bf c}\mapsto{\bf C}\in \mathbb{R}^{n^{\otimes d}}$.
Then the matrix vector multiplication ${\bf A} {\bf C}=({\bf A}_{R_s} + {\bf A}_{R_l}){\bf C}$ 
implemented in tensor format
can be accomplished in $O(cN + d R_l N \log n)$ operations, i.e., with the asymptotically optimal cost in
the number of sampling points $N$. The reason is that the matrix 
${\bf A}_{R_s}$ has the diagonal form, while the matrix-vector product
between Toeplitz matrices $A^{(\ell)}_k$ constituting the Kronecker factors ${\bf A}_{R_l}$,
and the corresponding $n$-columns (fibers) of
the tensor ${\bf C}$, can be implemented by 1D FFT in $O(n \log n)$ operations.
One can customary enhance this scheme by introducing the low-rank tensor structure in 
the target vector (tensor) ${\bf C}$.

Case (B). 
This construction can be generalized to
the situation when ${\cal X}$ is a subset of $\Omega_h$, i.e., $N< n^d$. In this case the 
complexity again scales linearly in $N$ if $N=O(n^d)$. 
In the situation when $N \ll n^d $ the matrix-vector operation applies to the vector
${\bf C}$ that vanishes beyond the small set ${\cal X}$. In this case the corresponding
block-diagonal sub-matrices in $A^{(\ell)}_k$ loose the Toeplitz form thus resulting
in the slight increase in the overall cost $O(N^{1+1/d})$.

In both cases (A) and (B) the presented new construction can be applied 
within any favorable preconditioned iteration for solving the linear system (\ref{eqn:RBF_Syst}).

\subsection{Interaction energy for charged many-particle system}\label{ssec:InterEnerg_Applic}

Consider the calculation of the interaction  energy (IE) for 
a charged multi-particle system. In the case of lattice-structured systems, 
the fast tensor-based computation scheme for IE was described in \cite{VeKhorTromsoe:15}.
Recall that the interaction energy of the total electrostatic potential generated by 
the system of $N$ charged particles located at ${x}_{k}\in \mathbb{R}^3$ ($k=1,...,N$) is defined by the 
weighted sum 
\begin{equation}\label{eqn:EnergyLatSum}
E_N =E_N(x_1,\dots,x_N)= \frac{1}{2} \sum\limits_{{j}=1}^N z_{j}
\sum\limits_{{k}=1, {k}\neq {j}}^N \frac{z_{k} }{\|{x}_{j} - {x}_{k}\|}, 
\end{equation}
where $z_{k}$ denotes the particle charge.
Letting $\sigma >0$ be the minimal physical distance between the centers of particles, 
we arrive at the $\sigma$-separable systems in the sense of 
Definition \ref{def:QU_Distrib}.
The double sum in (\ref{eqn:EnergyLatSum}) applies
only to the particle positions $\|{x}_{j} - {x}_{k}\|\geq \sigma $, hence, the quantity in 
(\ref{eqn:EnergyLatSum}) is computable also for singular kernels like $p(r)=1/r$.

We observe that the quantity of interest $E_N$ can be recast in terms 
of the interconnection matrix $A_{p,{\cal X}}$ defined by (\ref{eqn:RBF_Syst}) with $p(r)=1/r$, 
${\cal X}=\{x_1,\ldots,x_N\}$,
\begin{equation}\label{eqn:EnergySumMatr}
 E_N = \frac{1}{2} \langle (A_{p,{\cal X}} -\mbox{diag}A_{p,{\cal X}}) {\bf z}, {\bf z}   \rangle, 
 \quad \mbox{where}\quad {\bf z}=(z_1,\ldots,z_N)^T.
\end{equation}
Hence, $E_N$ can be calculated by using by using the approach briefly addressed in the previous section.

Here, we describe this scheme in the more detail. Recall that the reference canonical tensor
${\bf P}_R$ approximating the single Newton kernel on an $n\times n\times n $ tensor grid $\Omega_h$
in the computational box $\Omega=[-b,b]^3$ is represented by (\ref{eqn:sinc_general}), where $h>0$ 
is the fine mesh size.
For ease of exposition, we assume that the particle centers ${x}_{k}$
are located exactly at some grid points in $\Omega_h$
(otherwise, an additional approximation error may be introduced) such that
each point ${x}_{k}$ inherits some multi-index ${\bf i}_{k}\in {\cal I}$, and the origin
$x=0$ corresponds to  the central point on the grid, ${\bf n}_0=(n/2,n/2,n/2)$.
In turn, the canonical tensor ${\bf P}_0$ approximating the total interaction potential 
$P_N(x)$ ($x\in \Omega$) for the $N$-particle system,
\[
 P_N(x)=\sum\limits_{{k}=1}^N \frac{z_{k} }{\|{x} - {x}_{k}\|}\, \leadsto \,
 {\bf P}_0 = {\bf P}_s + {\bf P}_l\in \mathbb{R}^{n\times n\times n},
\]
is represented by (\ref{eqn:Total_Sum}) as a sum of short- and long-range tensor
components.
Now the tensor ${\bf P}_0={\bf P}_0(x^h)$ is defined at each point $x^h\in \Omega_h$, 
and, in particular, 
in the vicinity of each particle center $x_k$, i.e. at the grid-points $x_k + h{\bf e}$, where
the directional vector ${\bf e}=(e_1,e_2,e_3)^T$ is specified by some choice of coordinates
$e_\ell \in\{-1,0,1\}$ for $\ell=1,2,3$. This allows to introduce the useful notations 
${\bf P}_0(x_k + h{\bf e})$ which can be applied to all tensors leaving on $\Omega_h$.
Such notations simplify the definitions of entities like energy, gradients, forces, etc. 
applied to the RS tensors.

The following lemma describes the tensor scheme for calculating $E_N$ by 
utilizing the long-range part ${\bf P}_l$ only in the tensor representation of $P_N(x)$.
\begin{lemma}\label{lem:InterEnergy}
 Let the effective support of the short-range components in the reference
 potential ${\bf P}_R$ do not exceed $\sigma>0$.
 Then the interaction energy $E_N$ of the $N$-particle system  can be
 calculated by using only the long range part  in the total potential sum
 \begin{equation}\label{eqn:EnergyFree_Tensor}
 E_N =E_N(x_1,\dots,x_N)= \frac{1}{2} 
 \sum\limits_{{j}=1}^N z_{j}({\bf P}_l({x}_{j}) - z_j {\bf P}_{R_l}(x=0)),
\end{equation}
in $O(d R_l N)$ operations, where $R_l$ is the canonical rank of the long-range component.
\end{lemma}
\begin{proof} 
Similar to \cite{VeKhorTromsoe:15}, where the case of lattice structured systems was analyzed,
we show that the interior sum in (\ref{eqn:EnergyLatSum})
can be obtained from  the tensor ${\bf P}_0$
traced onto the centers of particles ${x}_{k}$, 
where the term corresponding to ${x}_{j}={x}_{k}$ is removed, 
\[
 \sum\limits_{{k}=1, {k}\neq {j}}^N \frac{z_{k} }{\|{x}_{j} - {x}_{k}\|}
  \, \leadsto \, {\bf P}_0({x}_{j}) - z_{j}{\bf P}_R(x=0).
\]
Here the value of the reference canonical tensor
${\bf P}_R$, see (\ref{eqn:sinc_general}), is evaluated at the origin $x=0$, i.e., corresponding 
to the multi-index ${\bf n}_0=(n/2,n/2,n/2)$. Hence, we arrive at the tensor approximation
\begin{equation}\label{eqn:EnergyFree_TensF}
 E_N \, \leadsto \, \frac{1}{2} 
 \sum\limits_{{j}=1}^N z_{j}({\bf P}_0({x}_{j}) - z_j {\bf P}_R(x=0)).
\end{equation}
Now we split ${\bf P}_0$ into the long-range part (\ref{eqn:Long-Range_Sum})
and the remaining short-range potential, to obtain
${\bf P}_0({x}_{j})= {\bf P}_s({x}_{j}) + {\bf P}_l({x}_{j})$, and the same for the reference
tensor ${\bf P}_R$.
By assumption, the short-range part ${\bf P}_s({x}_{j})$ at point ${x}_{j}$ 
in (\ref{eqn:EnergyFree_TensF}) consists only of the local term $P_{R_s}(x=0)=z_j {\bf P}_R(x=0)$.
Due to the corresponding cancellations in the right-hand side of
(\ref{eqn:EnergyFree_TensF}), we find that $E_N$ depends only on ${\bf P}_l$,
leading to the final tensor representation in (\ref{eqn:EnergyFree_Tensor}).

We arrive at the linear complexity scaling $O(d R_l N)$ taking into account the $O(d R_l )$
cost of the point evaluation for the canonical tensor ${\bf P}_l$.
\end{proof}

\begin{table}[tbh]
\begin{center}\footnotesize
\begin{tabular}
[c]{|r|r|r|r|r|r| }%
\hline
 3D Grid, mesh size        &  $N$   &  $100$      & $200$       & $400$       & $782$ \\
 \cline{2-6}    &  Exact $E_N$         & $-8.4888$  & $-18.1712$  & $-35.9625$  & $-90.2027$ \\
\hline
 $4096^3$,  $1.37\cdot 10^{-2}$  & $E_N-E_{N,T}$   & $0.01$ & $2\cdot 10^{-5}$ & $0.0034$ &  $0.0084$  \\
\cline{2-6}& $(E_N-E_{N,T})/E_N$ & $0.0012$ & $10^{-5}     $ & $10^{-4}$ &  $3\cdot 10^{-4}$  \\
\hline
\hline
$8192^3$,  $6.8\cdot 10^{-3}$ & $E_N-E_{N,T}$  & $0.0028$     & $0.005$             & $0.0074$ &  $0.0245$  \\
\cline{2-6}& $(E_N-E_{N,T})/E_N$ & $3\cdot 10^{-4}$ & $3\cdot 10^{-4}$ & $2\cdot10^{-4}$ &  $3\cdot 10^{-4}$  \\
  \hline
  \hline
$16384^3$,  $3.4\cdot 10^{-3}$ & $E_N-E_{N,T}$  & $0.0021$     & $0.0013$             & $0.0039$ &  $0.0053$  \\
\cline{2-6}& $(E_N-E_{N,T})/E_N$ & $2\cdot 10^{-4}$ & $  10^{-4}$ & $ 10^{-4}$ &  $  10^{-4}$  \\
  \hline
 \end{tabular}
\caption{\emph{\small Absolute and relative errors in the interaction energy of
 $N$-particle clusters  computed by full canonical tensor approximation ($R_s=0$)}. 
}
\label{Tab:EN_NU_sr0}
\end{center}
\end{table}
Table \ref{Tab:EN_NU_sr0} shows the error of energy computation by (\ref{eqn:EnergyFree_TensF})
using the tensor summation with full rank canonical tensors. We use the data for protein type 
molecular system provided by the authors of \cite{KwHeFeStBe:16}.
This table indicates that the relative error
of tensor computation remains of the order of $10^{-4}$ for the considered range of grid-size and
the number molecular clusters. The canonical ranks $R$ of the reference kernel are given by $34$, $37$, and $39$
for $n\times n \times n$ grids with $n$ equal to $4096$, $8192$ and $16384$, respectively. 
CPU times for energy computation with these sizes of molecular clusters are small for 
both tensor and the straightforward calculation schemes. Tensor computation on the grid of size
$n^3=8192^3$ takes $8\cdot 10^{-4}$s, while 
using the original  formula (\ref{eqn:EnergyLatSum}) it amounts to $0.04$s.

Table \ref{Tab:EN_NU_sr13} presents the error of energy computation by (\ref{eqn:EnergyFree_TensF})
by using the RS tensor format with $R_l=14$ and $R_s=13$.
Remarkably, that the approximation error does not exceed the errors in Table \ref{Tab:EN_NU_sr0}.

\begin{table}[tbh]
\begin{center}\footnotesize
\begin{tabular}
[c]{|r|r|r|r|r|r| }%
\hline
 3D Grid, mesh size        &  $N$  & $100$      & $200$       & $400$       & $782$ \\
 \cline{2-6}    &  Exact $E_N$         & $-8.4888$  & $-18.1712$  & $-35.9625$  & $-90.2027$ \\
\hline
 $4096^3$,  $1.37\cdot 10^{-2}$  & $E_N-E_{N,T}$   & $0.0044$ & $0.0191$ & $0.0265$ &  $0.1254$  \\
\cline{2-6}& $(E_N-E_{N,T})/E_N$ & $0.0005$ & $0.0002$ & $0.0007$ &  $0.0014$  \\
\hline
\hline
$8192^3$,  $6.8\cdot 10^{-3}$ & $E_N-E_{N,T}$  & $0.0010$ & $0.0044$  & $0.0074$ &  $0.0064$  \\
\cline{2-6}& $(E_N-E_{N,T})/E_N$ & $ 10^{-4}$ & $2\cdot 10^{-4}$ & $2\cdot10^{-4}$ &  $ 10^{-4}$  \\
  \hline
  \hline
$16384^3$,  $3.4\cdot 10^{-3}$ & $E_N-E_{N,T}$  & $0.0015$ & $0.0010$   & $0.002$ &  $0.0001$  \\
\cline{2-6}& $(E_N-E_{N,T})/E_N$ & $2\cdot 10^{-4}$ & $ 10^{-4}$ & $ 10^{-4}$ &  $10^{-5}  $  \\
  \hline
 \end{tabular}
\caption{\emph{\small Absolute and relative errors in the interaction energy of
 $N$-particle clusters computed by RS-tensor approximation with $R_l=14$,  ($R_s=13$)}.
}
\label{Tab:EN_NU_sr13}
\end{center}
\end{table}

Table \ref{Tab:Error_E_Nucl} represents the approximation error in $E_N$ computed
by RS tensor representation (\ref{eqn:EnergyFree_Tensor}) for the different values of system size.
Grid size is $n^3=4096^3$, $h=0.0137$, canonical rank for the reference tensor is $R=29$.
The short range part of the RS tensor is taken as $R_s=10$. 
\begin{table}[tbh]
\begin{center}\footnotesize
\begin{tabular}
[c]{|r|r|r|r|r|r|r|r|}%
\hline
$N$             & $200$     & $300$     & $400$     & $500$  & $600$ & $700$ \\
\hline
Exact $E_N$   & $-17.91$ & $-26.47$ & $-35.56$  & $-47.1009$  & $-62.32$ & $-77.47$ \\
\hline \hline
$E_N-E_{N,T}$   & $0.0018$ & $0.0004$ & $0.0026$  & $0.0083$  & $0.019$ & $0.017$ \\
\hline  
$(E_N-E_{N,T})/E_N$  & $6\cdot 10^{-5}$ & $9\cdot 10^{-7}$ & $3.8\cdot 10^{-5}$  & $2.4\cdot 10^{-4}$ &
$3.0\cdot 10^{-4}$ & $2.0\cdot 10^{-4}$ \\
\hline 
 \end{tabular}
\caption{\emph{\small Error in the interaction energy of
clusters of $N$ particles computed by the RS tensor approach ($R_s=10$).} }
\label{Tab:Error_E_Nucl}
\end{center}
\end{table}

Table \ref{Tab:EN_NU_LatRand} shows the results for several clusters of particles generated by
random assignment of charges $z_j$ to finite lattices of size $8^3$, $12^3$, $16\times 16 \times 8$ and $16^3$.
Newton kernel is approximated with $\varepsilon_N=10^{-4}$ on the grid of size $n^3=4096^3$, $h=0.0137$, 
with the rank $R=25$. Computation of the
interaction energy was performed using the only long-range part with $R_l=12$. 
For the rank reduction the 
multigrid C2T algorithm is used \cite{khor-ml-2009}, with the rank truncation
parameters $\varepsilon_{C2T}=10^{-5}$, $\varepsilon_{T2C}=10^{-6}$.
The box size is about $40\times 40\times 40$ atomic units, with the mesh size $h=0.0098$. 
\begin{table}[tbh]
\begin{center}\footnotesize
\begin{tabular}
[c]{|r|r|r|r|r| }%
\hline
$N$ of particles  &  $512$      & $1728$       & $2048$       & $4096$ \\
\hline   
\hline  
 Exact $E_N$     & $51.8439$  & $-133.9060$  & $-138.5562$  & $-207.8477$ \\
\hline
 $E_N-E_{N,T}$  & $0.1145$ & $0.1317$ & $0.2263$ &  $0.2174$  \\
 \hline
 $(E_N-E_{N,T})/E_N$  & $0.0022$ & $0.001$   & $0.0016$ &  $0.001$  \\
\hline
\hline  
 Ranks full can. & 12800 & 43200   & 51200 &  102400  \\
  \hline
  Tucker ranks  & 31, 29, 30 & 43,42,43   & 51, 51, 33 & 53,54,54  \\
  \hline
  Reduced RS rank & 688 & 1248  & 1256 & 1740  \\
  \hline
  \hline  
 Time stand. Sum. & 0.011 & 0.12  & 0.18 & 0.79  \\
  \hline
 Time  Tens Sum. & $2\cdot 10^{-5}$ & $6\cdot 10^{-5}$  & $7\cdot 10^{-5}$ & $1.5\cdot 10^{-4}$  \\
  \hline
 \end{tabular}
\caption{\emph{\small Errors in the interaction energy of $N$-particle
clusters computed by RS tensor approximation with 
the long-range rank parameter $R_l=12$ ($R_s=13$)}.
}
\label{Tab:EN_NU_LatRand}
\end{center}
\end{table}

Table \ref{Tab:EN_NU_LatRand} illustrates that
the relative accuracy of energy calculations by using the RS tensor format remains of the order
of $10^{-3}$ almost independent of the cluster size. Tucker ranks only slightly increase 
with the system size $N$.
The computation time for the tensor ${\bf P}_l$ remains almost constant, 
while the point evaluations time for this tensor (with pre-computed data) 
increases linearly in $N$, see Lemma \ref{lem:InterEnergy}.

\subsection{Gradients and forces}\label{ssec:Forces_Applic}

 Calculation of electrostatic forces and gradients of the interaction potentials 
 in multiparticle systems is a computationally extensive problem.
The algorithms based on Ewald summation technique were discussed in \cite{DesHolmII:98,HoEa:88}.
We propose an alternative approach using the RS tensor format. 

First, we consider computation of gradients.
Given an RS-canonical tensor ${\bf A}$ as in (\ref{eqn:RS_Can}) with the width parameter $\gamma >0$, 
the discrete gradient $\nabla_h=(\nabla_1,\ldots,\nabla_d)^T$ applied to the long-range
part in ${\bf A}$ at all grid points of $\Omega_h$ simultaneously,  
can be calculated as the $R$-term canonical tensor by applying the simple
one-dimensional finite-difference (FD) operations to the long-range 
part in ${\bf A}={\bf A}_s + {\bf A}_l$,
\begin{equation}\label{eqn:Gradients_Tens}
\nabla_h {\bf A}_l= {\sum}_{k =1}^{R} \xi_k ({\bf G}_k^{(1)},\ldots,{\bf G}_k^{(d)})^T,
\end{equation}
with tensor entries
\[
 {\bf G}_k^{(\ell)} ={\bf u}_k^{(1)}  \otimes \cdots\otimes 
 \nabla_\ell {\bf u}_k^{(\ell)} \otimes \cdots\otimes{\bf u}_k^{(d)},
\]
where $\nabla_\ell$ ($\ell=1,\dots,d$) is the univariate FD differentiation scheme
(by using backward or central differences). Numerical complexity of
the representation (\ref{eqn:Gradients_Tens}) can be estimated by $O(d R n )$ provided that
the canonical rank is almost uniformly bounded in the number of particles. 
The gradient operator applies locally to each short-range term in (\ref{eqn:RS_Can})
which amounts in the complexity $O(d R_0 \gamma N)$.

The gradient of an RS-Tucker tensor, for example, for evaluation of the field 
\[
 {\bf F}(x) =-\nabla P(x)=  \sum\limits_{{k}=1, {k}\neq {j}}^N  z_{k} 
 \frac{{x} - {x}_{k} }{\|{x} - {x}_{k}\|^3},
\]
can be calculated in a similar way. 
Furthermore, in the setting of \S\ref{ssec:InterEnerg_Applic}, 
the force vector ${\bf F}_j$ on the particle $j$ is obtained by differentiating the 
electrostatic potential energy $E_N(x_1,\dots,x_N)$ with respect to $x_j$,
\[
 {\bf F}_j=-\frac{\partial}{\partial x_j} E_N = - \nabla_{| x_j} E_N,
\]
which can be calculated explicitly (see \cite{HoEa:88}) in the form,
\[
 {\bf F}_j = \frac{1}{2} z_{j} \sum\limits_{{k}=1, {k}\neq {j}}^N  z_{k} 
 \frac{{x}_{j} - {x}_{k} }{\|{x}_{j} - {x}_{k}\|^3}.
\]
The Ewald summation technique for force calculations 
was presented in \cite{DesHolmII:98,HoEa:88}.
In principle, it is possible to construct the RS tensor representation for this vector 
field directly by using the radial basis function $p(r)=1/r^2$.

Here we describe the alternative approach based on 
numerical differentiation of the energy functional by using RS tensor representation 
of the $N$-particle interaction potential on fine spacial grid.
The differentiation in RS-tensor format with respect to $x_j$
is based on the explicit representation (\ref{eqn:EnergyFree_Tensor}),
which can be rewritten in the form
 \begin{equation}\label{eqn:EnergyFree_TensForce}
 E_N(x_1,\dots,x_N)= \widehat{E}_N(x_1,\dots,x_N)
 -  \frac{1}{2} (\sum\limits_{{j}=1}^N z_{j}^2) {\bf P}_{R_l}(x=0),
\end{equation}
where $\widehat{E}_N(x_1,\dots,x_N)= \frac{1}{2} \sum\limits_{{j}=1}^N z_{j}{\bf P}_l({x}_{j}) $
denotes the "non-calibrated" interaction energy with the long-range tensor component ${\bf P}_l$.
In the following discussion, for definiteness, we set $j=N$. 
Since the second term in (\ref{eqn:EnergyFree_TensForce}) does not depend on the particle positions 
it can be omitted in calculation of variations in $E_N$  with respect to $x_N$. 
Hence we arrive at the representation for 
the first difference in direction ${\bf e}_i$, $i=1,2,3$,
\[
 E_N(x_1,\dots,x_N) - E_N(x_1,\dots,x_N-h{\bf e}_i)=  
 \widehat{E}_N(x_1,\dots,x_N) - \widehat{E}_N(x_1,\dots,x_N-h{\bf e}_i).
\]
The straightforward implementation of the above relation for three different values of
${\bf e}_1=(1,0,0)^T$, ${\bf e}_2=(0,1,0)^T$ and ${\bf e}_3=(0,0,1)^T$
is reduced to four calls of the basic
procedure for computation of the tensor ${\bf P}_l$ corresponding to \emph{four} different dispositions 
of points $x_1,...,x_N$ leading to the cost $O(d R n )$.

However, the factor \emph{four} can be reduced to merely one taking into account that 
the two canonical/Tucker tensors ${\bf P}_l$ computed for particle 
positions $(x_1,\dots,x_{N-1},x_N)$ and $(x_1,\dots,x_{N-1},x_N-h{\bf e})$ differ 
in a small part (since positions $x_1,\dots,x_{N-1}$ remain fixed). 
This requires only minor modifications compared with the repeating the
full calculation of $\widehat{E}_N(x_1,\dots,x_N)$.

\subsection{Regularization scheme for the Poisson-Boltzmann equation}\label{ssec:PBE_Applic}

We describe the application scheme to the Poisson-Boltzmann equation (PBE)
commonly used for numerical modeling of the electrostatic potential of proteins.
The traditional numerical approaches to PBE are based on either multigrid  
\cite{Holst:2008} or domain decomposition \cite{CaMaSt:13} methods.

Consider a solvated biomolecular system modeled by dielectrically separated  domains
with singular Coulomb potentials distributed in the molecular region.
For schematic representation, we consider the system occupying a rectangular domain $\Omega$
with boundary $\partial \Omega$, see Fig. \ref{fig:Protein}. The solute (molecule) region 
is represented by $\Omega_m$ and the solvent region by $\Omega_s$.
\begin{figure}[htb]
\centering
\includegraphics[width=6.5cm]{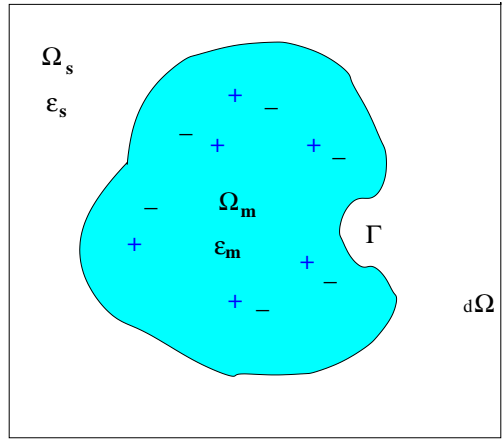} 
\caption{Computational domain for PBE. }
\label{fig:Protein}
\end{figure}

The linearized Poisson-Boltzmann equation takes a form, see \cite{Holst:2008},
\begin{equation}\label{eqn:PBE}
 -\nabla\cdot(\epsilon \nabla u)+ \kappa^2 u=\rho_f\quad \mbox{in } \quad \Omega,
\end{equation}
where $u$ denotes the target electrostatic potential of a protein, and 
$\rho_f= \sum\limits_{{k}=1}^N z_k \delta(\|{x} - {x}_{k}\|)$ 
is the scaled singular charge distribution supported at points $x_k$ in $\Omega_m$, where
$\delta$ is the Dirac delta. 
Here $\epsilon=1$ and $\kappa=0$ in $\Omega_m$, while in the solvent region $\Omega_s$ 
we have $\kappa \geq 0$ and $\epsilon\leq 1$.
The boundary conditions on the external boundary
$\partial\Omega$ can be specified depending on the particular problem setting.
For definiteness, we impose the simplest Dirichlet boundary condition $u_{| \partial \Omega}=0$.
The interface conditions on the interior boundary 
$\Gamma=\partial \Omega_m$ arise from the dielectric theory:
\begin{equation}\label{eqn:Intcond_PBE}
 [u]=0, \quad \left[ \epsilon\frac{\partial u}{\partial n}  \right]\quad \mbox{on}\quad \Gamma.
\end{equation}
The practically useful solution methods for the PBE are based on regularization schemes aiming at removing
the singular component from the potentials in the governing equation.
Among others, we consider one of the most commonly used approaches based on the additive splitting 
of the potential only in the molecular region $\Omega_m$, see \cite{Holst:2008}.
To that end we introduce the additive splitting
\[
 u=u^r +u^s, \quad \mbox{where}\quad u^s=0 \quad \mbox{in}\quad \Omega_s,
\]
and where the singular component satisfies the equation
\begin{equation}\label{eqn:us_PBE}
-\epsilon_m \Delta u^s = \rho_f\quad \mbox{in}\quad  {\Omega}_m; 
\quad u^s=0 \quad \mbox{on}\quad {\Gamma}.
\end{equation}
Now equation (\ref{eqn:PBE}) can be transformed to that for the regular potential $u^r$:
\begin{equation}\label{eqn:Regul_PBE}
 -\nabla\cdot(\epsilon \nabla u^r)+ \kappa^2 u^r=\rho_f\quad \mbox{in } \quad \Omega,
\end{equation}
\[
 [u^r]=0,\quad 
 \left[ \epsilon\frac{\partial u^r}{\partial n}\right]=
 -\epsilon_m \frac{\partial u^s}{\partial n}, 
 \quad \mbox{on}\quad \Gamma.
\]
To facilitate the solution of equation (\ref{eqn:us_PBE}) with singular data 
we define the singular potential $U$ in the free space by 
\[
\epsilon_m \Delta U = \rho_f \quad  \mbox{in}\quad \mathbb{R}^3,
\]
and introduce its restriction $U^s$ onto $\Omega_m$,
\begin{equation*}\label{eqn:singularPot_PBE}
 U^s=U_{|\overline{\Omega}_m}\quad \mbox{in}\quad  \overline{\Omega}_m; 
 \quad U^s=0\quad \mbox{in}\quad  {\Omega}_s.
\end{equation*}
Then we have $u^s=U^s + u^h$, where a harmonic function $u^h$ compensates the discontinuity 
of $U^s$ on $\Gamma$,
\begin{equation*}\label{eqn:Harmonic}
 \Delta u^h = 0 \quad \mbox{in}\quad  {\Omega}_m; \quad u^h=-U^s \quad\mbox{on}\quad \Gamma.
 \end{equation*}

The advantage of this formulation is due to (a) the absence of singularities in the
solution $u^r$, and (b) the localization of the solution splitting only on the domain
$\Omega_m$. Calculating the singular potential $U$ which may include a sum of hundreds 
or even thousands of
single Newton kernels in 3D leads to a challenging computational problem.
In our approach it can be represented on large tensor grids with controlled 
precision by using the range separated tensor formats described above. 
The long-range component in the formatted parametrization
remains smooth and allows global low-rank representation. 
 The approach can be combined with the reduced model approach for PBE with 
the parametric coefficients \cite{KwHeFeStBe:16}.

It is worth noting that the short-range part in the tensor representation of $U$ 
does not contribute to the right-hand side in the 
interface conditions on $\Gamma$ in equation (\ref{eqn:Regul_PBE}).
This crucial simplification is possible since the physical distance between the atomic centers 
in protein modeling is bounded from below by the fixed constant $\sigma >0$,
while the effective support of the 
localized parts in the tensor representation of $U$ can be chosen as the half of $\sigma$.
Moreover, all normal derivatives can be easily
calculated by differentiation of univariate canonical vectors in the long-range part of the 
electrostatic potential $U$ precomputed a on fine tensor grid in $\mathbb{R}^3$ 
(see \S\ref{ssec:Forces_Applic}).
Hence, the numerical cost to build up the interface conditions in (\ref{eqn:Regul_PBE})
becomes negligible compared with the solution of the equation (\ref{eqn:Regul_PBE}).
We conclude with the following 
\begin{proposition}\label{pro:PBE_RSTensor}
Let the effective support of the short-range components in the reference
 potential ${\bf P}_R$ be chosen  not larger than $\sigma/2$. Then 
 the interface conditions in the regularized formulation (\ref{eqn:Regul_PBE}) of the PBE 
 depend only on the low-rank long-range component in the free-space electrostatic potential
 of the system. The numerical cost to  build up the interface conditions 
 on $\Gamma$ in (\ref{eqn:Regul_PBE})  does not depend on the number of particles $N$.
\end{proposition}

Finally, we notice the important characterization of the protein molecule 
given by the electrostatic solvation energy \cite{Holst:2008}, which is the difference 
between the electrostatic free energy in the solvated state (described by the PBE) 
and the electrostatic free energy in the absence of solvent, i.e. $E_N$. 
Now the electrostatic solvation energy can be computed in the framework 
of the new regularized formulation (\ref{eqn:Regul_PBE}) of PBE.

\section{Conclusions}\label{sec:Conclusions}

In this paper, we introduce and analyze the new range-separated canonical and Tucker tensor formats 
for the grid representation of the long-range interaction potentials in multiparticle systems.
One can distinguish the RS tensors from the conventional rank-structured representations
due to their intrinsic features, originating from tensor approximation to 
multivariate functions with multiple singularities, in particular, generated by 
a weighted sum of the classical Green's kernels.

We show that the tensor approximation to the particle  interaction potentials
allows  to split their long- and short-range parts providing their efficient  
representation and numerical treatment in the low-rank RS tensor formats.
Indeed, the long-range part in the potential sum can be represented on a grid
by the low-rank canonical/Tucker tensor globally in the computational box, 
while its short-range component is parametrized by a reference tensor of local support
and a list of particle coordinates and charges.
In particular, we prove that the Tucker rank of the long-range part in 
$N$-particle potential depends only logarithmically on the number 
of particles in the system.

The RS formats prove  to be well suited for summation of the electrostatic
potentials in large many-particle systems in a box (e.g. proteins or large molecular clusters),  
providing the low-parametric tensor representation
of the total potential at any point of the fine 3D $n \times n \times n$ Cartesian grid.
For the computer realization of the RS tensor decomposition, 
a canonical-to-Tucker rank reduction algorithm is applied resulting in 
the $O(n \log N)$ grid representation of the long-range part in the many-particle potential.
Notice that the existing approaches are limited by the $O(n^3)$
complexity contrary to the almost linear scaling in univariate mesh size $n$ for 
the RS tensor format.

Numerical tests confirm the theoretical rank estimates and the asymptotically optimal 
complexity bound $O(N)$.  
In particular, the  electrostatic  potential for $N$-particle systems 
(up to several thousands of atoms) is computed in Matlab with controllable accuracy, 
resulting in the RS tensor living on large 3D grids of size up to $n^3=10^{12}$.

As examples of possible applications, we describe 
the tensor based representation to the electrostatic free energy 
of a protein in the absence of solvent and illustrate the efficiency by numerical tests.
We observe that for moderate accuracy requirements,  
the application of the RS canonical/Tucker formats exhibits very mild limitations 
on the system size. This situation may occur in 
the problems of protein docking and classification of biomolecules.
Furthermore, we demonstrate how the RS tensor decomposition allows to easily compute
the gradients and forces for multi-particle interaction potentials.
The benefits of the RS tensors in multi-dimensional scattered data
modeling are also discussed.
Finally, we propose the enhanced regularized formulation 
for the Poisson-Boltzmann equation that is based on pre-computing of only the long-range part in 
the electrostatic potential of protein in free space $\mathbb{R}^3$.

The presented analysis of the RS tensor formats indicates their potential
benefits  in various applications related to
modeling of many particle systems, and addresses a number of new interesting
theoretical and algorithmic questions on rank-structured tensor approximation 
of multivariate functions with generally located point singularities.

Finally, we notice that the RS tensor approach can be easily extended to the case of 
multi-dimensional scattered data in $\mathbb{R}^d$ for $d>3$.

\section{Appendix: CP-to-Tucker tensor transform by reduced HOSVD}\label{app:C2T_HOSVD} 

For the reader's convenience, in this section we recall the main ingredients of the rank-reduction 
approach for canonical tensors with large initial rank.
The multigrid accelerated canonical-to-Tucker tensor transform combined with the 
Tucker-to-canonical scheme was invented in \cite{khor-ml-2009}  
for the rank optimization of 3D function related canonical tensors given as the 
large sum of rank-$1$ components.
It was proved to be a useful approach to many  problems arising 
in grid-based computations in quantum chemistry.

The approximation results in \cite{HaKhtens:04I,khor-rstruct-2006} indicate that 
the tensor representation of regular multidimensional functions with point singularities
can lead to their accurate low-rank decomposition. 
However, to compute the Tucker (or canonical) decomposition in a traditional way
as described in \cite{DMV-SIAM2:00}, requires information on all entries of a tensor. 
The so-called higher-order SVD (HOSVD) method \cite{DMV-SIAM2:00} amounts to $O(n^{d+1})$ 
operations and memory cost.

The numerical complexity can be reduced dramatically if the initial tensor is given 
in the canonical format, presumably with rather large rank parameter $R$. 
In general, it might be storage/time consuming to generate, first, a full tensor from 
the canonical one, and then apply the HOSVD based Tucker decomposition to it. 
The canonical-to-Tucker decomposition eliminates this step,
and turns the Tucker decomposition into the alternating least square (ALS)
problem with easily precomputed initial orthogonal Tucker subspaces.

Specifically, the canonical-to-Tucker decomposition eliminates finding the initial 
guess for the standard Tucker ALS iteration (see \cite{DMV-SIAM2:00,KoldaB:07}) 
by using the SVD of the side matrices in the canonical tensor representation.
This approach, called the reduced HOSVD (RHOSVD), was introduced in \cite{KhKh:06,khor-ml-2009}.
The multigrid version of the RHOSVD allows to reduce the dominating cost in 3D case to
$O(R\,n)$ \cite{VeKh_Diss:10}.

Without loss of generality, we consider the case $d=3$.
To define the reduced rank-$\bf r$ HOSVD type
Tucker approximation to the tensor in (\ref{eqn:CP_form}),
we set $n_\ell=n$ and suppose for definiteness that $n\leq R$. Now the SVD of the
side-matrix $U^{(\ell)}$ is given by
\begin{equation}\label{eqn:SVD_SideMatr}
U^{(\ell)}={Z}^{(\ell)} D_\ell {V^{(\ell)}}^T=
\sum\limits_{k=1}^n \sigma_{\ell,k} {\bf z}_k^{(\ell)}\, {{\bf v}_k^{(\ell)}}^T,\quad
{\bf z}_k^{(\ell)}\in \mathbb{R}^{n},\; {\bf v}_k^{(\ell)} \in \mathbb{R}^{R},
\end{equation}
with the orthogonal matrices ${ Z}^{(\ell)}=[{\bf z}_1^{(\ell)},...,{\bf z}_n^{(\ell)}]$, and
${V}^{(\ell)}=[{\bf v}_1^{(\ell)},...,{\bf v}_n^{(\ell)}]$, $\ell=1,2,3$.
Given the rank parameter ${\bf r}=(r_1,r_2,r_3)$ with $r_1,r_2,r_3 <n $, we introduce 
the truncated SVD of the side-matrix
$$
U^{(\ell)}\mapsto U^{(\ell)}_0 =
\sum\limits_{k=1}^{r_\ell} \sigma_{\ell,k} {\bf z}_k^{(\ell)}\, {{\bf v}_k^{(\ell)}}^T=
{Z}_0^{(\ell)} D_{\ell,0} {V_0^{(\ell)}}^T, \quad \ell=1,2,3,
$$
where
$
D_{\ell,0}=\mbox{diag} \{\sigma_{\ell,1},\sigma_{\ell,2},...,\sigma_{\ell,r_\ell}\}
$
and ${Z}_0^{(\ell)}\in \mathbb{R}^{n\times r_\ell}$,
${V_0}^{(\ell)}\in \mathbb{R}^{R\times r_\ell} $, represent the orthogonal
factors being the respective sub-matrices in the SVD factors of ${U}^{(\ell)}$.
\begin{definition}\label{def:RHOSVD}
(\cite{khor-ml-2009}). 
The RHOSVD approximation of ${\bf U}$ in (\ref{eqn:CP_form_ContrpTuck}), further 
called ${\bf U}_{({\bf r})}^0$,
is defined as the rank-${\bf r}$ Tucker tensor obtained by the projection  of ${\bf U}$ 
onto the 
orthogonal matrices of the dominating singular vectors in 
$Z_0^{(\ell)}\in \mathbb{R}^{n\times r_\ell}$, ($\ell=1,2,3$),
\begin{equation}\label{eqn:CP_form_RHOSVD}
{\bf U} \mapsto {\bf U}_{({\bf r})}^0=\boldsymbol{\xi} \times_1 {U}^{(1)}_0\times_2 {U}^{(2)}_0  \times_d {U}^{(3)}_0. 
\end{equation}
\end{definition}

For the ease of presentation, we further
sketch the algorithm \emph{Canonical-to-Tucker} for the 3D tensor. This includes the following basic steps:

\emph{ Input data}: Side matrices $U^{(\ell)}=[{\bf u}_1^{(\ell)}\ldots {\bf u}_R^{(\ell)}] 
\in \mathbb{R}^{n_\ell \times R}$, $\ell=1, 2,3$,
composed of the vectors ${\bf u}_k^{(\ell)}\in \mathbb{R}^{n_\ell}$, $k=1,\ldots, R$, 
see (\ref{eqn:CP_form}); maximal Tucker-rank parameter ${\bf r}$; maximal  number of  
 the ALS iterations $m_{max}$ (usually a small number).

  {\bf (I)}  Compute the singular value decomposition (SVD) of the side matrices:
\[
 U^{(\ell)} = Z^{(\ell)} D^{(\ell)} V^{(\ell)}, \quad \ell=1, 2, 3.
\]
Discard the singular vectors in $ Z^{(\ell)}$ and the respective singular values 
up to given rank threshold, yielding the small orthogonal matrices 
$Z^{(\ell)}_{0} \in \mathbb{R}^{n_\ell \times r_\ell}$, and diagonal matrices
$D_{\ell,0}\in \mathbb{R}^{r_\ell \times r_\ell}$, $\ell=1, 2, 3$.
  
 {\bf (II)} Project the side matrices $U^{(\ell)}$ onto the orthogonal basis set 
 defined by $Z^{(\ell)}_{0}$
\begin{equation}\label{eq:alg1}
U^{(\ell)} \mapsto  \widetilde{U}^{(\ell)} = (Z^{(\ell)}_{0})^T U^{(\ell)}
= D_{\ell,0} {V_0^{(\ell)}}^T, \quad
 \widetilde{U}^{(\ell)}\in \mathbb{R}^{r_\ell \times R}, \quad \ell=1, 2,3.
\end{equation} 
and compute ${\bf U}_{({\bf r})}^0$ as in (\ref{eqn:CP_form_RHOSVD}).

  {\bf (III)} (Find dominating subspaces). Implement the following ALS iteration {\bf (IV)}
  $m_{max}$ times at most, starting from the RHOSVD initial guess ${\bf U}_{({\bf r})}^0$.
  

{\bf (IV)} Perform ALS iteration for  $\ell= 1, 2,3$:

$\diamond$  For $\ell=1$ : construct the partially projected image of the full tensor,
\begin{equation}\label{eq:alg2}
{\bf U} \mapsto
  \widetilde{\bf U}_1 = {\sum}_{k =1}^{R} c_k  {\bf u}_k^{(1)}  
\otimes \widetilde{\bf u}_k^{(2)} \otimes \widetilde{\bf u}_k^{(3)}, \quad  c_k \in \mathbb{R}.
\end{equation}
Here ${\bf u}_k^{(1)} \in \mathbb{R}^{n_1}$ is in physical space for mode $\ell=1$,
while  $\widetilde{\bf u}_k^{(2)} \in \mathbb{R}^{r_2}$ and $\widetilde{\bf u}_k^{(3)} \in \mathbb{R}^{r_3}$,
the column vectors of $\widetilde{U}^{(2)}$ and $\widetilde{U}^{(3)}$, respectively, leave 
the small coefficients index sets.

$\diamond$ Reshape the tensor $\widetilde{\bf U}_1\in \mathbb{R}^{n_1 \times r_2 \times r_3}$ into
a matrix $M_{U_1} \in \mathbb{R}^{n_1 \times (r_2  r_3)}$, representing the span of 
the optimized subset 
of mode-$1$ columns of the partially projected tensor $\widetilde{\bf U}_1$.
 Compute the SVD of the matrix $M_{U_1}$:
\[
 M_{U_1} = Z^{(1)} S^{(1)} V^{(1)},
\]
and truncate the set of singular vectors in 
$Z^{(1)} \mapsto \widetilde{Z}^{(1)}\in \mathbb{R}^{n_1 \times r_1}$,
according to the restriction on the mode-$1$ Tucker rank, $r_1$.

$\diamond$ Update the current approximation to the mode-$1$ dominating subspace, 
$Z^{(1)}_{r_1} \mapsto \widetilde{Z}^{(1)}$.

$\diamond$ Implement the single step of the ALS iteration for mode $\ell=2$ and $\ell=3$.

$\diamond$ End of the complete ALS iteration sweep.

$\diamond$ Repeat the complete ALS iteration $m_{max}$ times to obtain the optimized 
Tucker orthogonal side matrices $\widetilde{Z}^{(1)}$, $\widetilde{Z}^{(2)}$, $\widetilde{Z}^{(3)}$,
and final projected image $\widetilde{\bf U}_3$.
 
{\bf (V)}  Project the final iterated tensor $\widetilde{\bf U}_3$ in (\ref{eq:alg2}) 
using the resultant basis set in $\widetilde{Z}^{(3)}$ to obtain the core tensor, 
$\boldsymbol{\beta} \in \mathbb{R}^{r_1\times r_2 \times r_3} $. 
  
\emph{Output data}: The Tucker core tensor  $\boldsymbol{\beta}$ and the orthogonal
side matrices $\widetilde{Z}^{(\ell)}$, $\ell=1, 2,3$.

\vspace{0.1cm}
In such a way it is possible to obtain a Tucker decomposition of a canonical tensor with large
mode-size and with rather large ranks, as it may be the case for biomolecules or 
the electron densities in electronic structure calculations. 
The multigrid version of the Canonical-to-Tucker algorithm allows to avoid the expensive SVD 
calculation in Step {\bf (I)} thus reducing the cost of first two steps to $O(R\,n)$.

The \emph{Canonical-to-Tucker} algorithm can be easily modified to 
use an $\varepsilon$-truncation stopping criterion.
Notice that the 
maximal canonical rank\footnote{Further reduction of the canonical rank in the small-size core tensor 
$\boldsymbol{\beta}$ can be implemented by using the ALS-canonical iterative scheme, 
described e.g. in \cite{KoldaB:07}.} of the core 
tensor $\boldsymbol{\beta}$ does not exceed $r^2$ in the case of $r=r_\ell$,
see \cite{khor-ml-2009,VeKh_Diss:10}.

In our particular application we use the multigrid accelerated C2T algorithm, which 
eliminates the singular value decomposition of the side matrices $U^{(\ell)}$, $\ell=1,\, 2,\, 3$
of size $n\times R$ having the cost $O(n^2 R)$ for $n<R $ or $O(R^2 n)$ for $n>R $ 
and thus reduces the 
numerical costs to $O(R\,n)$ \cite{khor-ml-2009,VeKh_Diss:10}. 

\begin{footnotesize}

\bibliographystyle{abbrv}
\bibliography{BSE_Fock_Sums1.bib}
\end{footnotesize}

\end{document}